%% file: paper1a.tex
\documentclass[10pt,reqno]{amsart}
\usepackage[backend=biber,style=alphabetic,giveninits=true,maxnames=4,sorting=nyt]{biblatex}
\bibliography{OurReferences}
\input{preamble.tex}

\let\oldtikzcd\tikzcd
\def\tikzcd{{\ifnum0=`}\fi\oldtikzcd}

\let\oldendtikzcd\endtikzcd
\def\endtikzcd{\oldendtikzcd\ifnum0=`{\fi}}

\title
[Oplax {H}opf algebras]
{Oplax {H}opf algebras}

\author[Buckley]{Mitchell Buckley}
\address{Data61, CSIRO, Sydney NSW 1466, Australia}
\email{mitchell.buckley@data61.csiro.au}

\author[Fieremans]{Timmy Fieremans}
\address{Vakgroep Wiskunde, Faculteit Ingenieurswetenschappen, Vrije Universiteit Brussel, Pleinlaan 2, B-1050 Brussel, Belgium}
\email{tfierema@vub.ac.be} 

\author[Vasilakopoulou]{Christina Vasilakopoulou} 
\address{Departement of Mathematics, University of Patras, 26504 Greece}
\email{cvasilak@math.upatras.gr}

\author[Vercruysse]{Joost Vercruysse}
\address{D\'epartement de Math\'ematiques, Facult\'e des sciences, Universit\'e Libre de Bruxelles, Boulevard du
Triomphe, B-1050 Bruxelles, Belgium}
\email{jvercruy@ulb.ac.be}

\begin{document}

\begin{abstract}
We introduce the notion of an oplax Hopf monoid in any braided monoidal
bicategory, generalizing that of a Hopf monoid in a braided monoidal category in an appropriate way.
We show that Hopf $\Vv$-categories introduced in~\cite{BCV}
are a particular type of oplax Hopf monoids in the monoidal bicategory $\Span|\Vv$ described in~\cite{Gabipolyads}.
Finally, we introduce Frobenius $\Vv$-categories as the Frobenius objects in the same monoidal bicategory.
\end{abstract}

\maketitle

\setcounter{tocdepth}{1}
\tableofcontents

\subfile{1-introa}

\subfile{2-prelimi}

\subfile{3-oplax}

\subfile{4-X2Span}

\subfile{5-SpanV_updated}

\subfile{6-HopfVCats}

\subfile{7-frob}

\appendix

\subfile{8-appendix}

\subsection*{Acknowledgements}
{JV} would like to thank the {\em F\'ed\'eration Wallonie-Bruxelles} for supporting the ARC grant ``Hopf algebras and the symetries of non-commutative spaces'' and the {\em FNRS} for the MIS grant ``Antipode''.
{CV} would like to thank the {\em Universit\'e Libre de Bruxelles} for a postdoctoral position under the above ARC grant, as well as the General Secretariat for Research and Technology (GSRT) and the Hellenic Foundation for Research and Innovation (HFRI).
{MB} would like to thank the {\em Universit\'e Libre de Bruxelles} for support through a ULB Individual Fellowship.
All authors want to thank Gabriella B\"ohm, Nick Gurski and Steve Lack for inspiring discussions, as well as the anonymous reviewer who provided invaluable insight and comments, resulting in a solid framework without inconsistencies.
\printbibliography

\end{document}

%% file: preamble.tex
\usepackage{amssymb}
\usepackage{subfiles}
\usepackage{mathtools}
\usepackage[utf8]{inputenc}
\usepackage{float}
\usepackage{color}
\usepackage{tikz}
\usepackage[textsize=footnotesize]{todonotes}
\definecolor{darkgreen}{rgb}{0,0.45,0}
\usepackage{graphicx}
\usepackage{adjustbox}
\usepackage[all,2cell]{xy}\UseAllTwocells\SilentMatrices
\usepackage{bm}
\usepackage[colorlinks,citecolor=darkgreen,linkcolor=darkgreen,urlcolor=gray,final,hyperindex,linktoc=page,pdfpagelabels]{hyperref}
\usepackage[capitalize]{cleveref}
\crefname{equation}{}{}
\usepackage{enumerate}
\usepackage{multicol}

\AtEveryBibitem{\clearfield{doi}}
\AtEveryBibitem{\clearfield{url}}
\AtEveryBibitem{\clearfield{issn}}
\AtEveryBibitem{\clearfield{isbn}}
\DeclareSourcemap{
  \maps[datatype=bibtex]{
    \map[overwrite=true]{
      \step[fieldsource=fjournal]
      \step[fieldset=journal, origfieldval]
    }
  }
}

\usetikzlibrary{decorations.markings,arrows.meta,calc,fit,quotes,cd,math,external}

\tikzset{tick/.style={postaction={decorate,decoration={markings,
mark=at position 0.5 with {\draw[-] (0,.4ex) -- (0,-.4ex);}}}}}
\tikzset{tickx/.style={postaction={decorate,decoration={markings,mark=at position 0.5 with
{\fill circle [radius=.28ex];}}}}}
\newcommand{\tickar}
{\begin{tikzcd}[baseline=-0.5ex,cramped,sep=small,ampersand replacement=\&]{}\ar[r,tick]\&{}\end{tikzcd}}

\newcommand{\spn}[5]{\SelectTips{eu}{10}\xymatrix@C=.25in{#1 & #2\ar[l]|-{#4}\ar[r]|-{#5} & #3}}
\newcommand{\cspn}[5]{\SelectTips{eu}{10}\xymatrix@C=.25in{#1\ar[r]|-{#4} & #2 & #3\ar[l]|-{#5}}}

\newtheorem{proposition}{Proposition}[subsection]
\newtheorem{lemma}[proposition]{Lemma}
\newtheorem{corollary}[proposition]{Corollary}
\newtheorem{theorem}[proposition]{Theorem}

\theoremstyle{definition}
\newtheorem{definition}[proposition]{Definition}

\newtheorem{examples}[proposition]{Examples}

\theoremstyle{remark}
\newtheorem{remark}[proposition]{Remark}

\newcommand{\mlt}{m} 
\newcommand{\uni}{j} 
\newcommand{\comlt}{d} 
\newcommand{\couni}{\epsilon} 
\newcommand{\lcomlt}{\delta} 
\newcommand{\lcouni}{\varepsilon} 
\newcommand{\lmlt}{\mu}
\newcommand{\luni}{\eta}
\newcommand{\diag}{\Delta}

\newcommand{\Comon}{\sf Comon}

\newcommand{\id}{{\sf id}}
\newcommand{\ps}{\ca}
\newcommand{\PsMon}{\sf{PsMon}}

\newcommand{\Bim}{\sf{BMod}}
\newcommand{\PsComon}{\sf{PsComon}}

\newcommand{\Hom}{{\sf Hom}}

\newcommand{\Hopf}{{\sf Hopf}}
\newcommand{\sHopf}{{\sf sHopf}}
\newcommand{\Frob}{{\sf Frob}}
\newcommand{\Mod}{{\sf Mod}}
\newcommand{\LaxMod}{{\sf LaxMod}}
\newcommand{\OplMod}{{\sf OplMod}}

\newcommand{\Grph}{\sf Grph}

\newcommand{\Span}{{\sf Span}}
\newcommand{\Mat}{{\sf Mat}}
\newcommand{\Mon}{{\sf Mon}}
\newcommand{\Bimon}{{\sf Bimon}}
\newcommand{\VCat}{\Vv\textrm{-}\sf{Cat}}
\newcommand{\VopCat}{\Vv\textrm{-}\sf{opCat}}

\newcommand{\op}{\textrm{op}}

\newcommand{\OplBimon}{\sf OplBimon}

\def\ot{\otimes}

\newcommand{\ca}{\mathcal}
\newcommand{\caa}{\mathbb}

\newcommand{\Cc}{\mathcal{C}}

\newcommand{\Gg}{\mathcal{G}}

\newcommand{\Kk}{\mathcal{K}}

\newcommand{\Vv}{\mathcal{V}}

\def\text#1{{\rm {\rm #1}}}

\newcommand{\Cat}{{\sf Cat}}

\def\ul{\underline}

\newcommand{\Set}{\mathbf{Set}}

\def\lim{{\rm lim\,}}

\def\lax{{\sf lax}}
\def\opl{{\sf opl}}
\def\oplax{{\sf opl}}
\def\pseudo{{\sf ps}}

\def\joost{\color{blue}}
\definecolor{purple(x11)}{rgb}{0.5, 0.0, 0.5}

\def\Set{{\sf Set}}

\def\ot{\otimes}

\newcommand{\HopfCat}[1]{\Hopf\mhyphen #1 \mhyphen\Cat}
\newcommand{\sHopfCat}[1]{\sHopf\mhyphen #1 \mhyphen\Cat}
\newcommand{\FrobCat}[1]{\Frob\mhyphen #1 \mhyphen\Cat}

\newcommand{\ucorner}{\mathbin{\rotatebox[origin=c]{-135}{$\ulcorner$}}}
\newcommand{\upback}{\mathbin{\rotatebox[origin=c]{135}{$\ulcorner$}}}
\newcommand{\dpback}{\mathbin{\rotatebox[origin=c]{315}{$\ulcorner$}}}

\def\*C{{}^*\hspace*{-1pt}{\Cc}}

\def\text#1{{\rm {\rm #1}}}

\def\ul{\underline}

\newcommand{\dagarr}[3]{\ar@{<..}@<-2pt>[#1]_-{#3} \ar@<2pt>[#1]^-{#2}}

\mathchardef\mhyphen="2D 

\newcommand{\cd}[2][]{\vcenter{\hbox{\xymatrix#1@!C{#2}}}}
\newcommand{\Twocell}[4]{\arrow[Rightarrow,from=#1,to=#2,#4,shorten >=#3,shorten <=#3]}
\def\atpd{s}
\def\braid{\sigma}
\newcommand{\mapsfrom}{\mathrel{\reflectbox{\ensuremath{\mapsto}}}}

\newcommand{\coev}{{\rm coev}}
\newcommand{\Lan}{\mathrm{Lan}}
\newcommand{\ev}{{\rm ev}}


%% file: 1-introa.tex
\section{Introduction}

Over the last decades, a growing number of variations on the notion of Hopf algebra have surfaced, and 
this has led to increasing investigations on the intriguing question of how these various notions are related and can be unified in a single 
(categorical) framework.

{\em Weak Hopf algebras} \cite{BNS:whaI} were introduced to describe extensions of Von Neumann algebras as crossed products (by an action of a weak Hopf $C^*$-algebra). A weak Hopf algebra is an algebra that also has a coalgebra structure, but their compatibility as well as the antipode properties are weakened compared to those of usual Hopf algebras.
Weak Hopf algebras were one of the inspiring examples to define
{\em Hopf algebroids} \cite{BS:Hoid}, which in turn led to Hopf
monads \cite{BLV:Had}. Just as Hopf monads live in the monoidal
2-category of categories, functors and natural transformations,
Hopf-type objects can be defined in any monoidal bicategory, and
they have been studied in this way, see eg.
\cite{Monoidalbicatshopfalgebroids,CLS,Street2012,BohmLack}. In the latter, all previous mentioned Hopf-type objects (as well as a few more) 
are interpreted as particular opmonoidal monads in suitable monoidal bicategories, and different classical characterizations of Hopf algebras are 
recovered at an extremely high level of generality. 

In a different direction, {\em multiplier Hopf algebras} \cite{VD:multi} provide a generalization of Hopf algebras to the non-unital setting. Recently, the theory of weak and multiplier Hopf algebras have been merged \cite{VDW:wMHA,BGTLC} and brought to the categorical setting \cite{Bohm2017}.

An intermediate notion, called {\em Hopf category}, was introduced recently in \cite{BCV}, as a linearized version of groupoids.
Just as a category can be viewed as a `monoid with many objects',
Hopf $\Vv$-categories are a many-object generalization (or a {\em categorification}) of Hopf algebras. They are defined as categories enriched over comonoids (in a braided monoidal category $\Vv$), that admit a suitable notion of antipode. 
Moreover, a Hopf category can be `packed' by taking the coproduct of its hom-objects, leading to examples of weak (multiplier) Hopf algebras.
In \cite{Gabipolyads}, B\"ohm
showed that Hopf $\Vv$-categories fit in the framework of \cite{BohmLack}, namely they can be viewed as
a particular type of opmonoidal monad in a suitably constructed monoidal bicategory denoted $\Span|\Vv$, 
where a braided monoidal category $\Vv$ is viewed as a $1$-object monoidal bicategory.

For $H$ either a classical Hopf algebra, a weak Hopf algebra, a Hopf algebroid or a Hopf category,
there are two fundamental points for its bicategorical interpretation as described above: $H$ is always a $1$-cell in the considered bicategory, and
moreover its underlying monoid and comonoid structures 
make use of different monoidal products (that together fit into a {\em duoidal} category structure).

In this work, we aim to provide a next step in this process of settling the notion of Hopf-type objects in higher categories,
moving from monoidal bicategories i.e.\ a tricategory with a single $0$-cell, to braided monoidal bicategories \cite{CoherenceBrMonBicat} i.e. a 
tetracategory with a single $0$-cell and a single $1$-cell, see  \cite{ToddTrimbleNotes,montricat}. 
Within such a setting, we introduce the notion of an {\em oplax bimonoid} which consists of a $0$-cell endowed with the structure of a pseudomonoid
and pseudocomonoid, along with oplax compatibility constraints between them. 
In the spirit of Tannaka-Krein duality, we study the relation between monoidal structures on the category of lax pseudomodules over a pseudomonoid and oplax bimonoid structures on this pseudomonoid. 

Particular attention should be paid to the correct notion of an antipode for an oplax Hopf algebra, which is more elaborate than what one might expect 
at first sight. It is no longer a convolution inverse of the identity, but an \emph{oplax inverse}, a notion inspired by firm Morita contexts. 
However, we still retain the uniqueness of the antipode if it exists, and we discuss the relation between the existence of an antipode and the 
bijectivity of fusion morphism for an oplax bimonoid. 

A first motivating example for our new notion are groupoids which we show are oplax Hopf monoids in the bicategory of spans. More generally, Hopf categories are a well-described class of oplax Hopf monoids in a variation of B\"ohm monoidal bicategory $\Span|\Vv$, where now the monoidal category $\Vv$ is considered as a monoidal 2-category with trivial $2$-cells. We show that this monoidal bicategory has some interesting features, in particular under some mild conditions the forgetful functor to $\Span$ is a $2$-opfibration.

The conceptual advantage of our approach is that the underlying monoid and co\-monoid structure of an oplax bimonoid or Hopf monoid live over the same 
monoidal product, which places our concept closer to the self-dual notions of classical Hopf algebras and weak Hopf algebras, restoring
this feature for Hopf categories which was lost in their interpretation as Hopf monads in \cite{Gabipolyads}. 
In fact, the latter description and the current one are related
via a {\em dimension shift} in the following sense:
a Hopf category in \cite{Gabipolyads} is 
a $1$-cell in a certain monoidal bicategory, whereas in our setting
it is a $0$-cell in a (different) symmetric monoidal bicategory. 

We finish the paper by describing the notion of a Frobenius pseudomonoid in $\Span|\Vv$, leading to the notion of {\em Frobenius 
$\Vv$-category}.
Although both Hopf and Frobenius categories are viewed as pseudomonoids and pseudocomonoids in the same monoidal bicategory (albeit with different 
compatibility conditions), by taking a closer look it becomes evident that the algebra and coalgebra structures for a Hopf category are of different
`type', whereas for a Frobenius category both algebra and coalgebra structure are
of the same `type'. In the subsequent \cite{paper1b} we show that a Hopf $\Vv$-category that is locally rigid admits the structure of a Frobenius 
$\Vv$-category, providing a generalization of the well-known Larson Sweedler theorem.

Our paper is structured as follows. In \cref{sec:Preliminaries}, we recall some known
results and fix notation, especially for monoidal bicategories and pseudo(co)monoids.
In \cref{sec:oplaxHopf} we introduce oplax bimonoids and oplax Hopf monoids, and
prove some first results for them. \cref{sec:X2structures} consists of some elementary
observations about the set $X^2$ and more generally groupoids viewed as objects in $\Span$, where they give
examples of oplax Hopf monoids and Frobenius monoids.
In \cref{sec:SpanV}, we extensively describe the symmetric monoidal bicategory $\Span|\Vv$ and show that the forgetful functor to $\Span$ is a 
$2$-opfibration.
In \cref{sec:FrobeniusHopfVCatsSpanV} we show that Hopf
categories are oplax Hopf monoids in $\Span|\Vv$, as well as intermediate structures like enriched (op)categories and morphisms 
between them can all be realized in that setting. Finally, in 
\cref{FrobeniusVcats} we study the Frobenius monoids in $\Span|\Vv$, which we term Frobenius $\Vv$-categories.
In the appendix we list the various coherence conditions for the objects and morphisms we introduced.

%% file: 2-prelimi.tex
\section{Preliminaries}\label{sec:Preliminaries}

In this section, we provide some background material and fix notation and terminology for what
follows.

\subsection{Bimonoids, Hopf monoids and Frobenius monoids}\label{subsec:bimonoids}

We assume familiarity with the theory of (braided) monoidal categories
$(\Vv,\otimes,I,\alpha,\lambda,\braid)$; see e.g.\ \cite{McLane} or \cite{BraidedTensorCats}. Since any monoidal category is monoidally equivalent to a strict monoidal category, we will assume from now on without loss of generality that the associativity constraint $\alpha$ and the unitality constraint $\lambda$ are indentities.
Recall that a monoid (algebra) in $\Vv$ is a triple $(A,\mlt\colon A\otimes A\to A,\uni\colon I\to A)$ satisfying the usual associativity and unitality constraints. Dually, a comonoid (coalgebra) in $\Vv$ will be denoted as $(C,\comlt,\couni)$. We denote the categories of monoids and comonoids respectively by $\Mon(\Vv)$ and $\Comon(\Vv)$ $\left(\cong\Mon(\Vv^\op)^\op\right)$.
When $\Vv$ is braided, both these categories inherit a
monoidal structure from $\Vv$.

A bimonoid (bialgebra) in a braided monoidal category is an object $M$ with a monoid structure $(M,\mlt,\uni)$ and a comonoid structure $(M,\lcomlt,\lcouni)$ that are compatible, in the sense that a bimonoid is equivalently both a comonoid in $\Mon(\Vv)$ and a monoid in $\Comon(\Vv)$. Diagrammatically this is expressed by
\begin{equation}\label{eq:bimon}
\begin{tikzcd}[row sep=.15in,column sep=.25in]
M\otimes M \ar[d,"\lcomlt\otimes\lcomlt"'] \ar[r,"\mlt"] & M \ar[dd,"\lcomlt"] \\ 
M\otimes M\otimes M\otimes M \ar[d,"1\otimes\braid\otimes 1"'] & \\
M\otimes M\otimes M\otimes M \ar[r,"\mlt\otimes\mlt"'] & M\otimes M
\end{tikzcd}
\begin{tikzcd}
I \ar[dr,"\uni\otimes\uni"']\ar[r,"\uni"] & M\ar[d,"\lcomlt"] \\
& M\otimes M
\end{tikzcd}
\begin{tikzcd}
M\otimes M \ar[d,"{\lcouni \otimes \lcouni}"'] \ar[r,"\mlt"] & M\ar[dl,"\lcouni"] \\
I\otimes I &
\end{tikzcd}
\begin{tikzcd}
I \ar[dr,"1"'] \ar[r,"\uni"] & M\ar[d,"\lcouni"] \\
& I
\end{tikzcd}
\end{equation}
Bimonoids form a category $\Bimon(\Vv)$ whose morphisms
are simultaneously monoid and comonoid morphisms. 

Moreover, a Hopf monoid (Hopf algebra) is a bimonoid $(H,\mlt,\uni,\comlt,\couni)$ equipped with an antipode $\atpd\colon H\to H$
satisfying $\mlt\circ(s\otimes1_H)\circ\comlt=\uni\circ\couni=\mlt\circ(1_H\otimes s)\circ\comlt$.
Since $H$ has both a monoid and comonoid structure,
we can define the convolution product $f\odot g$ of any two endomorphisms $f,g \in \ca{V}(H,H)$
as the composite $\mlt\circ(f\otimes g)\circ\comlt$;
this is an associative binary operation on $\ca{V}(H,H)$ with $I^\odot := \uni\circ\couni$ as its unit. 
Using this terminology, the defining property of an antipode is that it is inverse to $1_H$ under the convolution product.
Hence, the antipode is unique when it exists, and bimonoid morphisms between Hopf monoids
can be seen to automatically commute with the antipode.
Therefore Hopf monoids form a full subcategory $\Hopf(\Vv)$ of $\Bimon(\Vv)$.
Equivalently to the existence of an antipode on a bimonoid $H$, a Hopf monoid can be characterised via the invertibility
of the canonical maps $(1_H\ot\mlt)\circ(\comlt\ot1_H)$ or $(\mlt\ot1_H)\circ(1_H\ot\comlt)$;
these are sometimes referred to as \emph{fusion morphisms}, see e.g.\ \cite{Street1998}.

Finally, a Frobenius monoid $A$ also has both a monoid and comonoid structure, but in this case they are compatible via the Frobenius laws
$(1_A\otimes \mlt)\circ(\comlt\otimes1_A)=\comlt\circ \mlt=
(\mlt\otimes1_A)\circ(1_A\otimes\comlt)$.
Morphisms between Frobenius monoids are simultaneously monoid and comonoid morphisms. We denote this category
by $\Frob(\ca{V})$.
For more details about Frobenius algebras,
see e.g.\ \cite{Frobeniusmonads}. 

\subsection{Enriched categories and opcategories}\label{enrichedstuff}

Recall that for a monoidal category $\Vv$, a $\Vv$-graph $A$ consists of a set of objects $X$ together with a 
family of hom-objects $\{A_{x,y}\}_{x,y\in X}$ in $\Vv$;
we will use the notation $A_{x,y}$ rather than the more common $A(x,y)$.
Morphisms of $\Vv$-graphs are functions $f$ between the sets of objects, together with a family of arrows $F_{xy}\colon A_{x,y}\to B_{fx,fy}$ in $\Vv$.
Together these form a category $\Vv$-$\Grph$.

A $\Vv$-\emph{category}
is a $\Vv$-\emph{graph} $\{A_{x,y}\}_{x,y\in X}$ together with composition (or multiplication) laws
$\mlt_{xyz}\colon A_{x,y}\otimes A_{y,z}\to A_{x,z}$ and identities $\uni_x\colon I\to A_{x,x}$
that satisfy the usual associativity and unity conditions. 
A $\Vv$-functor is a $\Vv$-graph morphism
that respects composition and units, and together these form a category $\Vv$-$\Cat$.
If $\Vv$ is braided monoidal, every $\Vv$-category $A$ has an \emph{opposite} $\Vv$-category $A^\op$ with the same objects and
hom-objects $A^\op_{x,y}\coloneqq A_{y,x}$, where composition arises from the one in $A$ composed with the braiding. More on the subject of enriched categories can be found in ~\cite{Kelly}. 

A $\ca{V}$-\emph{opcategory} $C$ is a category enriched in $\ca{V}^\op$~\cite[\S 9]{Monoidalbicatshopfalgebroids}.
Explicitly, there exist cocomposition and coidentity 
families of arrows in $\Vv$
\begin{equation}\label{Vopcat}
\comlt_{xyz}\colon C_{x,z}\to C_{x,y}\otimes C_{y,z},\quad \couni_{x}\colon C_{x,x}\to I
\end{equation}
satisfying coassociativity and counity axioms.
A $\Vv$-\emph{opfunctor} is a $\Vv^\op$-functor, so there is a category $\VopCat=\Vv^\op\textrm{-}\Cat$.


\subsection{Monoidal bicategories and pseudo(co)monoids}\label{sec:pseudomon}

Recall that a \emph{monoidal bicategory} $\ca{K}$ is a one-object tricategory~\cite[Def.~2.6]{CoherenceTricats}, namely a bicategory 
equipped with a pseudofunctor $\otimes\colon\Kk\times\Kk\to\Kk$ and a unit object $I\colon\mathbf{1}\to\Kk$ which are associative and unital up to 
coherent equivalence; for the explicit description see~\cite{Carmody,Pries}. A \emph{monoidal 2-category} is one whose underlying bicategory is 
really a 2-category. Notice that this is different from a \emph{2-monoidal 2-category} whose tensor product is a strict 2-functor, namely where 
$(f\ot g)\circ(h\ot k)=(f\circ h)\ot (g\circ k)$ rather than being coherently isomorphic as in the pseudo case.
A \emph{weak monoidal pseudofunctor} $\ps{F}\colon\ca{K}\to\ca{L}$ between monoidal bicategories preserves the monoidal structure via pseudonatural 
maps $\ps{F}(A)\ot\ps{F}(B)\to\ps{F}(A\ot B)$ and $\ps{F}(1_A)\to 1_{\ps{F}A}$ along with associativity and unity invertible modifications subject to 
coherence axioms, see \cite{Monoidalbicatshopfalgebroids}.

It is well-known that any monoidal bicategory is monoidally biequivalent to a 
monoidal $2$-category, and more precisely a {\em Gray monoid}: this is the one-object case of the coherence theorem for tricategories \cite{Coherence3dim,CoherenceTricats}.
A Gray monoid \cite{BaezHD,Monoidalbicatshopfalgebroids}, also called a \emph{semistrict monoidal $2$-category},
is a monoidal $2$-category with a tensor product which is strictly associative and unital 
on objects, but still a pseudofunctor and not a strict $2$-functor. In particular, for given $1$-cells $f\colon A\to B$ and $g\colon C\to D$ there exists a suitable invertible $2$-cell $c_{f,g}:(f\ot 1)\circ(1\ot g)\to (1\ot g)\circ (f\ot 1)$, natural in $f$ and $g$ and satisfying certain coherence conditions. By convention, we denote the $1$-cell $(f\ot 1)\circ(1\ot g)$ by $f\ot g$.
In what follows, whenever we consider a monoidal bicategory or monoidal 2-category, we will always suppose that it is a Gray monoid, where we suppress the morphisms $c_{f,g}$ and their inverses or (vertical and horizontal) compositions.   

A monoidal 2-category is \emph{braided} when it comes equipped with a pseudonatural equivalence with components $\braid_{A,B}\colon A\otimes B\to 
B\otimes A$, and two invertible modifications relating the tensors of three objects, subject to axioms that are detailed in references such as 
\cite[Def.~6]{BaezHD}, \cite[Def.~12]{Monoidalbicatshopfalgebroids} and finalised in \cite[Def.~2.2]{CransBrMon2cat}; it is \emph{sylleptic} when there exists an invertible modification $v\colon\sigma\circ\sigma\Rrightarrow1$ and 
\emph{symmetric} when $\braid_{A,B}\circ v_{A,B}=v_{B,A}\circ\braid_{B,A}$, subject to appropriate axioms. On the other hand, its weakened version of a braided monoidal bicategory as described in \cite[\S~2.4]{CoherenceBrMonBicat} is computationally more challenging; due to a coherence theorem though, every braided monoidal bicategory is braided monoidally biequivalent to a braided monoidal 2-category, as shown in \cite[Thm~ 27]{CoherenceBrMonBicat}.
In more detail, there is an induced braided structure on the biequivalent Gray monoid, which is then adjusted in order to constitute an actual braided monoidal 2-categorical structure as defined in \cite[Def.~2.2]{CransBrMon2cat}.
Furthermore, the respective coherence result for symmetric monoidal bicategories can be found in \cite{GurskiOsorno}. 
Therefore in what follows, we often work without loss of generality in the simpler context of a braided or symmetric monoidal 2-category (Gray 
monoid) $\ca{K}$.

A \emph{pseudomonoid} \cite[\S 3]{Monoidalbicatshopfalgebroids} (or \emph{monoidale}) $(A,\mlt,\uni,\alpha,\ell,r)$
in $\ca{K}$ is an object $A$ with multiplication $\mlt: A\ot A\to A$, 
unit $\uni: A \to I$, and invertible 2-cells
\begin{equation}\label{alphalambdarho}
\begin{tikzcd}
A\otimes A\otimes A\ar[r,"1\otimes\mlt"]\ar[d,"\mlt\otimes 1"']\ar[dr,phantom,"\scriptstyle\stackrel{\alpha}{\cong}"description]
& A\otimes A\ar[d,"\mlt"]\ar[d,"\mlt"]
& A\ar[r,"1\otimes\uni"]\ar[dr] & A\otimes A\ar[d,"\mlt","\stackrel{\ell}{\cong}\quad"'near start,
"\quad\;\stackrel{r}{\cong}"near start] &
A\ar[l,"\uni\otimes 1"']\ar[dl] \\
A\otimes A\ar[r,"\mlt"'] & A && A &
\end{tikzcd}
\end{equation}
satisfying appropriate coherence conditions found in \cref{sec:oplaxmaps}.
For example, a pseudomonoid in the cartesian monoidal 2-category $(\Cat,\times,\mathbf{1})$
is a monoidal category.
Dually, we have the notions of a \emph{pseudocomonoid} 
in a monoidal 2-category, denoted by $(C,\lcomlt,\lcouni,\beta,s,t)$.

An \emph{oplax} (or \emph{opmonoidal}) \emph{morphism} 
between pseudomonoids is a 1-cell $f\colon A\to B$ equipped with 2-cells
\begin{equation}\label{oplax2cells}
\begin{tikzcd}[row sep=.5in, column sep=.5in]
A\otimes A\ar[dr,phantom,"\Downarrow{\scriptstyle\phi}"description]\ar[r,"\mlt"]\ar[d,"f\otimes f"'] & |[alias=doma]| A\ar[d,"f"] \\
|[alias=coda]|B\otimes B\ar[r,"\mlt"'] & B
\end{tikzcd}\qquad
\begin{tikzcd}[row sep=.5in, column sep=.5in]
I\ar[dr,phantom,bend left=15,"\Downarrow{\scriptstyle\phi_0}"description]\ar[r,"\uni"]\ar[dr,bend right,"\uni"'{name=coda,below}] & |[alias=doma]| A\ar[d,"f"] \\
& B
\end{tikzcd}
\end{equation}
such that compatibility conditions for units and multiplications hold, see \cref{sec:oplaxmaps}.
A \emph{lax morphism} of pseudomonoids is defined similarly, with the 2-cells $\phi,\phi_0$
pointing in the opposite direction and the conditions adjusted accordingly. 
For example, (op)lax morphisms between pseudomonoids in 
$\Cat$ are precisely (op)lax monoidal functors. If $\ca{K}$ is a monoidal 2-category, then pseudomonoids together with (op)lax morphisms form 
categories $\PsMon_\opl(\ca{K})$ and $\PsMon_\lax(\ca{K})$.
Dually, we can talk of oplax morphisms between pseudocomonoids; if $\Kk^\op$ is the monoidal 2-category with reversed 
1-cells, it is the case that $\PsMon_\opl(\Kk^\op)\cong\PsComon_\opl(\Kk)^\op$.

In fact, for any monoidal bicategory $\ca{K}$, $\PsMon_\opl(\ca{K})$ is also a bicategory: if $(f,\phi,\phi_0)$ and $(g,\psi,\psi_0)$ are two oplax morphisms between pseudomonoids $A$ and $B$, a 2-cell between them is some $\alpha\colon f\Rightarrow g$ which is compatible with the oplax structure 2-cells, see \cref{sec:oplaxmaps}.
The similarly defined bicategory $\PsMon_\lax(\ca{K})$ was denoted by $\Mon(\ca{K})$ in \cite{CLS}.

Furthermore, when $\Kk$ is a braided monoidal bicategory with braiding $\braid$, the bicategory
$\PsMon_\opl(\ca{K})$ obtains a monoidal structure itself: the multiplication and unit on the tensor product of two pseudomonoids $A$ and $B$ is
\begin{gather}\label{monoidalpseudomonoids}
 A \ot B \ot A \ot B \xrightarrow{ 1\ot \braid \ot1} A \ot A \ot B \ot B \xrightarrow{\mlt \ot \mlt} A \ot B \\
 I\cong I\ot I\xrightarrow{\uni\ot\uni}A\ot B \nonumber
\end{gather}
and the coherence data can be readily constructed also. This can be seen as a result of the fact that the braided structure of a monoidal bicategory endows the pseudofunctor $\otimes$ with a monoidal structure, as discussed in \cite[117]{Monoidalbicatshopfalgebroids}, hence it preserves pseudomonoids by \cite[Prop.~5]{Monoidalbicatshopfalgebroids}.

Finally, a \emph{Frobenius pseudomonoid} inside a monoidal bicategory is an object with a pseudomonoid and pseudocomonoid
structure $(A,\mlt,\uni,\comlt,\couni)$ together with isomorphisms
\begin{equation}\label{pseudoFrobcond}
\begin{tikzcd}[row sep=.15in,column sep=.2in]
A\ot A\ar[rr,"\comlt\ot1"]\ar[dd,"1\ot\comlt"']\ar[dr,"\mlt"description] &&
A\ot A\ot A\ar[dd,"1\ot\mlt"] \\
 & A\ar[dr,"\comlt"description]\ar[ru,phantom,"\scriptstyle\stackrel{\phi}{\cong}"]\ar[ld,phantom,"\scriptstyle\stackrel{\psi}{\cong}"] & \\
A\ot A\ot A\ar[rr,"\mlt\ot1"'] && A\ot A
\end{tikzcd}
\end{equation}
satisfying certain coherence conditions, see \cref{sec:Frobpseudo}. For a detailed study of such objects and equivalent formulations see
\cite{Frobeniusmonads,LaudaFrobAlgs,Dualsinvert}.
If $f\colon A\to B$ is a 1-cell between Frobenius pseudomonoids,
we may choose any of the four combinations of lax or oplax structures between the pseudo(co)monoids for forming a bicategory.
For example, we can denote by $\sf{Frob}_{\opl,\opl}(\ca{K})$ the category of Frobenius pseudomonoids with oplax pseudomonoid, oplax pseudocomonoid
morphisms between them.

%% file: 3-oplax.tex
\section{Oplax bimonoids and Hopf monoids}\label{sec:oplaxHopf}

In this section, we define oplax variations of bimonoids and Hopf monoids in a braided monoidal bicategory
and examine some of their basic properties. These serve as basic concepts and tools for the following sections.

\subsection{Oplax bimonoids}
Previously we described the monoidal bicategory $\PsMon_{\opl}(\Kk)$ of pseudo\-mo\-noids
with oplax morphisms and 2-cells for a braided monoidal bicategory $\Kk$; one may consider pseudocomonoids therein and establish a notion of `oplax bimonoid', namely an object with a pseudomonoid and pseudocomonoid structure and an oplax interaction between them. 

\begin{definition}\label{oplaxbimonoid}
In a braided monoidal bicategory $(\Kk,\otimes,I,\braid)$, an \emph{oplax bimonoid} 
is an object $M$ in $\Kk$ endowed with a pseudomonoid structure $(M,\mlt,\uni)$
and a pseudocomonoid structure $(M,\lcomlt,\lcouni)$ along with 2-cells 
\begin{gather}\label{phi1}
\begin{tikzcd}[column sep=.7in,ampersand replacement=\&]
M\otimes M\ar[ddr,phantom,"\scriptstyle{\Downarrow\theta}"description] \ar[d,"\lcomlt\otimes\lcomlt"'] \ar[r,"\mlt"] \& M \ar[dd,"\lcomlt"] \\ 
M\otimes M\otimes M\otimes M \ar[d,"1\otimes\braid\otimes 1"'] \& \\
M\otimes M\otimes M\otimes M \ar[r,"\mlt\otimes\mlt"'] \& M\otimes M
\end{tikzcd}  \\ 
\begin{tikzcd}[ampersand replacement=\&]
I \ar[d,"\sim"']\ar[r,"\uni"]\ar[dr,phantom,"\scriptstyle{\Downarrow\theta_0}"] \& M\ar[d,"\lcomlt"] \\
I\otimes I\ar[r,"\uni\otimes\uni"'] \& M\otimes M
\end{tikzcd}
\qquad
\begin{tikzcd}[ampersand replacement=\&]
M\otimes M \ar[d,"{\lcouni \otimes \lcouni}"'] \ar[r,"\mlt"]\ar[dr,phantom,"\scriptstyle{\Downarrow\chi}"] \& M\ar[d,"\lcouni"] \\
I\otimes I \ar[r,"\sim"'] \& I
\end{tikzcd}
\qquad
\begin{tikzcd}[ampersand replacement=\&]
I \ar[d,"1"']\ar[dr,phantom,"\scriptstyle{\Downarrow\chi_0}"] \ar[r,"\uni"] \& M\ar[d,"\lcouni"] \\
I\ar[r,"1"']  \& I
\end{tikzcd}\label{phi234}
\end{gather}
expressing that comultiplication and counit are oplax pseudomonoid maps as in \cref{oplax2cells}.
These data satisfy certain axioms which are explicitly written in \cref{oplaxbimonoidaxioms}.
\end{definition}

The above notion indeed generalizes ordinary bimonoids by essentially inserting 2-cells inside the commutative diagrams \cref{eq:bimon}, subject to coherent conditions. For example, any bimonoid in a braided monoidal category $\Vv$ viewed as a ($2$-)monoidal $2$-category with trivial $2$-cells is an oplax bimonoid, strict in that the (co)associativity and (co)unit constraints are identities. Moreover, an oplax bimonoid in that strict sense in the cartesian $2$-category $\Kk=(\Cat,\times,\mathbf{1})$ is a strict monoidal category, since any object has a unique strict comonoid structure. A general oplax bimonoid in $\Cat$ is up to isomorphism a monoidal category: due to $\mathbf{1}$ being terminal, a pseudocomonoid is up to isomorphism still the trivial one namely the comultiplication is isomorphic to the diagonal, and $\chi_0,\chi=\mathrm{\id}$ whereas $\theta,\theta_0$ are uniquely determined isomorphisms. Non-trivial examples will be provided in the subsequent sections.


Morphisms between oplax bimonoids are oplax maps between the corresponding pseudocomonoids in $\PsMon_\opl(\ca{K})$. More explicitly, we 
have the following definition.

\begin{definition}\label{def:bimonoidmap}
An \emph{oplax bimonoid morphism} between oplax bimonoids $M$ and $N$ with structure 2-cells
$(\theta,\theta_0,\chi,\chi_0)$ and
$(\xi,\xi_0,\omega,\omega_0)$ in a braided monoidal bicategory $\ca{K}$ is a 1-cell $f \colon M \to N$ that
is both an oplax pseudomonoid morphism and an oplax pseudocomonoid morphism  with structure 2-cells
\begin{equation}\label{oplaxbimonoidmap}
\begin{tikzcd}[row sep=.5in, column sep=.5in]
M\otimes M\ar[dr,phantom,"\Downarrow{\scriptstyle\phi}"description]\ar[r,"\mlt"]\ar[d,"f\otimes f"'] & |[alias=doma]| M\ar[d,"f"] \\
|[alias=coda]|N\otimes N\ar[r,"\mlt"'] & N
\end{tikzcd}\quad
\begin{tikzcd}[row sep=.5in, column sep=.5in]
I\ar[dr,phantom,bend left=15,"\Downarrow{\scriptstyle\phi_0}"description]\ar[r,"\uni"]\ar[dr,bend right,"\uni"'{name=coda,below}] & |[alias=doma]| M\ar[d,"f"] \\
& N
\end{tikzcd}\quad
\begin{tikzcd}[row sep=.5in, column sep=.5in]
M\ar[dr,phantom,"\Downarrow{\scriptstyle\psi}"description]\ar[r,"\lcomlt"]\ar[d,"f"'] & |[alias=doma]| M\ot M\ar[d,"f\ot f"] \\
|[alias=coda]|N\ar[r,"\lcomlt"'] & N\ot N
\end{tikzcd}\quad
\begin{tikzcd}[row sep=.5in, column sep=.5in]
M\ar[dr,phantom,bend right=15,"\Downarrow{\scriptstyle\psi_0}"description]\ar[dr,bend left,"\lcouni"{name=doma}]\ar[d,"f"'] &  \\
|[alias=coda]| N\ar[r,"\lcouni"'] & I
\end{tikzcd}
\end{equation}
satisfying four axioms expressing that $\psi$ and $\psi_0$ are 2-cells between oplax pseudomonoid maps; axioms are explicitly recorded in 
\cref{oplaxbimonoidmapaxioms}.
\end{definition}

Oplax bimonoids form a bicategory $\OplBimon(\ca{K})=\PsComon_{\opl}\left(\PsMon_\opl(\ca{K})\right).$
Notice that the axioms oplax bimonoids satisfy are very similar to those of duoidal categories~\cite{Species}, which can be explained by observing that duoidal categories are pseudomonoids in $\PsMon_\opl(\Kk)$ or equivalently in $\PsMon_\lax(\Kk)$, for $\Kk=\Cat$. Hence these notions are `half-dual' to one other. 

As mentioned above, a weak monoidal pseudofunctor between monoidal bicategories preserves pseudomonoids and analogously a (strong) monoidal one preserves pseudocomonoids as well. Below we establish that a braided monoidal pseudofunctor as in \cite[Def.~14]{Monoidalbicatshopfalgebroids} preserves oplax bimonoids between braided monoidal bicategories.

\begin{proposition}\label{prop:Fpreserves}
Let $\ca{F}\colon\ca{K}\to\ca{L}$ be a braided monoidal pseudofunctor of bicategories.
If $M$ is an oplax bimonoid in $\Kk$, then $\ca{F}M$ is an oplax bimonoid in $\ca{L}$. 
\end{proposition}

\begin{proof}
Suppose $(M,\mlt,\uni,\lcomlt,\lcouni,\theta,\theta_0,\chi,\chi_0)$ is an oplax bimonoid as in \cref{oplaxbimonoid}. Then the pseudomonoid and pseudocomonoid structure on $\ca{F}M$ are given by
\begin{gather*}
\ca{F}M\ot\ca{F}M\to\ca{F}(M\ot M)\xrightarrow{\ca{F}\mlt}\ca{F}M, \qquad I\to\ca{F}I\xrightarrow{\ca{F}\uni}\ca{F}M \\
\ca{F}M\xrightarrow{\ca{F}\lcomlt}\ca{F}(M\ot M)\to\ca{F}M\ot\ca{F}M, \qquad \ca{F}M\xrightarrow{\ca{F}\lcouni}\ca{F}I\to I
\end{gather*}
where the unnamed arrows are the structure maps of the strong monoidal pseudofunctor $\ca{F}$, namely two equivalences and its pseudoinverses.
We can then equip $\ca{F}M$ with an oplax bimonoid structure \cref{phi1,phi234} as follows: the 2-cell $\theta_{\ca{F}M}$ is formed as the composite
\begin{displaymath}
 \begin{tikzcd}[row sep=.2in]
  \ca{F}M\ot\ca{F}M\ar[dr,phantom,"\scriptstyle\cong"]\ar[r]\ar[d,"\ca{F}\lcomlt\ot\ca{F}\lcomlt"'] & \ca{F}(M\ot M)\ar[r,"\ca{F}\mlt"]\ar[dr,phantom,"\scriptstyle\Downarrow\ca{F}\theta"]\ar[d,"\ca{F}(\lcomlt\ot\lcomlt)"'] & \ca{F}M\ar[d,"\ca{F}\lcomlt"] \\
  \ca{F}(M\ot M)\ot\ca{F}(M\ot M)\ar[ddr,phantom,"\scriptstyle\stackrel{(*)}{\cong}"]\ar[r]\ar[d] & \ca{F}(M\ot M\ot M\ot M)\ar[d,"\ca{F}(1\ot\braid\ot1)"'] & \ca{F}(M\ot M)\ar[dd] \\
  \ca{F}M\ot\ca{F}M\ot\ca{F}M\ot\ca{F}M\ar[d,"1\ot\braid\ot1"'] &  \ca{F}(M\ot M\ot M\ot M)\ar[r,phantom,"\scriptstyle\cong"]\ar[ur,"\ca{F}(\mlt\ot\mlt)"description]\ar[d] & \phantom{A} \\
  \ca{F}M\ot\ca{F}M\ot\ca{F}M\ot\ca{F}M\ar[r] & \ca{F}(M\ot M)\ot\ca{F}(M\ot M)\ar[r,"\ca{F}\mlt\ot\ca{F}\mlt"'] & \ca{F}M\ot\ca{F}M
 \end{tikzcd}
\end{displaymath}
where the two unnamed invertible 2-cells are due to pseudonaturality of the monoidal structure maps of $\ca{F}$, whereas $(*)$ comes from $\ca{F}$ being (strong) monoidal and braided. Similarly the rest three 2-cells are constructed as composites of $\ca{F}\theta_0$, $\ca{F}\chi$, $\ca{F}\chi_0$ with coherent isomorphisms, and the conditions of \cref{oplaxbimonoidaxioms} are satisfied.
\end{proof}

In the spirit of \cite{Species}, one could define a notion of `bilax' monoidal pseudofunctor between braided monoidal bicategories which would be precisely such that it preserves oplax bimonoids: essentially a weak and an opweak monoidal pseudofunctor which comes with coherent isomorphisms of the form $(*)$. 
For our purposes though, the above setting suffices for our example at hand -- the strict monoidal forgetful functor of bicategories of 
\cref{sec:functorU}.

\subsection{Oplax modules for oplax bimonoids}

One of the fundamental properties of a bimonoid in an ordinary monoidal category is that there exists a monoidal structure on its category of modules, with a strict monoidal forgetful functor. In order to generalize this for oplax bimonoids in monoidal bicategories, we introduce the notion of an oplax module. 

\begin{definition}\label{def:oplaxmod}
Let $(M,\mlt,\uni)$ be a pseudomonoid in a monoidal bicategory $\Kk$.
A {\em (right) oplax $M$-module} is an object
$X\in \Kk$ equipped with a $1$-cell $\rho\colon X\ot M\to X$ and 2-cells
\begin{equation}\label{oplaxmodstructure}
\begin{tikzcd}[row sep=.5in, column sep=.5in]
X\ot M\otimes M\ar[dr,phantom,"\Downarrow{\scriptstyle\xi}"description]\ar[r,"1\ot\mlt"]\ar[d,"\rho\ot1"'] & |[alias=coda]| X\ot M\ar[d,"\rho"] \\
|[alias=doma]|X\otimes M\ar[r,"\rho"'] & X
\end{tikzcd}\quad
\begin{tikzcd}[row sep=.5in, column sep=.5in]
X\ar[r,"1\ot\uni"]\ar[dr,phantom,bend left=10,"\Downarrow{\scriptstyle\xi_0}"']\ar[dr,bend right,"\id"',""{name=doma,above}] &
|[alias=coda]| X\ot M\ar[d,"\rho"] \\
& X
\end{tikzcd}
\end{equation}
satisfying appropriate compatibility axioms listed in \cref{oplaxmoduleaxioms}.
When $\xi$ and $\xi_0$ are invertible, we say that $X$ is a \emph{pseudo $M$-module}.

A \emph{(right) oplax $M$-module morphism} $(f,\phi):(X,\rho_X,\xi,\xi_0)\to (Y,\rho_Y,\zeta,\zeta_0)$ consists of a 1-cell $f:X\to Y$ in $\Kk$ and a $2$-cell
\begin{equation}\label{eq:oplaxmorphism}
\begin{tikzcd}
 X\ot M\ar[d,"f\ot 1"']\ar[r,"\rho_X"]\ar[dr,phantom,"\Downarrow{\scriptstyle\phi}"description] & X\ar[d,"f"] \\
 Y\ot M\ar[r,"\rho_Y"'] & Y
\end{tikzcd}
\end{equation}
satisfying coherence conditions \cref{oplaxmodmorph1,oplaxmodmorph2}. 
If $\phi$ is invertible, we say that $(f,\phi)$ is an oplax $M$-module {\em pseudomorphism}.
Finally, a \emph{transformation} between oplax $M$-module morphisms $(f,\phi),(g,\psi)$ is a $2$-cell
$\alpha\colon f\Rightarrow g$ which is compatible with the structure morphisms, see \eqref{oplaxmodtransax}.
\end{definition}

Oplax $M$-modules together with their oplax morphisms and transformations form a bicategory
$\OplMod^{\oplax}_M$, which comes with an evident forgetful strict functor to $\Kk$. 
Similarly, we have the bicategory $\OplMod^{\pseudo}_M$ of oplax $M$-modules with pseudomorphisms and transformations.

\begin{remark}
Clearly, one could also consider (op)lax (co)modules over a pseudo(co)monoid.
In fact, the notion of pseudomodule as defined above is closely related to the notion of algebra for the pseudomonad
$(-\ot M)$ on $\Kk$ as considered in \cite{Marmolejo1999}. However, it is important to remark that 
some (but not all) structure 2-cells in our context have the opposite direction of those e.g. in \cite[p.96]{Marmolejo1999}.

Since the appearance of the first pre-print version of this work, oplax modules over a {\em skew monoidale} also have been studied 
under the name of {\em oplax actions} in \cite{Ramon}. Therein, oplax modules are required to satisfy an additional axiom, which however follows from 
the other axioms in our setting (see \cite{Marmolejo1999}).
\end{remark}

\begin{proposition}\label{tannakaeasy}
For any oplax bimonoid $M$, the bicategories $\OplMod^{\pseudo}_M$ and $\OplMod^{\oplax}_M$ have a monoidal structure, such that the forgetful functors $\OplMod^{\pseudo}_M\to \OplMod^{\oplax}_M\to \Kk$ are strict monoidal.
\end{proposition}
%
%

\begin{proof}
Suppose $(\Kk,\otimes,I,\braid)$ is a braided monoidal bicategory. Take an oplax bimonoid $(M,\mlt,\uni,\lcomlt,\lcouni)$ with structure 2-cells 
$(\theta,\theta_0,\chi,\chi_0)$
as in \cref{oplaxbimonoid}, and two oplax $M$-modules $(X,\rho_X,\xi,\xi_0)$ and $(Y,\rho_Y,\zeta,\zeta_0)$ as in \cref{def:oplaxmod}.
We can endow $X\ot Y$ with an oplax $M$-module structure as follows: the $M$-action is the composite
\begin{equation}\label{eq:coactiontensor}
\rho_{X\ot Y}\colon X\ot Y\ot M\xrightarrow{1\ot1\ot\lcomlt}X\ot Y\ot M\ot M\xrightarrow{1\ot\braid\ot1}
X\ot M\ot Y\ot M\xrightarrow{\rho_X\ot\rho_Y}X\ot Y
\end{equation}
and the structure 2-cells are given by
\begin{displaymath}
\begin{tikzcd}[column sep=.35in,row sep=.5in]
 X{\ot}Y{\ot}M{\ot}M\ar[rrr,"11\mlt"]\ar[d,"11\lcomlt1"']\ar[dr,"11\lcomlt\lcomlt"description]\ar[drrr,phantom,"\Downarrow{\scriptstyle11\theta}"description] &&& X{\ot}Y{\ot}M\ar[d,"11\lcomlt"] \\
X{\ot}Y{\ot}M{\ot}M{\ot}M\ar[r,"1111\lcomlt"']\ar[d,"1\braid11"'] & X{\ot}Y{\ot}M{\ot}M{\ot}M{\ot}M\ar[r,"111\braid1"]\ar[d,"1\braid111"'] & X{\ot}Y{\ot}M{\ot}M{\ot}M{\ot}M\ar[r,"11\mlt\mlt"]\ar[d,"1\braid1"]
& X{\ot}Y{\ot}M{\ot}M\ar[d,"1\braid1"] \\
X{\ot}M{\ot}Y{\ot}M{\ot}M\ar[d,"\rho\rho1"'] & X{\ot}M{\ot}Y{\ot}M{\ot}M{\ot}M\ar[r,"11\braid11"]\ar[d,"\rho\rho11"'] & X{\ot}M{\ot}M{\ot}Y{\ot}M{\ot}M\ar[dr,phantom,"\Downarrow{\scriptstyle\xi\zeta}"description]\ar[r,"1\mlt1\mlt"]\ar[d,"\rho1\rho1"'] & X{\ot}M{\ot}Y{\ot}M\ar[d,"\rho\rho"] \\
X{\ot}Y{\ot}M\ar[r,"1\lcomlt"'] & X{\ot}Y{\ot}M{\ot}M\ar[r,"1\braid1"'] & X{\ot}M{\ot}Y{\ot}M\ar[r,"\rho\rho"'] & X{\ot}Y
 \end{tikzcd}
\end{displaymath}
\begin{equation}\label{eq:xtensory}
\begin{tikzcd}[column sep=.6in,row sep=.4in]
  X\ot Y\ar[drrr,near end, bend left=15, phantom, "\Downarrow{\scriptstyle11\theta_0}"]
  \ar[dddrrr,phantom,bend right=10,"\Downarrow{\scriptstyle\xi_0\zeta_0}"']
  \ar[rrr,"11\uni"]\ar[drrr,"11\uni\uni"description]\ar[ddrrr,"1\uni1\uni"description]\ar[dddrrr,bend right,"1"']
  &&& X\ot Y\ot M\ar[d,"11\lcomlt"] \\
  &&& X\ot Y\ot M\ot M\ar[d,"1\braid1"] \\
  &&& X\ot M\ot Y\ot M\ar[d,"\rho_X\rho_Y"] \\
  &&& X\ot Y
 \end{tikzcd}
\end{equation}
where the empty squares should be filled by the coherence isomorphisms of the braided monoidal bicategory $\Kk$ (or its equivalent Gray monoid), and we have 
suppressed the $\ot$-symbol for morphisms. Combining the axioms \cref{oplaxmod1,oplaxmod2} for $X$ and $Y$ with
the oplax bimonoid axioms for $M$ as found in \cref{oplaxbimonoidaxioms}, we can verify that this constitutes
an oplax module structure on $X\otimes Y$.

Moreover, the monoidal unit $I$ of $\Kk$ can also be endowed with the structure of an oplax $M$-module, with action and structure 2-cells
\begin{equation}\label{eq:rhoI}
\rho_I\colon I\ot M\cong M\xrightarrow{\lcouni} I
\end{equation}

\begin{displaymath}
\begin{tikzcd}[column sep=.5in,row sep=.5in]
 I\ot M\ot M\cong M\ot M\ar[r,"\mlt"]\ar[d,"\lcouni\ot1"']\ar[dr,bend left=20,phantom,"\Downarrow{\scriptstyle\chi}"]\ar[dr,"\lcouni\ot\lcouni"description] & M\ar[d,"\lcouni"] \\
 M\ar[r,"\lcouni"'] & I
 \end{tikzcd}\qquad\qquad
 \begin{tikzcd}[column sep=.5in,row sep=.5in]
I\ar[r,"\uni"]\ar[dr,bend right,"1"']\ar[dr,bend left=10,phantom,"\Downarrow{\scriptstyle\chi_0}"] & M\ar[d,"\lcouni"] \\
& I
 \end{tikzcd}
\end{displaymath}

Given morphisms of oplax modules $(f,\phi):X\to Y$ and $(g,\psi):Z\to U$, define $(f,\phi)\ot (g,\psi)$ as $(f\ot g,\tau)$,  where $f\ot g$ is the tensor product in $\Kk$ and $\tau$ is given by
\begin{equation}\label{eq:ftensorg}
\begin{tikzcd}
X\ot Z\ot M \ar[d,"1\ot 1\ot \lcomlt"'] \ar[rr,"f\ot g\ot 1"] && Y\ot U\ot M \ar[d,"1\ot 1\ot \lcomlt"] \\
X\ot Z\ot M\ot M \ar[rr,"f\ot g\ot 1\ot 1"] \ar[d,"1\ot \braid\ot 1"']  && Y\ot U\ot M\ot M \ar[d,"1\ot \braid\ot 1"]\\
X\ot M\ot Z\ot M \ar[drr,phantom,"\Downarrow{\scriptstyle\phi\ot \psi}"] \ar[rr,"f\ot 1\ot g\ot 1"] \ar[d,"\rho_X\ot \rho_Z"'] && Y\ot M\ot U\ot M \ar[d,"\rho_Y\ot \rho_U"] \\
X\ot Z \ar[rr,"f\ot g"'] && Y\ot U
\end{tikzcd}
\end{equation}
Clearly, if the morphisms $(f,\phi)$ and $(g,\psi)$ are pseudo, namely if $\phi$ and $\psi$ above are invertible, then $f\ot g$ is a pseudo morphism as well, since \cref{eq:ftensorg} is a composite of invertible $2$-cells.
The tensor product of oplax module transformations is given by the tensor product of 2-cells in $\Kk$; there is no extra structure.

With the above descriptions, it can be verified that the category of oplax modules over an oplax bimonoid is a monoidal bicategory, and it directly 
follows that the forgetful functor to $\Kk$ strictly preserves the monoidal structure. 
\end{proof}

\begin{remark}\label{rem:strictmorphismtensor}
Notice that if $X$ and $Y$ are pseudo $M$-modules, $X\ot Y$ is naturally still only an oplax module even in the stricter 2-monoidal 2-category 
case, due to the non-invertible $2$-cells $\theta$ and $\theta_0$ of the oplax bimonoid used for the coherence $2$-cells \cref{eq:xtensory} of $X\ot 
Y$ in the above proof. Hence the bicategory of pseudomodules over an oplax bimonoid is in general not a monoidal bicategory.
\end{remark}

For a monoid $M$ in an ordinary monoidal category $\Cc$ (under some conditions, see e.g.\ \cite{Ver:reconstruction}), it is known that there is a bijective correspondence between bimonoid structures on $M$ and monoidal structures on the category of $M$-modules such that the forgetful functor to $\Cc$ is strict monoidal. Since the previous proposition lifts one part of this correspondence to oplax bimonoids, we now consider part of the converse direction by showing that
the monoidal structure on the bicategory of oplax modules over an oplax bimonoid completely determines its pseudocomonoid structure.

Let us assume that for a pseudomonoid $(M,\mlt,\uni)$, the bicategory $\OplMod^{\pseudo}_M$ is monoidal in such a way that the forgetful functor to $\Kk$ is strict monoidal.
Since $M$ is a pseudomonoid, it is a pseudo $M$-module $(M,\mlt)$ and so $M\ot M$ is an oplax $M$-module: denote its action by $\rho_{M\ot M}$ and structure 2-cells $(\xi^{M\ot M},\xi_0^{M\ot M})$ as in \cref{oplaxmodstructure}.
Then we can define a comultiplication $1$-cell as the composite
\begin{equation}\label{eq:comultiplication}
 \lcomlt\colon M\xrightarrow{\uni\ot\uni\ot1}M\ot M\ot M\xrightarrow{\rho_{M\ot M}} M\ot M.
\end{equation}
Moreover, the monoidal unit $I$ is also an oplax $M$-module with action denoted $\rho_I$ and 2-cells $(\xi^I,\xi_0^I)$; therefore a counit can be defined as
\begin{equation}\label{eq:counit}
 \lcouni\colon M\cong I\ot M\xrightarrow{\rho_I}I.
\end{equation}
Using notation as in \cref{oplaxbimonoid}, the required 2-cells $(\theta,\theta_0,\chi,\chi_0)$ for an oplax bimonoid structure on $M$ can be constructed as follows: $\theta$ is built from $\xi^{M\ot M}$ as the composite
\begin{equation}\label{eq:theta}
 \begin{tikzcd}[column sep=.6in,row sep=.5in]
  M{\ot}M\ar[r,"\mlt"]\ar[d,"\uni\uni1\uni\uni1"']\ar[dr,"\uni\uni11"description] & M\ar[r,"\uni\uni1"] & M{\ot}M{\ot}M\ar[ddd,"\rho_{MM}"] \\
  M{\ot}M{\ot}M{\ot}M{\ot}M{\ot}M\ar[dd,bend right=70,"\rho_{MM}\rho_{MM}"']\ar[d,"\rho_{MM}111"description] & M{\ot}M{\ot}M{\ot}M\ar[l,"111\uni\uni1"description]\ar[ur,"11\mlt"description]\ar[d,"\rho_{MM}1"description]\ar[ddr,phantom,bend left,"\Downarrow{\scriptstyle\xi^{MM}}"] & \\
  M{\ot}M{\ot}M{\ot}M{\ot}M\ar[d,"11\rho_{MM}"description] & M{\ot}M{\ot}M\ar[dr,"\rho_{MM}"description]\ar[l,"11jj1"description] & \\
  M{\ot}M{\ot}M{\ot}M\ar[r,"1\braid1"']& M{\ot}M{\ot}M{\ot}M\ar[r,"\mlt\mlt"'] & M{\ot}M  
 \end{tikzcd}
\end{equation}
Moreover, $\theta_0$ can be built using $\xi_0^{M\ot M}$ as the composite
\begin{equation}\label{eq:theta0}
\begin{tikzcd}[column sep=.5in,row sep=.5in]
I\ar[r,"\uni\ot\uni"]\ar[drr,bend right,"\uni\ot\uni"'] & M\ot M\ar[r,"1\ot\uni"]\ar[dr,bend right,"1"description]\ar[dr,phantom,"\Downarrow{\scriptstyle{\xi_0^{M\ot M}}}"description]
& M\ot M\ot M\ar[d,"{\rho_{M\ot M}}"] \\
&& M\ot M
 \end{tikzcd}
\end{equation}
In a similar way, $\chi$ and $\chi_0$ can be obtained from $\xi^I$ and $\xi^I_0$ as follows
\begin{equation}\label{eq:chichi0}
\begin{tikzcd}[column sep=.25in, row sep=.2in]
M\ot M\ar[rr,"\mlt"]\ar[dd,"\lcouni\ot\lcouni"']\ar[dr,"\lcouni\ot1"']\ar[ddrr,phantom,bend left,"\Downarrow{\scriptstyle\xi^I}"] && M\ar[dd,"\lcouni"'] \\
& M\ar[dr,"\lcouni"'] & \\
I\ot I\ar[rr] &&  I
\end{tikzcd}\qquad
\begin{tikzcd}[column sep=.25in, row sep=.2in]
I\ar[rr,"\uni"]\ar[ddrr,bend right,"1"']\ar[ddrr,phantom,bend left,"\Downarrow{\scriptstyle\xi_0^I}"'] && M\ar[dd,"\lcouni"] \\
\hole \\
&& I
\end{tikzcd}
\end{equation}

The following result shows that when we consider the monoidal structure on $\OplMod^{\pseudo}_M$ as described in \cref{tannakaeasy}, by the above construction we recover the initial pseudocomonoid structure of the oplax bimonoid $M$, which is a so-called {\em Tannakian recontruction theorem} for oplax bimonoids.

\begin{proposition}\label{reconstruction}
Suppose that $M$ is an oplax bimonoid. Then the 1-cells \cref{eq:comultiplication,eq:counit} and the $2$-cells \cref{eq:theta,eq:theta0,eq:chichi0} constructed using the monoidal structure of $\OplMod_M^\opl$ as described in \cref{tannakaeasy} endow the underlying pseudomonoid $M$ with an oplax bimonoid structure which is isomorphic to the initial one.
\end{proposition}
\begin{proof}
The following diagram can be filled with invertible $2$-cells, showing that the the reconstructed comultiplication \cref{eq:comultiplication} using \cref{eq:coactiontensor} is isomorphic to the initial comultiplication $\comlt$.
\[
\xymatrix{
M \ar[rr]^-{\mlt\ot\mlt\ot1} \ar[d]_{\comlt} && M\ot M\ot M \ar[rr]^-{1\ot1\ot\comlt}  && M\ot M\ot M\ot M \ar[d]^-{1\ot\braid\ot1}\\
M\ot M \ar[rrrr]^{\uni\ot1\uni\ot1} \ar@{=}[rrrrd] && && M\ot M\ot M\ot M \ar[d]^{\mlt\ot\mlt}\\
&&&& M\ot M
}
\]
Furthermore, by definition it is obvious that the reconstructed counit coincides with the initial one since $\rho_I$ itself is defined as isomorphic to the counit \cref{eq:rhoI}. In the same way, one observes that the reconstructed $2$-cells $\theta,\theta_0,\chi,\chi_0$ coincide with the initial ones and hence we obtain by reconstruction an isomorphic oplax bimonoid structure on $M$.
\end{proof}

\begin{remark}
A next step would be a {\em recognition theorem} for oplax bimonoids, in the following sense. Given a pseudomonoid $M$ in $\Kk$ such that the category 
$\OplMod^{\pseudo}_M$ is monoidal and the forgetful functor to $\Kk$ is strict monoidal, when do the constructions described above 
\cref{reconstruction} provide an oplax bimonoid structure on $M$, such that the monoidal structure of $\OplMod^{\pseudo}_M$ coincides with the one 
given in \cref{tannakaeasy}? In view of the $1$-categorical case (see e.g.\ \cite{Ver:reconstruction}), in order to obtain such a theorem one would 
need a suitable generator condition on the monoidal unit of the bicategory $\Kk$. Such a result, which would 
lead to a 
bijective correspondence between oplax bimonoid structures on a given pseudo monoid $M$ and monoidal structures on its category of oplax modules with 
a strict monoidal forgetful functor to $\Kk$, is beyond the scope of the present paper and left open for future work. 
\end{remark}

\subsection{Oplax inverses}

Having established the notion of an oplax bimonoid in \cref{oplaxbimonoid}, we now investigate the appropriate notion of an antipode in this generalized setting. For that purpose, we introduce the notion of an oplax inverse in a monoidal category as a special case Morita context. We begin with some motivating definitions; in what follows, $\circ$ denotes vertical composition of 2-cells, $*$ denotes horizontal composition of 2-cells and we suppress horizontal composition of 1-cells.

\begin{definition}\cite{ElK}\label{def:Moritacontext}
A {\em (wide) Morita context} $(A,B,p,q,\mu,\tau)$ in a bicategory $\Kk$ consists of two objects $A$ and $B$, two $1$-cells $p\colon A\to B$ and $q\colon B\to A$ and two $2$-cells $\mu\colon q p\Rightarrow 1_A$ and $\tau\colon p q\Rightarrow 1_B$ such that $1_p*\mu=\tau*1_p\colon p q p\Rightarrow p$ explicitly as in
\begin{displaymath}
\begin{tikzcd}[row sep=.05in]
& B\ar[dr,bend left=10,"q"] &&& \\
A\ar[ur,bend left=10,"p"]\ar[rr,bend right,"1_A"description]\ar[rr,phantom,"\Downarrow{\scriptstyle\mu}"description]\ar[rrrr,bend right=50,"p"description]\ar[rrrr,phantom,bend right,"{\scriptstyle\cong}"] && A\ar[rr,bend left,"p"]\ar[rr,bend right,"p"description]\ar[rr,phantom,"\Downarrow{\scriptstyle1_p}"description] && B
\end{tikzcd}
\begin{tikzcd}[row sep=.05in]
 \hole \\
 =
\end{tikzcd}
\begin{tikzcd}[row sep=.05in]
&&& A\ar[dr,bend left=10,"p"description] && \\
A\ar[rrrr,phantom,bend left,"{\scriptstyle\cong}"]\ar[rrrr,bend left=50,"p\circ(q\circ p)"]\ar[rr,bend left,"p"]\ar[rr,bend right,"p"description]\ar[rr,phantom,"\Downarrow{\scriptstyle1_p}"description]\ar[rrrr,bend right=50,"p"description]\ar[rrrr,phantom,bend right,"{\scriptstyle\cong}"] && B\ar[ur,bend left=10,"q"]\ar[rr,bend right,"1_B"description]\ar[rr,phantom,"\Downarrow{\scriptstyle\tau}"description] && B
\end{tikzcd}
\end{displaymath}
and similarly $1_q*\tau=\mu*1_q\colon qpq\Rightarrow q$.

A wide Morita context is called {\em strict} if the 2-cells $\mu$ and $\tau$ are invertible.
\end{definition}

We will henceforth call a wide Morita context {\em firm} if the above composite 2-cells $1_p*\mu,\tau*1_p,1_q*\tau,\mu*1_q$ are in fact invertible. The motivation for this terminology comes from the following example. Notice that a strict wide Morita context is always firm, but the inverse is not true; firmness serves as an intermediate notion.

There exists a straightforward notion of (iso)morphism of wide Morita contexts. From now on, we will suppress the word `wide' for Morita contexts.

\begin{examples}
In the bicategory $\Bim$ of algebras, bimodules and bilinear maps, wide Morita contexts as in \cref{def:Moritacontext} coincide with the classical notion of a Morita context. In that case, it is well-known that a Morita context $(A,B,p,q,\mu,\tau)$ is strict if and only if $p$ is finitely generated and projective both as a left $A$-module and a right $B$-module and there are isomorphisms $B{\cong}{_A{\sf End}}(p)$, $A{\cong}{\sf End}_B(p)^\op$, $q{\cong}{_A\Hom}(p,A)$ ${\cong} {\Hom_B}(p,B)$, if and only if the bimodules $p$ and $q$ induce an equivalence between the categories of $A$-modules and $B$-modules. Similarly, such a classical Morita context is firm if and only if the rings $R=p\ot_B q$ and $S=q\ot_A p$ are firm and $p$ is $R$-firmly projective as left $A$-module and $S$-firmly projective as right $B$-module; in this case, there is an equivalence between the categories of firm $R$-modules and firm $S$-modules, see \cite{BV:Moritafirm,Ver:equivalence}. 

In $\Bim$, it is therefore clear that any strict Morita context is firm, but not vice versa: not every firmly projective module is necessarily 
finitely generated and projective.
\end{examples}

Similarly to the essential uniqueness of adjoints in a bicategory, for a given 1-cell $p\colon A\to B$ it is known that there can exist only one (up to isomorphism) strict Morita context $(A,B,p,q,\mu,\tau)$.
The subsequent theorem generalizes this result for firm Morita contexts in bicategories, and is fundamental for what follows; before that, a lemma establishes some required identities. 

\begin{lemma}\label{alphabeta}
Let $(A,B,p,q,\mu,\tau)$ and $(A,B,p,q',\mu',\tau')$ be two firm Morita contexts. If $\alpha$, $\beta$, $\alpha'$ and $\beta'$ are the respective inverse 2-cells of $1_p*\mu $, $1_q*\tau $, $1_p*\mu'$ and $1_{q'}*\tau'$, the following hold:
\begin{multicols}{2}
\begin{enumerate}[(i)]
\item $1_q*\alpha=\beta*1_p$ \label{it:1}
\item $\alpha*1_p=1_q*\beta$ \label{it:2}
\item $(\alpha'*1_{qp})\circ\alpha= (1_{pq'}*\alpha)\circ\alpha'$ \label{it:3}
\end{enumerate}
\end{multicols}
as well as appropriate identities including $\alpha'$ and $\beta'$.
\end{lemma}

\begin{proof}
\ul{(i)}. 
By the defining axioms of a Morita context, all arrows (2-cells in $\Kk$) from left to right in the next diagram are equal, and therefore have equal inverses. By definition, the inverse of the upper arrows is $1_q*\alpha$ and the inverse of the lower arrows is $\beta*1_p$.
\[\xymatrix{
qpqp  \ar@{=}[d] \ar@<.5ex>[rr]^-{1_q*(1_p*\mu)} \ar@<-.5ex>[rr]_-{1_q*(\tau*1_p)} && qp \ar@/_2.0pc/[ll]_{1_q*\alpha} \ar@{=}[d] \\
qpqp \ar@<.5ex>[rr]^-{(1_q*\tau)*1_p} \ar@<-.5ex>[rr]_-{(\mu*1_q)*1_p} && pq \ar@/^2.0pc/[ll]^{\beta*1_p}
}\] 
Clause (ii) is proven in a similar way.\\
\ul{(iii)}. 
This 2-dimensional equality can be depicted as
\begin{displaymath}
\begin{tikzcd}[row sep=.4in,column sep=.5in]
A\ar[drr,phantom,"\Downarrow{\scriptstyle1_{qp}}"]\ar[rrr,bend left=40,"p"]\ar[rrr,bend left=20,phantom,"\Downarrow{\scriptstyle\alpha}"']\ar[r,"p"]\ar[dr,"p"'] & B\ar[r,"q"] & A\ar[dr,phantom,"\Downarrow{\scriptstyle\alpha'}"]\ar[r,"p"]\ar[d,"p"'] & B
\\
& B\ar[ur,"q"'] & B\ar[r,"q'"'] & A \ar[u,"p"'] 
\end{tikzcd}\quad=\quad
\begin{tikzcd}[row sep=.4in,column sep=.5in]
A\ar[dr,phantom,"\Downarrow{\scriptstyle\alpha}"]\ar[rrr,bend left=40,"p"]\ar[rrr,bend left=20,phantom,"\Downarrow{\scriptstyle\alpha'}"']\ar[r,"p"]\ar[d,"p"'] & B\ar[drr,phantom,"\Downarrow{\scriptstyle1_{pq'}}"]\ar[r,"q'"]\ar[dr,"q'"'] & A\ar[r,"p"] & B
\\
B\ar[r,"q"'] & A\ar[u,"p"'] & A\ar[ur,"p"'] &\phantom{A}
\end{tikzcd}
\end{displaymath}
Similarly to the proof of (i), we will show that both 2-cells have the same inverse. Consider the following diagram, where the parallel arrows are the same due to Morita context axioms, and the inner rectangle commutes due to the interchange law.
\[
\xymatrix{
pq'pqp \ar@{=}[d] \ar@<-.5ex>[rr]_-{\tau*1_{pq'}*1_{p}} \ar@<.5ex>[rr]^-{1_{pq'p}*\mu}  
&& pq'p \ar@<-.5ex>[rr]_-{\tau'*1_{p}} \ar@<.5ex>[rr]^-{1_{p}*\mu'} \ar@/_2.0pc/[ll]_-{1_{pq'}*\alpha}
&& p \ar@{=}[d] \ar@/_2.0pc/[ll]_-{\alpha'}
\\
pq'pqp
 \ar@<.5ex>[rr]^-{\tau'*1_{pqp}}  \ar@<-.5ex>[rr]_-{1_{p}*\mu'*1_{qp}} && pqp
\ar@<.5ex>[rr]^-{\tau*1_p} \ar@<-.5ex>[rr]_-{1_p*\mu} \ar@/^2.0pc/[ll]^{\alpha'*1_{qp}} && p \ar@/^2.0pc/[ll]^{\alpha}
}
\]
\end{proof}

\begin{theorem}\label{uniquenesslaxinverse}
If $(A,B,p,q,\mu,\tau)$ and $(A,B,p,q',\mu',\tau')$ are two firm Morita contexts, there is an isomorphism $q\cong q'$ that makes the above Morita contexts isomorphic.
\end{theorem}

\begin{proof}
Denote by $\alpha$ and $\beta$ the inverses of respectively $1_p*\mu $ and $1_q*\tau $, and similarly $\alpha'$ and $\beta'$ for the corresponding 2-cells of the alternate firm Morita context, as in \cref{alphabeta}. We will show that the following $2$-cells between $q$ and $q'$
\[
\xymatrix{\phi: q \ar[rr]^-\beta && qpq \ar[rr]^-{1_q*\alpha'*1_q} && qpq'pq \ar[rr]^-{\mu*1_{q'}*\tau} && q' \\
\psi:q' \ar[rr]^-{\beta'} && q'pq' \ar[rr]^-{1_{q'}*\alpha*1_{q'}} && q'pqpq' \ar[rr]^-{\mu'*1_{q}*\tau'} && q}
\]
are mutual inverses, therefore providing the required isomorphism. 
Indeed, consider the following identities, where the first one is due to the
interchange law, and the unnamed ones are due to firmness of the Morita context: 
\begin{eqnarray*}
\psi\circ \phi 
&=& (\mu*\mu'*1_{q}*\tau'*\tau) \circ (1_{qpq'}*\alpha*1_{q'pq}) \circ (1_{qp}*\beta'*1_{pq}) \circ (1_q*\alpha'*1_q) \circ \beta\\
&\stackrel{(\ref{it:1})}{=}& (\mu*\mu'*1_{q}*\tau'*\tau) \circ (1_{qpq'}*\alpha*1_{q'pq}) \circ (1_{qpq'}*\alpha'*1_{q}) \circ (1_q*\alpha'*1_q) \circ \beta\\
&\stackrel{(\ref{it:3})}{=}& (\mu*\mu'*1_{q}*\tau'*\tau) \circ (1_q*\alpha'*1_{qpq'pq}) \circ (1_{qpq'}*\alpha *1_q ) \circ (1_q*\alpha'*1_q) \circ \beta\\
&\stackrel{(\ref{it:3})}{=}& (\mu*\mu'*1_{q}*\tau'*\tau) \circ (1_q*\alpha'*1_{qpq'pq}) \circ (1_{q}*\alpha' * 1_{qpq}) \circ (1_q*\alpha*1_q) \circ \beta\\
&=& (\mu*\mu'*1_{q}*\tau)\circ (1_{q}*\alpha' *1_{qpq} ) \circ (1_q*\alpha*1_q) \circ \beta\\
&=& (\mu*1_q*\tau)\circ (1_q*\alpha*1_q) \circ \beta \\
&\stackrel{(\ref{it:1})}{=}& (\mu*1_q*\tau)\circ (\beta*1_{qp}) \circ \beta\\
&=& (1_q*\tau)\circ \beta \\
&=& 1_q
\end{eqnarray*}
In a very analogous way, $\phi\circ\psi=1_{q'}$ and hence $q$ and $q'$ are isomorphic.
\end{proof}

%


We now consider the particular case of an one-object bicategory, i.e. a monoidal category.
A firm Morita context $(\star,\star,X,Y,t_1,t_2)$ in a bicategory $\Kk$ with one object $\star$ and $\Vv=\Kk(\star,\star)$ gives rise to the following definition.

\begin{definition}\label{def:oplaxinverse}
An {\em oplax inverse} for an object $X$ in a monoidal category $\Vv$
is an object $Y$, together with morphisms $\tau_1\colon X\ot Y\to I$ and $\tau_2\colon Y\ot X\to I$ such
that $1_Y\ot \tau_1=\tau_2\ot 1_Y$ and $\tau_1\ot1_X=1_X\ot \tau_2$, and these are invertible morphisms in $\Vv$.
\end{definition}

As a result of \cref{uniquenesslaxinverse}, such a notion is unique up to isomorphism. Moreover, as a special case of the fact that pseudofunctors between bicategories preserve Morita contexts \cite[Prop.~1.10]{ElK}, we obtain that a strong monoidal functor between monoidal categories preserve oplax inverses.

In many immediate examples, oplax inverses are trivial. Indeed, in the monoidal category $\Vv=\mathsf{Vect}_k$ of vector spaces over a field, a simple dimension argument shows that an object has an oplax inverse if and only if it is isomorphic to the monoidal unit $k$. Nevertheless, as will be evident in the next sections, oplax inverses provide the required structure to capture the proper notion of antipodes for oplax bialgebras.

\subsection{Oplax Hopf monoids}

Let us first recall the {\em convolution} monoidal structure
\cite[Prop.4]{Monoidalbicatshopfalgebroids} on a hom-category between a pseudomonoid
and pseudocomonoid in a monoidal bicategory,
that naturally generalizes the classical convolution for (co)monoids
in monoidal categories. 

\begin{lemma}
Let $(M,\mlt,\uni)$ be a pseudomonoid and $(C,\lcomlt,\lcouni)$ a pseudocomonoid in $\Kk$.
The hom-category $\Kk(C,M)$ is a monoidal category, with tensor product $\odot$ defined on $1$-cells $f,g\colon C\to M$ by
\begin{equation*}
f\odot g :=
\xymatrix@!C{
C \ar[r]^-\lcomlt & C\ot C \ar[r]^-{f\ot g} & M\ot M \ar[r]^-\mlt & M
}
\end{equation*}
and on two-cells $\alpha\colon f\Rightarrow f'$, $\beta\colon g\Rightarrow g'$ by 
\begin{equation*}
\alpha\odot\beta:=\xymatrix{
C \ar[r]^-\lcomlt & 
C\ot C \ar@/^1em/[rr]^-{f\ot g} \ar@/_1em/[rr]_-{f'\ot g'} 
\ar@{}[rr]|-{\Downarrow\alpha\otimes\beta} && 
M\ot M \ar[r]^-\mlt & 
M \rlap{\ .}
}
\end{equation*}
The monoidal unit is given by $I^\odot=\uni\circ\lcouni\colon C\to I\to M$. 
\end{lemma}

If $M$ is an oplax bimonoid in $\Kk$ as in \cref{oplaxbimonoid}, the above lemma assures that
$\Kk(M,M)$ has a monoidal structure given by convolution. This is clearly different to the standard monoidal structure of every endo-hom category
$(\Kk(M,M),\circ,1_M)$ via horizontal composition, for an arbitrary (not necessarily monoidal) bicategory $\Kk$; for a meaningful relation between these two monoidal structures, see \cref{lem:monoidaliso}. Another relation, albeit not relevant to our current development, is that if $M$ is a map monoidale these form a duoidal structure on $\Kk(M,M)$, see e.g.\ \cite[Section 3.3]{BohmLack}.

Since in the classical case, the antipode of a Hopf algebra is a convolution inverse to the identity, the following definition sets the respective oplax version in an analogous setting.

\begin{definition}\label{def:oplaxantipode}
An \emph{oplax antipode} for an oplax bimonoid $M$ in a braided monoidal bicategory $\Kk$ is an oplax inverse of $1_M$ in the convolution monoidal 
category $\Kk(M,M)$.
\end{definition}

To unpack the above definition, consider an oplax bimonoid $(M,\mlt,\uni,\lcomlt,\lcouni)$ in $\Kk$. \cref{def:oplaxinverse} implies that an oplax antipode is a 1-cell $\atpd\colon M\to M$ along with 2-cells $\tau_1\colon 1_M\odot\atpd \Rightarrow I^\odot$
and $\tau_2\colon\atpd\odot1_M \Rightarrow I^\odot$ as in
\begin{equation}\label{oplaxatpd}
\begin{tikzcd}
 & M\otimes M \ar[rr,"1\ot\atpd"] & \phantom{A}\ar[d,phantom,"\Downarrow{\scriptstyle\tau_1}"] & M\ot M \ar[dr,"\mlt"] & \\
 M \ar[rr,"\lcouni"] \ar[ur, "\lcomlt"] \ar[dr, "\lcomlt"'] & & I \ar[rr,"\uni"] & & M \\
 & M\ot M \ar[rr,"\atpd\ot 1"'] & \phantom{A}\ar[u,phantom,"\Uparrow{\scriptstyle\tau_2}"] & M\ot M \ar[ur,"\mlt"'] &
\end{tikzcd}
\end{equation} 
such that $1_\atpd\odot\tau_1=\tau_2\odot1_\atpd$ and $\tau_1\odot1_{1_M}=1_{1_M}\odot\tau_2$, namely
\begin{displaymath}
\adjustbox{scale=.85}{
\begin{tikzcd}[column sep=.2in,row sep=.2in]
 &&MMM\ar[r,"\atpd1\atpd"] & MMM\ar[dr,"1\mlt"] && \\
 M\ar[r,"\lcomlt"]\ar[rrd,bend right=10,"1"',"\cong"] &
 MM\ar[rrr,phantom,"\Downarrow{\scriptstyle{1_s\ot\tau_1}}"description]\ar[dr,"1\lcouni"description]\ar[ur,"1\lcomlt"] &&& MM\ar[r,"\mlt"] & M \\
&& M\ar[r,"\atpd"'] & M \ar[ur,"1\uni"description]\ar[urr,bend right=10,"1"',"\cong"] && 
\end{tikzcd}=
\begin{tikzcd}[column sep=.2in,row sep=.2in]
 &&MMM\ar[r,"\atpd1\atpd"] & MMM\ar[dr,"\mlt1"]\ar[rrd,bend left=10,"\mlt\circ1\mlt","\cong"'] && \\
 M\ar[r,"\lcomlt"]\ar[rrd,bend right=10,"1"',"\cong"]\ar[rru,bend left=10,"1\lcomlt\circ\lcomlt","\cong"'] &
 MM\ar[rrr,phantom,"\Downarrow{\scriptstyle{\tau_2\ot1_s}}"description]\ar[dr,"\lcouni1"description]\ar[ur,"\lcomlt1"] &&& MM\ar[r,"\mlt"] & M \\
&& M\ar[r,"\atpd"'] & M \ar[ur,"\uni1"description]\ar[urr,bend right=10,"1"',"\cong"] && 
\end{tikzcd}}
\end{displaymath}
\begin{equation}\label{eq:oplaxatpdax}
\adjustbox{scale=.85}{
\begin{tikzcd}[column sep=.2in,row sep=.2in]
 &&MMM\ar[r,"1\atpd1"] & MMM\ar[dr,"\mlt1"]\ar[rrd,bend left=10,"\mlt\circ1\mlt","\cong"'] && \\
 M\ar[r,"\lcomlt"]\ar[rrd,bend right=10,"1"',"\cong"]\ar[rru,bend left=10,"1\lcomlt\circ\lcomlt","\cong"'] &
 MM\ar[rrr,phantom,"\Downarrow{\scriptstyle{\tau_1\ot1_{1_M}}}"description]\ar[dr,"\lcouni1"description]\ar[ur,"\lcomlt1"] &&& MM\ar[r,"\mlt"] & M \\
&& M\ar[r,"1"'] & M \ar[ur,"\uni1"description]\ar[urr,bend right=10,"1"',"\cong"] && 
\end{tikzcd}=
\begin{tikzcd}[column sep=.2in,row sep=.2in]
 &&MMM\ar[r,"1\atpd1"] & MMM\ar[dr,"1\mlt"] && \\
 M\ar[r,"\lcomlt"]\ar[rrd,bend right=10,"1"',"\cong"] &
 MM\ar[rrr,phantom,"\Downarrow{\scriptstyle{1_{1_M}\ot\tau_2}}"description]\ar[dr,"1\lcouni"description]\ar[ur,"1\lcomlt"] &&& MM\ar[r,"\mlt"] & M \\
&& M\ar[r,"1"'] & M \ar[ur,"1\uni"description]\ar[urr,bend right=10,"1"',"\cong"] && 
\end{tikzcd}}
\end{equation}
and all 2-cells $1_\atpd\odot\tau_1$, $\tau_2\odot1_\atpd$, $\tau_1\odot1_{1_M}$ and $1_{1_M}\odot\tau_2$
are invertible.

The above says that $\atpd\odot 1_M\odot \atpd\cong\atpd$ and $1_M\odot\atpd\odot 1_M\cong 1_M$ via specific 2-isomorphisms. 

\begin{definition}\label{oplaxHopfmonoid}
An \emph{oplax Hopf monoid} in a monoidal bicategory $\Kk$ is an oplax bimonoid with an oplax antipode.
A \emph{morphism} of oplax Hopf monoids is an oplax bimonoid morphism which preserves the antipode.
\end{definition}

Notice that from \cref{uniquenesslaxinverse} it follows that the antipode is essentially unique when it exists; hence `being Hopf' is a property on 
oplax bimonoids and not structure. We obtain a (bi)category $\sf{OplHopf}(\Kk)$ of oplax Hopf monoids and morphisms between them,
with a faithful forgetful functor to $\OplBimon(\Kk)$ of oplax bimonoids. Moreover, \cref{prop:Fpreserves} extends to the case of oplax Hopf monoids as follows.


\begin{proposition}\label{prop:Fpreservesatpd}
If $\ca{F}\colon\ca{K}\to\ca{L}$ is a braided monoidal pseudofunctor and $M$ is an oplax Hopf monoid in $\ca{K}$ with antipode $\atpd$, then $\ca{F}M$ is an oplax Hopf monoid in $\ca{L}$ with antipode $\ca{F}s$.  
\end{proposition}

\begin{proof}
Suppose the monoidal pseudofunctor $\ca{F}$ has structure maps the pseudonatural equivalences $\phi_{X,Y}\colon\ca{F}X\ot\ca{F}Y\to\ca{F}(X\ot Y)$ 
and 
$\psi_{X,Y}\colon \ca{F}(X\ot Y)\to \ca{F}X\ot \ca{F}Y$ with $\phi\psi\cong\psi\phi\cong\id$, and similarly for $\phi_0$, $\psi_0$. The 1-cell $\ca{F}\atpd\colon\ca{F}M\to\ca{F}M$ along 
with the composite 2-cells formed from $\ca{F}\tau_1$ as in
\begin{displaymath}
 \begin{tikzcd}[column sep=.3in]
&\phantom{A}\ar[dr,phantom,near end,"\scriptstyle\cong"]& \ca{F}M\ot\ca{F}M\ar[drr,phantom,"\scriptstyle\cong"]\ar[d,"\phi_{M,M}"]\ar[rr,"1\ot 
\ca{F}s"] && \ca{F}M\ot\ca{F}M\ar[dr,"\phi_{M,M}"]\ar[d,"\phi_{M,M}"'] && \\
\ca{F}M\ar[rrrrrr,phantom,bend right=5,"\scriptstyle\cong"]\ar[rrrrrr,bend right=10,"\ca{F}(\mlt\circ(1\ot\atpd)\circ\lcomlt)"description]\ar[dd,"1"']\ar[r,"\ca{F}\lcomlt"]& \ca{F}(M\ot M)\ar[ur,"\psi_{M,M}"]\ar[r,"1"description] & \ca{F}(M\ot M)\ar[rr,"\ca{F}(1\ot\atpd)"description] 
&\phantom{A}\ar[dd,phantom,"\scriptstyle\Downarrow\ca{F}\tau_1"] &\ca{F}(M\ot 
M)\ar[r,"1"description] & 
\ca{F}(M\ot M)\ar[r,"\ca{F}\mlt"] & \ca{F}M\ar[dd,"1"]\\
\hole \\
\ca{F}M\ar[rrrrrr,phantom,bend left=5,"\scriptstyle\cong"]\ar[rrrrrr,bend left=10,"\ca{F}(\uni\circ\lcouni)"description]\ar[drr,"\ca{F}\lcouni"']\ar[rrr,"\ca{F}\lcouni"'] &&& 
\ca{F}I\ar[dr,"1"description]\ar[d,phantom,"\scriptstyle\cong"]\ar[rrr,"\ca{F}\uni"'] &&&\ca{F}M \\
&& \ca{F}I\ar[ur,"1"description]\ar[r,"\psi_0"'] & I\ar[r,"\phi_0"'] & \ca{F}I\ar[urr,"\ca{F}\uni"'] &&
 \end{tikzcd}
\end{displaymath}
and similarly the one formed from $\ca{F}\tau_2$ constitute an oplax antipode for the oplax bimonoid $\ca{F}M$. This can be also deduced from the fact that $\ca{F}_{M,M}\colon\ca{K}(M,M)\to\ca{K}(\ca{F}M,\ca{F}M)$ is a strong $\odot$-monoidal functor with structure maps essentially those on the top and the bottom of the diagram, therefore preserves oplax inverses.
\end{proof}

Finally, we can also express the definition of an oplax Hopf monoid in terms of fusion morphisms, similarly to the classical case recalled in \cref{subsec:bimonoids}. For an oplax bimonoid $(M,\mlt,\uni,\lcomlt,\lcouni)$,
consider $M\ot M$ as a (strict) left $M$-module and (strict) right $M$-comodule with actions given by
\begin{align*}
\rho=& M\ot M\ot M\xrightarrow{\mlt\ot1_M} M\ot M\\
\chi=& M\ot M\xrightarrow{1_M\ot \lcomlt}M\ot M\ot M
\end{align*}
If we denote by ${_M\Kk^M}(M\ot M,M\ot M)$ the monoidal subcategory of the endo-hom category
$(\Kk(M\ot M,M\ot M),\circ,1_{M\ot M})$ that consists of strict
$M$-module and $M$-comodule morphisms, we have the following result.

\begin{lemma}\label{lem:monoidaliso}
For $M$ an oplax bimonoid in a monoidal bicategory $\Kk$, there exists an isomorphism of monoidal categories
$$({_M\Kk^M}(M\ot M,M\ot M),\circ,1_{M\ot M}) \cong (\Kk(M,M),\odot,I^\odot).$$
\end{lemma}

\begin{proof}
Define $F\colon\Kk(M,M)\to_M\Kk^M(M\ot M,M\ot M)$ mapping any $f:M\to M$ to
\begin{displaymath}
F(f)=M\ot M\xrightarrow{1_M\ot\lcomlt}M\ot M\ot M\xrightarrow{1_M\ot f\ot 1_M}
M\ot M \ot M\xrightarrow{\mlt\ot 1_M}M\ot M.
\end{displaymath}
One easily verifies that this is left $M$-linear and right $M$-colinear.
Conversely, given any left $M$-module and right $M$-comodule map
$g\colon M\ot M\to M\ot M$, 
define
\begin{displaymath}
G(g)=M\xrightarrow{\uni\ot1_M}M\ot M\xrightarrow{g}M\ot M\xrightarrow{1_M\ot\lcouni}M 
\end{displaymath}
The bijection is established using the strict (co)unity conditions
for $M$, 
and furthermore it can be checked that  $F(f\odot f')=F(f)\circ F(f')$, hence this is a monoidal isomorphism.
\end{proof}

Under the above isomorphism, the identity $1_M\in\Kk(M,M)$ corresponds
to the so-called {\em fusion} $1$-cell 
\begin{equation}\label{fusion1cell}
M\ot M\xrightarrow{1_M\ot\lcomlt}M\ot M\ot M\xrightarrow{\mlt\ot 1_M}M\ot M
\end{equation}
and the following can be easily deduced from \cref{def:oplaxantipode}.

\begin{proposition}\label{fusion}
An oplax bimonoid is an oplax Hopf monoid if and only if the fusion 1-cell $(\mlt\ot1_M)\circ(1_M\ot\lcomlt)$ has an oplax inverse in ${_M\Kk^M}(M\ot M,M\ot M)$. 
\end{proposition}

%% file: 4-X2Span.tex
\section{The object \texorpdfstring{$X^2$}{X} in \texorpdfstring{$\Span$}{Span}}
\label{sec:X2structures}

In this section, we describe various algebraic properties of $X^2$ as an object in the symmetric monoidal bicategory $\Span$;
these results will be necessary to capture the structure of Hopf and Frobenius $\Vv$-categories in
\cref{sec:FrobeniusHopfVCatsSpanV,FrobeniusVcats}, where $X$ serves
as their set of objects. Although the results of this section are valid for arbitrary groupoids as discussed in \cref{rem:groupoids},
we provide explicit proofs only for $X^2$.

\subsection{Spans}

Recall that in any category $\Cc$, a span $X\tickar Y$ is a diagram of the form $\spn{X}{S}{Y.}{f}{g}$
A morphism between spans is a map $u: S \to T$ making both left and right triangles commute:
\begin{equation}\label{mapofspans}
\begin{tikzcd}[row sep=.25in]
& S \ar[d, "u"description] \ar[ddl,bend right,"f"'] \ar[ddr,bend left,"g"]&\\
& T \ar[dl,"h"description]\ar[dr,"k"description]&\\
X & & Y
\end{tikzcd}
\end{equation}
If $\Cc$ has pullbacks, the composite $X\tickar Y\tickar Z$
is $\spn{X}{S\times_Y R}{Z}{}{}$ in the usual way. Since pullbacks are unique only up to isomorphism,
the above data forms a bicategory.
As explained in \cref{sec:pseudomon}, using coherence we can work with $\Span$ as if it was a strict $2$-category; on the other hand, taking 
isomorphism classes of spans we can work with a 1-category also denoted $\Span$, and it should be clear from the context which one we mean.
From this point on, we restrict ourselves to $\Span$ for $\Cc = \Set$, where spans can also be expressed as $(f,g)\colon S\to X\times Y$
as is the case in any category with products. 

\begin{remark}\label{rem:uniquespanmap}
An evident observation, which will nevertheless be very useful in checking various axioms in the following sections, is that there is at most one 
span morphism with target a span with a monic leg. Explicitly, if there were two maps of spans $u, u'\colon S\to T$ as in \cref{mapofspans} and for 
example the leg $h\colon T\to X$ is monic, then $u=u'$ since $h\circ u=f=h\circ u'$. More generally, a similar argument shows that there is at 
most one span morphism with target a span with jointly monic legs. 
As a result, in later axioms that involve equalities of pasted composites of 2-cells in $\Span$, whenever the target span has that property, the axiom 
verification is straightforward since there can only be one such 2-cell. 
\end{remark}

The bicategory $(\Span,\times,1=\{\star\})$ is symmetric monoidal,
with $\braid\colon X\times Y\tickar Y\times X$
given by the span $\spn{X\times Y}{X\times Y}{Y\times X}{\id}{\mathrm{sw}}$ where $\mathrm{sw}$
is the switch map $(x,y) \mapsto (y,x)$ in $\Set$. If we view $\Set$ as a 2-category with trivial 2-cells, there is a faithful (strict) monoidal (strict) functor
\begin{equation}\label{embeddingSet}
\Set \to \Span 
\end{equation} 
which acts as the identity on objects, and maps a function $f\colon X \to Y$
to the span $\spn{X}{X}{Y}{\id}{f}$. Indeed $g\circ f\colon X\to Y\to Z$ is mapped to the span $\spn{X}{X}{Z}{\id}{g\circ f}$ by choosing 
pullbacks along the identity to be identity in $\Set$. In a similar way there is an embedding from $\Set^\op$ to $\Span$.

\subsection{Trivial monoid and comonoid structures}\label{sec:trivialmon}

Since every set $X$ has a canonical comonoid
structure in $\Set$ given by the diagonal map $\diag \colon X \to X\times X$ and unique map $! \colon X \to 1$,
the monoidal embedding $\Set\to\Span$ as in \cref{embeddingSet} 
yields a comonoid structure on $X$ in $\Span$ with $(\id,\diag)\colon X\tickar X\times X$ as the comultiplication
and $(\id,!)\colon X\tickar 1$ as the counit; we refer to this as the \emph{trivial comonoid structure} on $X$ in $\Span$.
Similarly, the contravariant embedding yields a monoid structure on $X$ which is the reverse of the comonoid structure and is called the
\emph{trivial monoid structure} on $X$. Notice that these structures are \emph{strict} rather than pseudo in the monoidal bicategory of spans, namely 
the (co)associator and (co)unitors are identities and the axioms \cref{eq:psmoneqs} are trivially satisfied (since for example $c_{m,m}$ in this 
case is the identity).



In particular, on an object of the form $X^2$ in $\Span$, we will denote the trivial comonoid structure as $(X^2,\zeta,\nu)$ where
\begin{equation}\label{trivialcomonoidstr}
\begin{tikzcd}
& X^2\ar[dr,"\diag_{X^2}"]\ar[dl,"\id"'] & \\
X^2\ar[rr,tick,"\zeta"'] && X^4
\end{tikzcd}\qquad
\begin{tikzcd}
& X^2\ar[dr,"!"]\ar[dl,"\id"'] & \\
X^2\ar[rr,tick,"\nu"'] && 1
\end{tikzcd}
\end{equation}

\begin{remark}
Notice that while $\Comon(\Set) \cong \Set$ as is the case in any cartesian monoidal
category, here $\Comon(\Span)\ncong\Span$ since the latter is no longer cartesian: the categorical product in this category is given by the disjoint union.
Hence our so-called `trivial' comonoid structure above is not so in that usual way.
\end{remark}

\subsection{Groupoid pseudomonoid and pseudocomonoid structures}\label{sec:groupmoncomon}

For any set $X$, one can form the codiscrete category whose objects are elements of $X$
and where there is a single arrow $(x,y) \in X^2$ between any pair of objects $x,y \in X$. The identity arrows are the pairs $(x,x)$ and composition
is $((x,y),(y,z)) \mapsto (x,z)$. This category is in fact a groupoid: inverses are computed by  $(x,y) \mapsto (y,x)$. This groupoid gives rise to a 
pseudomonoid structure on $X^2$ in $\Span$
with the following multiplication and unit
\begin{equation}\label{groupoidmonoidstr}
\cd[]{
 & X^3 \ar[dl]_{1\times\diag\times1} \ar[dr]^{1\times!\times1=\pi_{13}} & \\
 X^4\ar[rr]|-{\object@{|}}_-\mu & & X^2
}
\quad \quad 
\cd[]{
 & X \ar[dl]_{!} \ar[dr]^{\diag} & \\
 1\ar[rr]|-{\object@{|}}_-\eta & & X^2
}
\end{equation}
We call this the \emph{groupoid pseudomonoid structure} on $X^2$.
The reverse structure $u=(\pi_{13},1{\times}\diag{\times}1)\colon X^3\to X^2{\times}X^4$
and $v=(\diag,!)\colon X\to X^2{\times}1$ is called the \emph{groupoid pseudocomonoid structure} on $X^2$.

Explicitly, the coassociator $\beta\colon(1\times u)\circ u\cong (u\times1)\circ u$ and counitors $t\colon (v\times1)\circ u\cong1$, $s\colon 
(1\times 
v)\circ u\cong 1$ are formed between
\begin{equation}\label{eq:coassociator}
 \begin{tikzcd}[column sep={.7in,between origins},row sep=.2in]
  && P\ar[dd,phantom, near start,"\upback"description]\ar[dl]\ar[dr]\ar[dddd,dotted,bend left] && \\
  & X^3\ar[dl,"\pi_{13}"']\ar[dr,"1\Delta1"] && X^2\times X^3\ar[dl,"11\pi_{13}"']\ar[dr,"111\Delta1"] & \\
  X^2\ar[rr,tick,"u"] && X^2\times X^2\ar[rr,shift left,tick,"1\times u"]\ar[rr,shift right,tick,"u\times1"'] && X^2\times X^2\times X^2 \\
  & X^3\ar[ul,"\pi_{13}"]\ar[ur,"1\Delta1"'] && X^3\times X^2\ar[ul,"\pi_{13}11"]\ar[ur,"1\Delta111"'] & \\
  && Q\ar[uu,phantom, near start,"\dpback"description]\ar[ul]\ar[ur] &&
 \end{tikzcd}
\end{equation}

\begin{displaymath}
 \begin{tikzcd}[column sep={.5in,between origins},row sep=.2in]
&& A\ar[dd,phantom, near start,"\upback"description]\ar[dl]\ar[dr]\ar[ddd,dotted,bend left] && \\
& X^3\ar[dl,"\pi_{13}"']\ar[dr,"1\Delta1"] && X^2\times X\ar[dl,"11\Delta"']\ar[dr,"11!"] & \\
X^2\ar[rrrr,bend right=8,tick,"\id_{X^2}"']\ar[rr,tick,"u"] && X^2\times X^2\ar[rr,tick,"1\times v"] && X^2 \\
&& X^2\ar[ull,"1"]\ar[urr,"1"'] &&
\end{tikzcd}\qquad
 \begin{tikzcd}[column sep={.5in,between origins},row sep=.2in]
&& B\ar[dd,phantom, near start,"\upback"description]\ar[dl]\ar[dr]\ar[ddd,dotted,bend left] && \\
& X^3\ar[dl,"\pi_{13}"']\ar[dr,"1\Delta1"] && X^2\times X\ar[dl,"\Delta11"']\ar[dr,"!11"] & \\
X^2\ar[rrrr,bend right=8,tick,"\id_{X^2}"']\ar[rr,tick,"u"] && X^2\times X^2\ar[rr,tick,"v\times 1"] && X^2 \\
&& X^2\ar[ull,"1"]\ar[urr,"1"'] &&
\end{tikzcd}
\end{displaymath}
where $P$ and $Q$ are isomorphic to $X^4$, and $A$, $B$ are isomorphic to $X^2$. The precise span isomorphisms can be computed in a clear but tedious 
way: for example, taking the usual choice of pullbacks in $\Set$, we can compute that
\begin{align*}
X^3{\times}X^2{\times}X^3\supseteq P &{=}\{((x_1,x_2,x_3),(x_4,x_5),(x_6,x_7,x_8))\;|\;(x_1,x_2,x_2,x_3){=}(x_4,x_5,x_6,x_8)\}\cong 
\{(x_1,x_2,x_7,x_3)\} \\
X^3{\times}X^3{\times}X^2\supseteq Q &{=}\{((y_1,y_2,y_3),(y_4,y_5,y_6),(y_7,y_8))\;|\;(y_1,y_2,y_2,y_3){=}(y_4,y_6,y_7,y_8)\}\cong 
\{(y_1,y_5,y_2,y_3)\}
\end{align*}
and the precise isomorphism that commutes with the legs can be written according to the above 
representation, namely $x_1\mapsto y_1$, $x_2\mapsto y_5$, $x_7\mapsto y_2$ and $x_3\mapsto y_3$.

Notice that by \cref{rem:uniquespanmap}, the above structure span isomorphisms are unique: the counitors involve the identity, and the coassociator 
has monic legs on the right side. For the same reasons, the pseudocomonoid axioms -- dual to \cref{eq:psmoneqs} -- hold: both span 2-isomorphisms 
involve a span with a monic leg, therefore they must be equal.

\subsection{Oplax bimonoid and oplax Hopf monoid structures}\label{oplaxHopfX2}

The following result establishes the oplax bimonoid structure of $X^2$ in the symmetric monoidal bicategory $\Span$.

\begin{proposition}\label{X2oplaxbimonoid}\hfill
\begin{enumerate}[(i)]
\item For any set $X$, the groupoid pseudomonoid $(\lmlt,\luni)$ \cref{groupoidmonoidstr} 
and trivial comonoid $(\zeta,\nu)$ \cref{trivialcomonoidstr} structures
on $X^2$ make it an oplax bimonoid in $\Span$.
\item For any set $X$, the trivial monoid and groupoid pseudocomonoid structures on $X^2$ make it an
oplax bimonoid in $\Span$.
\end{enumerate}
\end{proposition}
\begin{proof}
Since $\Span$ is a symmetric monoidal bicategory and $X^2$ is a strict comonoid with the trivial comonoid structure,
$X^2\times X^2$ is a pseudocomonoid with comultiplication and counit similarly to \cref{monoidalpseudomonoids}
\begin{displaymath}
\begin{tikzcd}[column sep=.6in]
 X^2\times X^2\ar[r,tick,"\zeta\times\zeta"] & X^2\times X^2\times X^2\times X^2\ar[r,tick,"1\times\sigma_{X^2,X^2}\times1"] & X^8, & X^2\times 
X^2\ar[r,tick,"\nu\times\nu"] & 1
\end{tikzcd}
\end{displaymath}
The data for the oplax bimonoid structure (\cref{oplaxbimonoid}) are 
span morphisms ${\theta},{\theta_0},{\chi},{\chi_0}$ defined below.
\begin{equation}\label{phi12}
\cd[@=1.2em]{
& & & X^3 \ar[dl]|-{\id} \ar[dr]|-{\pi_{13}} \ar@{.>}@/^1em/[dddddd]|-{{\theta}} & & & \\ 
& & X^3 \ar[ddll]|-{1\diag1} \ar[dr]|-{\pi_{13}} & & X^2 \ar[dl]|-{\id} \ar[ddrr]|-{\diag} \\ 
& & & X^2 & & & \\ 
X^4 & & X^8 & & X^8 & & X^4 \\ 
 & X^4 \ar[ul]|-{\id} \ar[ur]|-{\diag^2} & & X^8 \ar[ul]|-{\id} \ar[ur]|-{11\mathrm{sw} 11} & & X^6 \ar[ul]|{1\diag 11\diag 1} \ar[ur]|-{\pi_{1346}} &\\ 
 & & X^4 \ar[ul]|-{\id} \ar[ur]|-{\diag^2} & & & &\\
& & & P \ar[ul]|-{1\diag 1} \ar[uurr]^{\diag}
}
\quad
\cd[@!R@=1.2em]{
& & X \ar[dl]_{\id} \ar[dr]^{\diag} \ar@{.>}@/^1em/[dddd]|-(.65){{\theta_0} = \diag} & & \\
& X \ar[dl]_{!} \ar[dr]^{\diag} & & 
X^2 \ar[dl]_{\id} \ar[dr]^{\diag} & \\
1 & & X^2 & & X^4 \\
& & & & \\
& & X^2 \ar[uull]^{!} \ar[uurr]_{\diag^2} & & 
}
\end{equation}
\begin{equation}\label{phi34}
\cd[@!R@=1em]{
& & X^3 \ar[dl]_{\id} \ar[dr]^{\pi_{13}} \ar@{.>}@/^1em/[dddd]|-(.65){{\chi} = 1\diag 1} & & \\
& X^3 \ar[dl]_{1\diag1} \ar[dr]^{\pi_{13}} & & 
X^2 \ar[dl]_{\id} \ar[dr]^{!} & \\
X^4 & & X^2 & & 1 \\
& & & & \\
& & X^4 \ar[uull]^{\id} \ar[uurr]_{!} & & 
}
\quad 
\cd[@!R@=1em]{
& & X \ar[dl]_{\id} \ar[dr]^{\diag} \ar@{.>}@/^1em/[dddd]|-(.65){{\chi_0} = !} & & \\
& X \ar[dl]_{!} \ar[dr]^{\diag} & & 
X^2 \ar[dl]_{\id} \ar[dr]^{!} & \\
1 & & X^2 & & 1 \\
& & & & \\
& & 1 \ar[uull]^{\id} \ar[uurr]_{!} & & 
}
\end{equation}
where $X^4\times X^6\supseteq P\cong X^3$ and $\theta$ is the (unique) span isomorphism that fits therein, like above.
The conditions listed in \cref{oplaxbimonoidaxioms} can all be verified again by \cref{rem:uniquespanmap} since all have a span with monic leg as 
target, therefore 
$(X^2,\mu,\eta,\zeta,\nu,\id,\Delta,1{\times}\Delta{\times}1,!)$ is an oplax bimonoid in the symmetric monoidal bicategory $\Span$.

Since $\Span$ is isomorphic to $\Span^\op$, the two statements are equivalent.
\end{proof}

\begin{remark}\label{oplaxbimonoidproperty}
In general, asking for a pseudomonoid and pseudocomonoid to be an oplax bimonoid amounts
to specifying structure, not a property; it requires the existence
of four 2-cells subject to axioms. 
However, there are cases like above where that structure is unique when it exists, according to \cref{rem:uniquespanmap}: in all cases 
\cref{phi12,phi34}, the target span has at least one monic leg.
\end{remark}

Finally, $X^2$ is an oplax Hopf monoid in $\Span$ as in \cref{oplaxHopfmonoid}.

\begin{proposition}\label{X2Hopf}
For any set $X$, the oplax bimonoid structure $(X^2,\mu,\eta,\zeta,\nu)$
uniquely extends to an oplax Hopf monoid structure on $X^2$.
\end{proposition}

\begin{proof}
The oplax antipode is given by the span $\atpd\colon\spn{X^2}{X^2}{X^2}{\id}{\textrm{sw}}$,
along with
two 2-cells $\tau_1,\tau_2$ \cref{oplaxatpd} in $\Span$, computed to be the upper and lower span morphisms
\begin{equation}\label{tauforX2}
\cd[@C+3em]{
 & X^2 \ar@/_1em/[dl]_{\id} \ar@/^1em/[dr]^{\diag\circ\pi_1} \ar[d]^-{\tau_1} & \\
 X^2 & X^3 \ar[l]^{\pi_{12}} \ar[r]_{\diag\circ\pi_3} & X^2 \\
 & X^2 \ar@/^1em/[ul]^{\id} \ar@/_1em/[ur]_{\diag\circ\pi_2} \ar[u]_{\tau_2} & 
 }
\end{equation}
with $\tau_1 \colon (a,b) \mapsto (a,b,a)$
and $\tau_2 \colon (a,b)\mapsto(a,b,b)$. These are monomorphisms,
and moreover according to \cref{oplaxbimonoidproperty} they are unique.
One can now verify that $s$ and $1_{X^2}$ by means of the 
 2-cells $\tau_1$ and $\tau_2$ are oplax inverses in the convolution category $\Span(X^2,X^2)$.
\end{proof}

\begin{remark}
We could alternatively consider the fusion 1-cell \cref{fusion1cell} for $(X^2,\mu,\eta,\zeta,\nu)$
\begin{displaymath}
\begin{tikzcd}[column sep=.45in]
c\colon X^4\ar[r,tick,"1_{X^2}\times\zeta"] & X^6\ar[r,tick,"\mu\times1_{X^2}"]& X^4  
\end{tikzcd}
\end{displaymath}
which is computed to be a span $X^4\leftarrow X^3\rightarrow X^4$ 
defined by $(a,b,b,c) \mapsfrom (a,b,c) \mapsto (a,c,b,c)$. 
As expected from \cref{fusion}, this is not invertible in $\Span$, but it allows an oplax inverse in ${_{X^2}}\Span^{X^2}(X^4,X^4)$ namely the following composite
\begin{displaymath}
\begin{tikzcd}[column sep=.7in]
\bar c\colon X^4\ar[r,tick,"1_{X^2}\times\zeta"] & X^6 \ar[r,tick,"1_{X^2}\times\sigma\times{1_{X^2}}"] & X^6 \ar[r,tick,"\lmlt\times 1_{X^2}"] & X^4
\end{tikzcd}
\end{displaymath}
Once we recognize that both legs of the span $c$ are monomorphisms, it is immediate that $\bar c$ and $c$ are oplax inverses. 
Notice that ${\bar c}$ given by $(a,b,c,b)\mapsfrom (a,b,c)\mapsto (a,c,c,b)$ is the reverse span of $c$ up to a switch between $b$ and $c$.

\end{remark} 

\subsection{Frobenius monoid structure}\label{FrobMonStr}

The following proposition exhibits how different combinations of the earlier monoid and comonoid structures give rise to Frobenius (pseudo)monoids in 
the symmetric monoidal bicategory $\Span$.

\begin{proposition}\label{X2Frobenius}\hfill
\begin{enumerate}[(i)]
\item For any set $X$, the trivial monoid $((\diag,\id),(!,\id))$ and its reverse trivial comonoid structures 
make $X$ into a Frobenius (strict) monoid in $\Span$; clearly also $X^2$ with \cref{trivialcomonoidstr} and its reverse.
\item For any set $X$, the groupoid monoid $(\mu,\eta)$ \cref{groupoidmonoidstr} and its reverse groupoid comonoid
structure $(u,v)$ on $X^2$ make it a Frobenius pseudomonoid in $\Span$.
\end{enumerate}
\end{proposition}
\begin{proof}
The first statement can be easily verified. For the second one, we need to provide the structure isomorphisms of \cref{pseudoFrobcond} and verify the 
conditions of \cref{sec:Frobpseudo}.
The middle composite is
\begin{displaymath}
\begin{tikzcd}[minimum width = width("$X^2 \times X^2$"), column sep = 0.3cm]
& & P\cong X^4 \ar[dr,dashed,"\pi_{134}"] \ar[dl,dashed,"\pi_{124}"'] \ar[dd,phantom, near start,"\upback"description] & & \\
& X^3\ar[dl,"1\times\diag\times1"']\ar[dr,"\pi_{13}"] & & X^3\ar[dl,"\pi_{13}"']\ar[dr,"1\times\diag\times1"] & \\
X^2\times X^2\ar[rr,tick,"\mu "] & & X^2\ar[rr,tick,"\delta "] & & X^2\times X^2
\end{tikzcd}
\end{displaymath}
which explicitly acts as $(a,b,b,d)\mapsfrom (a,b,c,d)\mapsto (a,c,c,d)$. 
The downside composite on the other hand is
\begin{displaymath}
\begin{tikzcd}[minimum width = width("$X^2 \otimes X^2 \times X^2$"), column sep = 0.15cm]
& & Q\cong X^4\ar[dr,dashed, near start,"1^2\times\diag\times1"]\ar[dl,dashed, near start,"1\times\diag\times 1^2"']
\ar[dd,phantom, near start,"\upback"description] & & \\
& X^2\times X^3\ar[dl, near start,"1^2\times\pi_{13}"']\ar[dr,"1^3\times\diag\times1" description]
& & X^3\times X^2\ar[dl,"1\times\diag\times1^3" description]\ar[dr, near start,"\pi_{13}\times1^2"] & \\
X^2\times X^2\ar[rr,tick,"1\times\delta "] & & X^2\times X^2\times X^2\ar[rr,tick,"\mu \times1"] & & X^2\times X^2
\end{tikzcd}
\end{displaymath}
So the span isomorphism is between $P\subseteq X^3\times X^3$ and $Q\subseteq X^2\times X^3 \times X^3\times 
X^3$, and the other isomorphism is formed 
similarly. Even though the spans involved do not have monic legs in this case as earlier, they are jointly monic; again by \cref{rem:uniquespanmap}, 
this fact renders the structure span isomorphisms unique. Finally, all conditions 
\cref{eq:rightpsmod12,eq:leftpsmod12,eq:leftpscomod12,eq:rightpscomod12} hold for the same reasons.
\end{proof}


\begin{remark}\label{rem:groupoids}
The above structures arise not just on the set $X^2$, but for
an arbitrary groupoid $\Gg=(\xymatrix{G_1 \ar@<.5ex>[r]|s \ar@<-.5ex>[r]|t & G_0})$ with set of objects $G_0$,
set of morphisms $G_1$, and $s,t$ the source and target maps.
Any $\Gg$ gives rise to a pseudomonoid (and a pseudocomonoid) in $\Span$ with multiplication and counit as follows
\begin{equation}\label{groupoidmonoidstr2}
\cd[]{
 & G_2 \ar[dl]_-{(s,t)} \ar[dr]^{\mlt} & \\
 G_1\times G_1\ar[rr]|-{\object@{|}} & & G_1
}
\quad \quad 
\cd[]{
 & G_0 \ar[dl]_{!} \ar[dr]^{e} & \\
 1\ar[rr]|-{\object@{|}} & & G_1
}
\end{equation}
where $G_2=G_1{_s\times_t}G_1$ are the composable morphisms of $\Gg$ and $e:G_0\to G_1$ gives the identity;
this determines the `groupoid' structures of \cref{sec:groupmoncomon}. Moreover,
the trivial monoid and groupoid pseudocomonoid structures on $\Gg_1$ make it an oplax bimonoid in $\Span$
that uniquely extends to an oplax Hopf monoid, analogously to \cref{X2oplaxbimonoid,X2Hopf}.
Finally, the groupoid pseudomonoid and pseudocomonoid structures on $\Gg_1$ make it a Frobenius pseudomonoid in $\Span$
as in \cref{X2Frobenius}.
\end{remark}

The results of this section can be summarised in the following table. The combinations
of structures in the columns give oplax Hopf monoid structures on $X^2$ in $\Span$,
and in the rows give Frobenius monoid structures.
\begin{table}[H]
\begin{tabular}{l|ll}
& Hopf & Hopf\\
\hline
Frob & Trivial monoid & Trivial comonoid\\
Frob & Groupoid comonoid & Groupoid monoid\\
\hline
\end{tabular}
\caption{}\label{Streetstable}
\end{table}
A similar observation was made by Street in \cite{streetFAMC}, where
the possible Frobenius and Hopf algebra structures for a group algebra $kG$ were described.

%
%

%% file: 5-SpanV_updated.tex
\section{The symmetric monoidal bicategory \texorpdfstring{$\Span|\Vv$}{Span|V}}\label{sec:SpanV}

In~\cite{Gabipolyads}, starting from a bicategory $\Kk$ a new bicategory $\Span|\Kk$ is constructed,
with the property that when $\Kk$ is monoidal so is $\Span|\Kk$. 
The two special cases that are addressed 
(due to the Hopf objects of interest) are $\Span|\Cat$ and $\Span|\Vv$, for $\Vv$ a braided monoidal category considered 
as a one-object monoidal bicategory.

In our context, we take $\Vv$ to be any monoidal category, this time viewed as a monoidal 2-category with trivial 2-cells;
this induces some sort of `level-shift' compared to B{\"o}hm's primary examples.
We give an explicit description of the monoidal bicategory $\Span|\Vv$~\cite[\S 2.1]{Gabipolyads} in our particular case of interest,
and in the following \cref{sec:FrobeniusHopfVCatsSpanV,FrobeniusVcats} we show how it serves as a common framework for Hopf and Frobenius $\Vv$-categories: they are expressed as oplax Hopf monoids and Frobenius pseudomonoids in $\Span|\Vv$ respectively.

\subsection{Monoidal bicategory structure}

The 0-cells in $\Span|\Vv$ are pairs $(X,M)$ where $X$ is a set and $M\colon X\to\Vv$ is a functor,
given by a family of objects $\{M_x\}_{x\in X}$ in $\Vv$; we use the shorthand notation $M_X$ for the pair $(X,M)$.
A 1-cell from $M_X$ to $N_Y$ in $\Span|\Vv$ consists of a span $\spn{X}{S}{Y}{f}{g}$
of sets, along with a natural transformation
 \begin{equation}\label{1cell2}
 \begin{tikzcd}[row sep=.07in]
 & |[alias=doma]|X\ar[dr,bend left=15,"M"] & \\
 S\ar[ur,bend left=15,"f"]\ar[dr,bend right=15,"g"'] & & \ca{V} \\
 & |[alias=coda]|Y\ar[ur,bend right=15,"N"']
 \arrow[Rightarrow,from=doma,to=coda,"\alpha",shorten >=.1in,shorten <=.1in]
 \end{tikzcd}
 \end{equation}
whose components are arrows $\alpha_s\colon M_{fs}\to N_{gs}$ in $\ca{V}$.
We use the shorthand $^f\alpha^g\colon M_X\to N_Y$ and specify $S$ separately when needed.
A 2-cell in $\Span|\Vv$ denoted by $\phi_u\colon{^f\alpha^g}\Rightarrow {^h\beta^k}\colon M_X\to N_Y$ is a map of spans $u\colon S\to T$,
i.e.\ $hu=f$ and $ku=g$, such that $\alpha$ factorizes through $\beta$ in the sense that $\alpha = \beta*1_u$:
\begin{equation}\label{2cellsSpanV}
\begin{tikzcd}[row sep=.11in]
& & |[alias=doma]|X\ar[dr,"M"] & \\
S\ar[r,"u"]\ar[urr,bend left=15,"f"]\ar[drr,bend right=15,"g"'] & T\ar[ur,"h"]\ar[dr,"k"'] & & \ca{V} \\
& & |[alias=coda]|Y\ar[ur,"N"']
\arrow[Rightarrow,from=doma,to=coda,"\beta",shorten >=.15in,shorten <=.15in]
\end{tikzcd}
=
 \begin{tikzcd}[row sep=.07in]
 & |[alias=doma]|X\ar[dr,bend left=15,"M"] & \\
 S\ar[ur,bend left=15,"f"]\ar[dr,bend right=15,"g"'] & & \ca{V} \\
 & |[alias=coda]|Y\ar[ur,bend right=15,"N"']
 \arrow[Rightarrow,from=doma,to=coda,"\alpha",shorten >=.1in,shorten <=.1in]
 \end{tikzcd}
 \end{equation}
Componentwise, this means that $\alpha_s=\beta_{u s}:M_{fs}\to N_{gs}$ in $\Vv$. Notice that a 2-cell is invertible if and only if the underlying map of spans is invertible.
Vertical composition 
is performed by the usual composition of maps of spans,
and horizontal composition $M_X\xrightarrow{\alpha}N_Y\xrightarrow{\gamma}O_Z$ of 1-cells is obtained via the pasted composite
of natural transformations
\begin{equation}\label{SpanVcomposition}
 \begin{tikzcd}[row sep=.05in,column sep=.5in]
 & & |[alias=doma]|X\ar[ddr,bend left,"M"description] & \\
 & S\ar[ur,"f"]\ar[dr,"g"'] & & \\
 S{\times_{Y}}Q\ar[ur,"\pi_1"]\ar[dr,"\pi_2"']\ar[rr,phantom, near start,"\ucorner"description] & &
 |[alias=coda]|Y\ar[r,"N"description]
 & \ca{V} \\
 & Q\ar[ur,"m"]\ar[dr,"n"'] & & \\
 & & |[alias=codb]|Z\ar[uur,bend right,"O"description] &
 \arrow[Rightarrow,from=doma,to=coda,"\alpha",shorten >=.1in,shorten <=.1in]
 \arrow[Rightarrow,from=coda,to=codb,"\gamma",shorten >=.1in,shorten <=.1in]
 \end{tikzcd}
\end{equation} 
given by $\Vv$-arrows $(\gamma\circ\alpha)_{(s,q)}\colon
M_{fs}\xrightarrow{\alpha_s}N_{gs}\xrightarrow{=}N_{mq}\xrightarrow{\beta_{q}}O_{nq}$ for any
$(s,q)\in S{\times_Y}Q$.
The identity $1_M\colon M_X\to M_X$ is given by the identity span and natural transformation.
Finally, horizontal composition of 2-cells in $\Span|\Vv$ follows from horizontal composition of 2-cells in $\Span$: in detail,
\begin{displaymath}
\phi*\psi=
\begin{tikzcd}[column sep=.6in]
M_X\ar[r,bend left,"\alpha"]\ar[r, bend right, "\beta"']\ar[r,phantom,"\Downarrow{\scriptstyle\phi}"description]
& N_Y\ar[r,bend left,"\gamma"]\ar[r, bend right, "\delta"']\ar[r,phantom,"\Downarrow{\scriptstyle\psi}"description] & O_Z
\end{tikzcd}
\end{displaymath}
is given by the map of spans $w$ below,
\begin{equation}\label{hcomp2cellsSpanV}
\begin{tikzcd}[row sep=.1in]
&&&& |[alias=a]|X\ar[rrdd,bend left,"M"] && \\
& S\ar[rrru,bend left=15,"f"description]\ar[rr,"u"description] &&
T\ar[ur,"h"description]\ar[dr,"k"description] &&& \\
{\scriptstyle\blacksquare}
\ar[rr,"w"description]\ar[ur,dashed]\ar[dr,dashed] && 
\bullet
\ar[ur,dashed]\ar[dr,dashed] && |[alias=b]|Y\ar[rr,"N"description]\ar[from=lllu,bend right=7,"g"description] && \Vv \\
& Q\ar[rrru,bend left=7,"m"description]\ar[rrrd,bend right=15,"n"description]\ar[rr,"v"description] &&
P\ar[ur,"r"description]\ar[dr,"l"description] &&& \\
&&&& |[alias=c]|Z\ar[rruu,bend right,"O"'] &&
\Twocell{a}{b}{1em}{"\beta"}
\Twocell{b}{c}{1em}{"\delta"}
\end{tikzcd}
\end{equation}
where ${\scriptstyle\blacksquare}$ is the pullback of $g$ and $m$ and $\bullet$ is the pullback of $r$ and $k$,
and the factorization of $\gamma\circ\alpha$ through $\delta\circ\beta$ is
\begin{displaymath}
\gamma_q\circ\alpha_s=(\gamma\circ\alpha)_{(s,q)}=(\delta\circ\beta)_{w(s,q)}=(\delta\circ\beta)_{(us,vq)}=\delta_{vq}\circ\beta_{us}
\end{displaymath}
for any $(s,q)$ in the square pullback, since $\alpha_s=\beta_{us}$ and
$\gamma_q=\delta_{vq}$ by $\phi$ and $\psi$.

With the above data, $\Span|\Vv$ becomes a bicategory. Notice that $\Span|(\Vv^\op) \cong (\Span|\Vv)^\op$.

\begin{remark}\label{2cellsequality}
By definition of a 2-cell in $\Span|\Vv$, namely a map of spans which satisfies a factorization property~\cref{2cellsSpanV}, it follows that two 2-cells are equal if and only if their underlying maps of spans in $\Span$ are equal; the accompanying factorizations are conditions that already hold. 
This fact is very useful when verifying axioms including 2-cells in subsequent sections.
\end{remark}

Since $\ca{V}$ is a monoidal 2-category (with trivial 2-cells), the bicategory $\Span|\Vv$ has an induced monoidal structure as follows.
On objects, the tensor product $M_X\otimes N_Y=(M\otimes N)_{X\times Y}$ is given by $(M\otimes N)_{x,y}=M_x\otimes N_y$ in $\ca{V}$, namely the composite 
$$
X\times Y\xrightarrow{M\times N}\ca{V}\times\ca{V}\xrightarrow{\otimes\times\otimes}\ca{V}.
$$
For 1-cells, $^f\alpha^g\otimes {^k\beta^l}$ is of the form
 \begin{equation}\label{tensor1cells}
 \begin{tikzcd}[row sep=.05in]
 & |[alias=doma]|X\times Z\ar[dr,bend left=15,"M\otimes Q"] & \\
 S\times T\ar[ur,bend left=15,"f\times k"]\ar[dr,bend right=15,"g\times l"'] & & \ca{V} \\
 & |[alias=coda]|Y\times W\ar[ur,bend right=15,"N\otimes P"']
 \arrow[Rightarrow,from=doma,to=coda,"\alpha\otimes\beta",shorten >=.1in,shorten <=.1in]
 \end{tikzcd} 
 \end{equation}
given by morphisms $M_{fs}\otimes Q_{kt}\xrightarrow{\alpha_s\otimes\beta_t}N_{gs}\otimes R_{lt}$ in $\Vv$.
For 2-cells, it is given by the product of the factorizing functions (maps of spans) in~\cref{2cellsSpanV}, and
the monoidal unit is $I_\ca{V}\colon\mathbf{1}\to\ca{V}$ that picks out the unit in $\ca{V}$. Finally, 
if $(\Vv,\otimes,I,\braid)$ is braided, then so is $\Span|\Vv$ via $\bar{\braid}_{MN}\colon M_X\ot N_Y\to N_Y\ot M_X$ given by
 \begin{displaymath}
 \begin{tikzcd}[row sep=.05in]
 & |[alias=doma]|X\times Y\ar[dr,bend left=15,"M\otimes N"] & \\
 X\times Y\ar[ur,bend left=15,"\id"] \ar[dr,bend right=15,"\mathrm{sw}"'] & & \ca{V} \\
 & |[alias=coda]|Y\times X\ar[ur,bend right=15,"N\otimes M"']
 \arrow[Rightarrow,from=doma,to=coda,"\bar\braid_{MN}",shorten >=.1in,shorten <=.1in]
 \end{tikzcd} 
 \end{displaymath}
with components $M_x \otimes N_y \xrightarrow{\braid_{M_x,N_y}} N_y \otimes M_x$ in $\Vv$.

\subsection{The underlying span functor}\label{sec:functorU}

There is an evident strict functor of bicategories $U\colon\Span|\Vv\to\Span$ that forgets the data associated with $\Vv$.
In more detail, $U$ maps a 0-cell
$A_X$ to the set $X$, a 1-cell $^f\alpha^g$ to the span $\spn{X}{S}{Y}{f}{g}$ and a 2-cell $\phi_u$ to 
the map of spans $u$; it is clear that the composition is preserved on the nose. 
This functor is strict monoidal, since
\begin{equation}\label{eq:Ustrict}
U(A_X)\otimes_\Span U(B_Y)=X\times Y=U(A_X\otimes_{\Span|\Vv} B_Y), \qquad U(I_1)=1
\end{equation}
and similarly for morphisms. Since $U$ is a (symmetric) strict monoidal strict functor of bicategories, \cref{prop:Fpreserves} ensures that 
any pseudomonoid or pseudocomonoid object in $\Span|\Vv$ has an underlying pseudo(co)monoid in $\Span$ 
which by the definitions is part of the structure; e.g. if $A_X$ is a pseudomonoid in $\Span|\Vv$, $X$ must be a pseudomonoid in $\Span$.

In fact, this strict functor of bicategories is a kind of \emph{2-opfibration}, see \cite{Mitch2fibrations}.
Briefly, therein the usual Grothendieck construction (see \cite[\S 8]{Handbook2}) extends
to give a correspondence between 2-functors $F\colon\caa{X}^\op\to{2\mhyphen\Cat}$
and 2-fibrations $\int F \to \caa{X}$, and a similar correspondence between trihomomorphisms from a bicategory into $\mathsf{Bicat}$ and fibrations of bicategories.
The covariant case is not explicitly described therein, but is easily obtained by modifying the contravariant construction.
In the case of a covariant $F$, the resulting functor is a 2-opfibration and the 2-category $\int F$ has objects $(X,A{\in}FX)$, morphisms
pairs
$(f\colon X{\to}Y\in\caa{X},\alpha\colon Ff(A){\to}B\in FY)$ and 2-cells $(f,\alpha)\Rightarrow(g,\beta)$
\begin{equation}\label{2cellsGroth}
\begin{tikzcd}[column sep=.8in]
X \ar[r,bend left=25,"f"]\ar[r,bend right=25,"g"']\ar[r,phantom,"\Downarrow{\scriptstyle\phi}"description] & Y
\end{tikzcd}\textrm{ in }\caa{X}\textrm{ such that }
\begin{tikzcd}[row sep=.2in]
Ff(A)\ar[dr,"\alpha"]\ar[d,"(F\alpha)_A"']\ar[dr,phantom,bend right=20,"="] & \\
Fg(A)\ar[r,"\beta"'] & B
\end{tikzcd} \textrm{ in } FY. 
\end{equation}
The work of Buckley also omits any version of this construction where $F$ is not as strict as it can be (a 2-functor from a 2-category into ${2\mhyphen\Cat}$) or as weak as it can be (a trihomomorphism from a bicategory into $\mathsf{Bicat}$), but those intermediate versions can be generated by adapting the construction accordingly. In our case, we need the covariant version where $\caa{X}$ is a bicategory and $F$ is a pseudofunctor into $2\mhyphen\Cat$, when $\int F$ becomes a bicategory and $\int F\to\caa{X}$ a strict functor.

\begin{theorem}\label{prop:fibration}
If $\Vv$ has colimits, $U\colon\Span|\Vv\to\Span$ is a 2-opfibration.
\end{theorem}

\begin{proof}
We build a pseudofunctor $\Span\to\Cat\hookrightarrow {2\mhyphen\Cat}$
and apply a modified Grothendieck construction from~\cite{Mitch2fibrations};
the resulting 2-opfibration is isomorphic to $U$.

Briefly recall that the left Kan extension of a functor $K\colon\ca{A}\to\ca{C}$ along some $F\colon\ca{A}\to\ca{B}$
is a functor $\Lan_FK\colon\ca{B}\to\ca{C}$ equipped with a natural transformation
$\eta_K\colon K\Rightarrow(\Lan_FK)\circ F$
such that $\Lan_F\dashv F_*=\mhyphen\circ F$ with unit $\eta$.
When $\Cc$ has colimits, the left Kan extension always exists and can be computed as the coend~\cite[~{}(4.25)]{Kelly}
$$\Lan_FK(b)\coloneqq\int^{a\in\ca{A}}\ca{B}(Fa,b)\cdot Ka$$
where $\cdot$ denotes the set-theoretic copower. 

Define a pseudofunctor $\Span\to\Cat$ which maps $X$ to $[X,\Vv]$,
a span $\spn{X}{S}{Y}{f}{g}$ to $\Lan_g(-\circ f)\colon[X,\Vv]\to[Y,\Vv]$
and a map of spans $u\colon(f,g)\to(h,k)$ \cref{mapofspans} to a natural transformation $\bar{u}$ whose components are
\begin{equation}\label{baru}
\bar{u}_A\colon\Lan_g(Af)=\Lan_g(Ahu)\cong\Lan_k(\Lan_u(Ahu))\xrightarrow{\Lan_k(\varepsilon_{Ah})}\Lan_k(Ah)
\end{equation}
for all $A\in[X,\Vv]$, where the isomorphism is a standard property for Kan extensions of composites~\cite[~{}(4.48)]{Kelly}.
This action on 1- and 2-cells defines a functor for all $X,Y$ in $\Span$.
For pseudofunctoriality, we need natural isomorphisms $\Lan_{\id}(-\circ\id)=\id_{[X,\Vv]}$
and $\Lan_k(\Lan_g(-\circ f)\circ h)\cong\Lan_{k\pi_2}(-\circ f\pi_1)$ for every pullback
of spans
\begin{displaymath}
\begin{tikzcd}[column sep = 0.3cm,row sep=.3cm]
& & Q \ar[dr,dashed,"\pi_2"] \ar[dl,dashed,"\pi_1"'] \ar[dd,phantom, near start,"\upback"description] & & \\
& S\ar[dl,"f"']\ar[dr,"g"] & & T\ar[dl,"h"']\ar[dr,"k"] & \\
X\ar[rr,tick] & & Y\ar[rr,tick] & & Z
\end{tikzcd}
\end{displaymath}
The identity is straightforward, while the second isomorphism arises as the composite
\begin{displaymath}
\begin{tikzcd}[row sep=.3in,column sep=1in]
{[X,\Vv]}\ar[r,"{f_*}"]\ar[dr,"{(f\pi_1)_*}"']
\ar[dr,phantom,bend left=10,"="description]
& {[S,\Vv]}\ar[r,shift left=2,"\Lan_g"]
\ar[r,phantom,"{\scriptstyle\bot}"description]\ar[d,"{(\pi_1)_*}"']\ar[dr,phantom,"\cong"description]
& {[Y,\Vv]}\ar[l,shift left=2,"{g_*}"]\ar[d,"{h_*}"] \\
& {[Q,\Vv]}\ar[r,shift left=2,"{\Lan_{\pi_2}}"]\ar[r,phantom,"{\scriptstyle\bot}"description]
\ar[dr,"{\Lan_{k\pi_2}}"']\ar[dr,phantom,bend left=15,near end,"\cong"description]&
{[T,\Vv]}\ar[l,shift left=2,"{(\pi_2)_*}"]\ar[d,"{\Lan_k}"] \\
&& {[Z,\Vv]}
\end{tikzcd}
\end{displaymath}
where the middle isomorphism is realised via the following coends computation,
for any $B\in[S,\ca{V}]$:
\begin{align*}
\Lan_gB(ht) &=\int^{s{\in}S}Y(gs,ht)\cdot Bs \cong \int^{\substack{s{\in}S\\ gs=ht}} Bs=
\sum_{gs=ht}Bs\qquad (Y,S\textrm{ are sets})\\
\Lan_{\pi_2}(B\pi_1)t &=\int^{(s,t'){\in}Q}T(\pi_2(s,t'),t)\cdot(B\pi_1)(s,t') \cong
\int^{\substack{(s,t'){\in}Q \\t=t' }}Bs=\sum_{gs=ht}Bs.
\end{align*}
The axioms for pseudofunctoriality can be shown to hold using the universal properties of these coends. 
If we compose this functor with the usual 2-functor $\Cat \to 2\mhyphen\Cat$ and apply the Grothendieck construction, we obtain the bicategory whose objects are $(X,A\in[X,\Vv])$, 1-cells are pairs
\begin{equation}\label{Groth1cell}
\begin{cases}
\Lan_g(Af)\xrightarrow{\alpha}B & \textrm{in } [Y,\Vv] \\
\spn{X}{S}{Y}{f}{g}  & \textrm{in } \Span
\end{cases}
\end{equation}
and 2-cells are,
for $u$ as in \cref{mapofspans} and $\bar{u}$ as in \cref{baru},
\begin{equation}\label{Groth2cell}
(f,g)\xrightarrow{u}(h,k)\textrm{ in }\Span\textrm{ such that }
\begin{tikzcd}[row sep=.2in]
\Lan_g(Af)\ar[dr,"\alpha"]\ar[d,"\bar{u}_A"'] & \\
\Lan_k(Ah)\ar[r,"\beta"'] & B
\end{tikzcd}
\textrm{ in } [Y,\Vv].
\end{equation}
This bicategory has exactly the same objects as $\Span|\Vv$, and there is an isomorphism
between their hom-categories: the necessary bijections between \cref{Groth1cell} and \cref{1cell2},
\cref{Groth2cell} and \cref{2cellsSpanV} are exactly the universal property of left Kan extension.
This isomorphism commutes with the projections to $\Span$ and thus makes $U\colon\Span|\Vv\to\Span$ a 2-opfibration.
\end{proof}
The above theorem shows that the definition of the bicategory $\Span|\Vv$ in fact arises from a series of quite natural constructions.

%% file: 6-HopfVCats.tex
\section{Hopf \texorpdfstring{$\Vv$}{V}-categories as oplax Hopf monoids}\label{sec:FrobeniusHopfVCatsSpanV}

In this section, we recall the notion of Hopf $\Vv$-category~\cite{BCV} for a braided monoidal category $\Vv$, and we realize them as oplax Hopf monoids in the monoidal bicategory $\Span|\Vv$, described respectively in \cref{sec:oplaxHopf,sec:SpanV}. This allows us to understand exactly how such a concept can be described as a `Hopf object' internal to some monoidal structure, generalizing the classical setting via this relaxed notion of a Hopf monoid in higher dimensions.

Our notation in what follows uses Latin letters to denote `global' operations that relate hom-objects of different indices (like category composition),
and Greek letters to denote `local' operations that relate hom-objects of fixed indices (like monoid multiplication).

\subsection{Hopf \texorpdfstring{$\ca{V}$}{V}-categories}\label{Hopfcats}

Intuitively, a category can be thought of as a many-object generalization
of a monoid.
Since a Hopf monoid --- internal to any braided monoidal category --- has both a monoid and a comonoid structure,
it is reasonable to ask what its many-object generalization should be.
A natural answer to this question is as follows.

\begin{definition}\cite[\S 2]{BCV}\label{def:sHopfcat}
If $(\Vv,\otimes,I,\braid)$ is braided, a \emph{semi-Hopf $\Vv$-category} $H$
is a $\Comon(\Vv)$-enriched category. Explicitly, it consists of a collection of objects $H_0$ and for
every $x,y\in H_0$ an object $H_{x,y}$ of $\Vv$, together with families of morphisms in $\Vv$
\begin{gather*}
\mlt_{xyz}\colon H_{x,y}\otimes H_{y,z}\to H_{x,z}\qquad \uni_x\colon I\to H_{x,x} \\
\lcomlt_{ab}\colon H_{a,b}\to H_{a,b}\otimes H_{a,b}\qquad \lcouni_{ab}\colon H_{a,b}\to I 
\end{gather*}
which make $ H$ a $\Vv$-category, each $ H_{x,y}$ a comonoid in $\Vv$, and
satisfy
\begin{equation}
\begin{gathered}
\xymatrix@C=.6in{
 H_{x,y}\otimes  H_{y,z} \ar[r]^-{\lcomlt_{xy}\otimes\lcomlt_{yz}} \ar[dd]_{\mlt_{xyz}} &  H_{x,y}\otimes H_{x,y}\otimes
 H_{y,z}\otimes H_{y,z} \ar[d]^{ 1\otimes\braid\otimes1} \\ 
&  H_{x,y}\otimes H_{y,z}\otimes H_{x,y}\otimes H_{y,z} \ar[d]^{\mlt_{xyz}\otimes\mlt_{xyz}} \\ 
 H_{x,z} \ar[r]_{\lcomlt_{xz}} &  H_{x,z}\otimes H_{x,z} }\label{Hax1} \\
\xymatrix{
I \ar[r]^{\sim} \ar[d]_-{\uni_{x}} & I\otimes I \ar[d]^-{\uni_{x}\otimes\uni_{x}} \\ 
 H_{x,x} \ar[r]_-{\lcomlt_{xx}} &  H_{x,x}\otimes H_{x,x} } \quad
\xymatrix@C=.5in{
 H_{x,y}\otimes H_{y,z} \ar[r]^-{\lcouni_{xy}\otimes\lcouni_{yz}} \ar[d]_-{\mlt_{xyz}} & I\otimes I
\ar[d]^-{\sim} \\ 
 H_{x,z} \ar[r]_-{\lcouni_{xz}} & I }\quad
\xymatrix{
I \ar[r]^{\id} \ar[d]_-{\uni_{x}} & I \ar[d]^{\id} \\ 
 H_{x,x} \ar[r]_-{\lcouni_{xx}} & I. }
\end{gathered}
\end{equation}
\end{definition}

Semi-Hopf $\Vv$-categories together with $\Comon(\Vv)$-functors and $\Comon(\Vv)$-natural tran\-sformations
form the 2-category $\Comon(\Vv)\mhyphen\Cat$, denoted by $\sHopfCat{\Vv}$.

\begin{examples}\hfill
\begin{enumerate}
 \item Every bimonoid in a braided monoidal category $\Vv$ can be viewed as a 1-object semi-Hopf $\Vv$-category,
therefore semi-Hopf categories are indeed many-object generalizations of bialgebras.
\item For any cartesian monoidal category $\ca{V}$, where by default $\Comon(\Vv)\cong\Vv$, any $\Vv$-enriched category is automatically a semi-Hopf $\Vv$-category.
\item In~\cite{Measuringcomonoid}, it is established that when $\ca{V}$ is a locally presentable braided monoidal closed category, the category of monoids $\Mon(\Vv)$ is enriched in $\Comon(\Vv)$ via the existence of \emph{universal measuring comonoids}, first introduced by Sweedler in~\cite{Sweedler}. Therefore $\Mon(\Vv)$ is a semi-Hopf $\Vv$-category. The assumptions on $\Vv$ are satisfied by a wide class of examples, like vector spaces, modules over a commutative ring, chain complexes over a commutative ring, Grothendieck toposes etc.
\end{enumerate}
\end{examples}

\begin{definition}\cite[Def.~2.3]{BCV}\label{def:Hopfcat}
A \emph{Hopf $\Vv$-category} is a semi-Hopf $\Vv$-category equipped with a family of maps
$\atpd_{xy} \colon  H_{x,y} \to  H_{y,x}$ in $\Vv$ satisfying
\begin{equation}\label{HopfCatAntipodeEquations}
\begin{gathered}
\cd[@C-2em]{
 &  H_{x,y}\otimes H_{x,y} \ar[rr]^{ H_{x,y} \otimes \atpd_{xy}} & &  H_{x,y}\otimes H_{y,x} \ar[dr]^{\mlt_{xyx}} & \\
  H_{x,y} \ar[rr]^{\lcouni_{xy}} \ar[ur]^{\lcomlt_{xy}} & & I \ar[rr]^{\uni_{x}} & &  H_{x,x} 
} \\ 
\cd[@C-2em]{
 &  H_{x,y}\otimes H_{x,y} \ar[rr]^{\atpd_{xy} \otimes H_{x,y}} & &  H_{y,x}\otimes H_{x,y} \ar[dr]^{\mlt_{yxy}} & \\
  H_{x,y} \ar[rr]^{\lcouni_{xy}} \ar[ur]^{\lcomlt_{xy}} & & I \ar[rr]^{\uni_{y}} & &  H_{y,y} \rlap{\ .}
} 
\end{gathered}
\end{equation}
Such an identity-on-objects
$\Vv$-graph map $\atpd\colon H\to H^\op$ is called the \emph{antipode} of $H$.
\end{definition}

If $H$ and $K$ are Hopf $\Vv$-categories, a $\Comon(\Vv)$-functor $F\colon H\to K$
is called a \emph{Hopf $\Vv$-functor} if $\atpd_{f(x)f(y)}\circ F_{xy}
=F_{yx}\circ\atpd_{xy}$ for all $x,y\in X$. In fact,
any $\Comon(\Vv)$-functor automatically satisfies that condition~\cite[{}2.10]{BCV}; hence we have a full $2$-subcategory
$\HopfCat{\Vv}$ of $\sHopfCat{\Vv}$.

\begin{examples}\hfill
\begin{enumerate}
 \item Every Hopf algebra $H$ in a braided monoidal $\Vv$ is a 1-object Hopf $\Vv$-category; 
in particular, each `diagonal' hom-object $H_{x,x}$ of a Hopf $\Vv$-category $H$ is a Hopf monoid in $\Vv$.
\item A Hopf $\Set$-category is precisely a groupoid; in general, for a cartesian monoidal $\Vv$, a Hopf $\Vv$-category is a \emph{$\Vv$-groupoid},
namely a $\Vv$-category whose hom-objects are equipped with inversion maps.
\item If we replace $\Vv$ with $\Vv^\op$,
we obtain the notion of a \emph{(semi-)Hopf $\Vv$-opcategory}, called
\emph{dual} Hopf category in \cite{BCV}. Since $\Comon(\Vv^\op)\cong\Mon(\Vv)^\op$,
a semi-Hopf $\Vv$-opcategory $(C,\comlt,\couni,\lmlt,\luni)$
is precisely a $\Mon(\Vv)$-opcategory, i.e.\ is equipped with `global' cocomposition
and coidentity morphisms $\comlt_{xyz},\couni_{x}$ as in \cref{Vopcat}, together with `local' multiplication
and unit morphisms $\lmlt_{xy}\colon C_{x,y}\ot C_{x,y}\to C_{x,y}$, $\luni_{xy}\colon I\to C_{x,y}$
making each hom-object a monoid in $\Vv$, subject to compatibility conditions.
Moreover, a Hopf $\Vv$-opcategory comes with arrows $\atpd_{xy}\colon C_{y,x}\to C_{x,y}$ satisfying dual axioms to \cref{HopfCatAntipodeEquations}.
\end{enumerate}
\end{examples}

\begin{remark}
 Hopf $\Vv$-categories can be realized as \emph{Hopf monads} inside two different monoidal bicategories as follows.
In~\cite[\S 4.8]{Gabipolyads}, Hopf $\Vv$-categories are captured as certain Hopf monads in $\Span|\Vv$ (those with indiscrete underlying category), for a one-object monoidal bicategory $\Vv$. On the other hand, Hopf $\Vv$-categories are in bijection to Hopf monads in $\Vv$-$\Mat$, the bicategory of $\Vv$-matrices; the relevant structure is sketched in \cite[Remark 4.8]{VCocats}.
\end{remark}

\subsection{Oplax Hopf monoid structure}
In the following theorem, we summarize the main results of this section regarding oplax bimonoid and Hopf monoid structures on an object in $\Span|\Vv$. Recall that by \cref{prop:Fpreserves,prop:Fpreservesatpd}, the strict monoidal structure \cref{eq:Ustrict} of the forgetful functor $U\colon\Span|\Vv\to\Span$ ensures that 
the underlying set of an oplax Hopf monoid in $\Span|\Vv$ is an oplax Hopf monoid in $\Span$; the relevant structures on $X^2$ were described in \cref{sec:X2structures}.

\begin{theorem}\label{thm:centralthm}
Let $\Vv$ be a braided monoidal category, and $X$ any set.
\begin{enumerate}
\item A pseudomonoid (resp.\ pseudocomonoid) in $\Span|\Vv$ over the groupoid pseudomonoid
(resp.\ pseudocomonoid) $X^2$ in $\Span$ is exactly a $\Vv$-category (resp.\ $\Vv$-opcategory) with objects $X$.
\item A monoid (resp.\ comonoid) in $\Span|\Vv$ over the trivial monoid (resp.\ comonoid)
$X^2$ in $\Span$ is exactly a $\Mon(\Vv)$-graph (resp.\ $\Comon(\Vv)$-graph) with objects $X$.
\item An oplax bimonoid in $\Span|\Vv$ over the (groupoid pseudomonoid, trivial comonoid)
oplax bimonoid $X^2$ in $\Span$ is exactly a semi-Hopf $\Vv$-category.
\item An oplax bimonoid in $\Span|\Vv$ over the (trivial monoid, groupoid pseudocomonoid) oplax bimonoid
$X^2$ in $\Span$ is exactly a semi-Hopf $\Vv$-opcategory.
\item An oplax Hopf bimonoid in $\Span|\Vv$ over the (groupoid pseudomonoid, trivial comonoid)
oplax Hopf bimonoid $X^2$ in $\Span$ is exactly a Hopf $\Vv$-category.
\item An oplax Hopf bimonoid in $\Span|\Vv$ over the (trivial monoid, groupoid pseudocomonoid)
oplax Hopf bimonoid $X^2$ in $\Span$ is exactly a Hopf $\Vv$-opcategory.
\end{enumerate}
\end{theorem}

Similar results are obtained for the morphisms between such objects, leading to isomorphisms of the corresponding categories. We provide proofs of some parts of the theorem below, and the rest follow analogously.

\begin{proposition}\label{Vopcatsascomonoids}
A pseudocomonoid in $\Span|\Vv$ over the set $X^2$ with the groupoid pseudocomonoid structure in $\Span$ is precisely a $\ca{V}$-opcategory
with set of objects $X$. Dually, a pseudomonoid in $\Span|\Vv$ over the set $X^2$
with the groupoid pseudomonoid structure in $\Span$ is a $\Vv$-category. 
\end{proposition}
\begin{proof}
Consider a pseudocomonoid $C_{X^2}$ in $\Span|\Vv$, with underlying pseudocomonoid $X^2$ in $\Span$ equipped with the reversed
\cref{groupoidmonoidstr}. It consists of an object
$C\colon X^2\to\Vv$, i.e.\ a family of objects $\{C_{x,y}\}_{x,y\in X}$ in $\ca{V}$, along with
$\comlt\colon C\otimes C\to C$ and $\couni\colon I\to C$ in $\Span|\Vv$, i.e.\ natural transformations
\begin{displaymath}
\begin{tikzcd}[row sep=.1in]
& |[alias=doma]|X^2\ar[dr,bend left=15,"C"] & \\
 X^3\ar[ur,bend left=15,"\pi_{13}"]\ar[dr,bend right=15,"1\times\diag\times1"'] & & \ca{V}, \\
 & |[alias=coda]|X^4\ar[ur,bend right=15,"C\otimes C"']
 \arrow[Rightarrow,from=doma,to=coda,"\comlt",shorten >=.1in,shorten <=.1in] 
\end{tikzcd}\quad
\begin{tikzcd}[row sep=.1in]
& |[alias=doma]|X^2\ar[dr,bend left=15,"C"] & \\
 X\ar[ur,bend left=15,"\diag"]\ar[dr,bend right=15,"!"'] & & \ca{V} \\
 & |[alias=coda]|1\ar[ur,bend right=15,"I"']
 \arrow[Rightarrow,from=doma,to=coda,"\couni",shorten >=.1in,shorten <=.1in] 
\end{tikzcd} \textrm{ with components}
\end{displaymath}
\begin{displaymath}
\comlt_{xyz}\colon C_{x,z}\to C_{x,y}\otimes C_{y,z}\qquad\textrm{ and }\qquad
\couni_{x}\colon C_{x,x}\to I.
\end{displaymath}
The coassociativity isomorphism $\beta$ dual to \cref{alphalambdarho}, using the composition formula~\cref{SpanVcomposition}, amounts to
\begin{displaymath}
 \begin{tikzcd}[row sep=.1in,column sep=.25in]
 & & |[alias=doma]|X^2\ar[ddrr,bend left,"C"description] & & \\
 & X^3\ar[ur,"\pi_{13}"]\ar[dr,"1\diag1"'] & && \\
 P\cong X^4\ar[ur,dashed,"\pi_{124}"]
 \ar[dr,dashed,"1\Delta11"']
 \ar[rr,phantom, near start,"\ucorner"description] & & |[alias=coda]|X^4\ar[rr,"C\otimes C"] & &\ca{V} \\
 & X^2\times X^3\ar[ur,"1^2\pi_{13}"]\ar[dr,"1^3\diag1"'] & & \\
 & & |[alias=codb]|X^6\ar[uurr,bend right,"C\otimes C\otimes C"description] &
 \arrow[Rightarrow,from=doma,to=coda,"\comlt",shorten >=.1in,shorten <=.1in]
 \arrow[Rightarrow,from=coda,to=codb,"1_C\otimes \comlt",shorten >=.1in,shorten <=.1in]
 \end{tikzcd}\stackrel{\beta}{\cong}
 \begin{tikzcd}[row sep=.1in,column sep=.25in]
 & & |[alias=doma]|X^2\ar[ddrr,bend left,"C"description] & & \\
 & X^3\ar[ur,"\pi_{13}"]\ar[dr,"1\diag1"'] & && \\
 Q\cong X^4\ar[ur,dashed,"\pi_{134}"]\ar[dr,dashed,"11\Delta1"']
 \ar[rr,phantom, near start,"\ucorner"description] & & |[alias=coda]|X^4\ar[rr,"C\otimes C"] & &\ca{V} \\
 & X^3\times X^2\ar[ur,"\pi_{13}1^2"]\ar[dr,"1\diag1^3"'] & & \\
 & & |[alias=codb]|X^6\ar[uurr,bend right,"C\otimes C\otimes C"description] &
 \arrow[Rightarrow,from=doma,to=coda,"\comlt",shorten >=.1in,shorten <=.1in]
 \arrow[Rightarrow,from=coda,to=codb,"\comlt\otimes1_C",shorten >=.1in,shorten <=.1in]
 \end{tikzcd}
\end{displaymath} 
which is the invertible map of spans \cref{eq:coassociator}, satisfying the factorization 
\cref{2cellsSpanV} expressed by
\begin{gather*}
\left((1_C\otimes \comlt)\circ \comlt\right)_{xyzw}\colon C_{x,w}\xrightarrow{\comlt_{xyw}}C_{x,y}\otimes C_{y,w}\xrightarrow{1\otimes\comlt_{yzw}}
C_{x,y}\otimes C_{y,z}\otimes C_{z,w} \\
\left((\comlt\otimes1_C)\circ \comlt\right)_{xyzw}\colon C_{x,w}\xrightarrow{\comlt_{xzw}}C_{x,z}\otimes C_{z,w}\xrightarrow{\comlt_{xyz}\otimes1}
C_{x,y}\otimes C_{y,z}\otimes C_{z,w}
\end{gather*}
which is precisely the first $\ca{V}$-opcategory axiom,
see \cref{enrichedstuff}.
Analogously, from the counity isomorphisms we obtain the coidentity axioms for enriched opcategories.
\end{proof}

On the level of morphisms between such structures, the following indicates the relation
of $\ca{V}$-functors and oplax morphisms between pseudomonoids in $\Span|\Vv$.

\begin{proposition}\label{prop:Vfunctors}
Suppose we have two pseudomonoids $A_{X^2}$ and $B_{Y^2}$ in $\Span|\Vv$ with the groupoid pseudomonoid structure on their underlying sets,
i.e.\ two $\ca{V}$-categories. An oplax pseudomonoid morphism 
in $\Span|\Vv$ of the form
\begin{equation}\label{Vfunctorcell}
\begin{tikzcd}[row sep=.1in]
& |[alias=doma]|X^2\ar[dr,bend left=15,"A"] & \\
 X^2\ar[ur,bend left=15,"\id"]\ar[dr,bend right=15,"f\times f"'] & & \ca{V} \\
 & |[alias=coda]|Y^2\ar[ur,bend right=15,"B"']
 \arrow[Rightarrow,from=doma,to=coda,"\alpha",shorten >=.1in,shorten <=.1in] 
\end{tikzcd}
\end{equation}
is precisely a $\ca{V}$-functor. Dually, an oplax pseudocomonoid morphism of the same form is a $\Vv$-opfunctor between two $\Vv$-opcategories.
\end{proposition}

\begin{proof}
The map $^\id\alpha^{f\times f}$ has components $\alpha_{xy}\colon A_{x,y}\to B_{fx,fy}$
and comes equipped with two 2-cells $\phi$ and $\phi_0$ as in \cref{oplax2cells}.
The first 2-cell has to be of the form
\begin{equation}\label{Vfunctphi}
 \begin{tikzcd}[row sep=.1in,column sep=.25in]
 & & |[alias=doma]|X^4\ar[ddrr,bend left,"A\otimes A"] & & \\
 & X^3\ar[ur,"1\diag1"]\ar[dr,"\pi_{13}"'] & && \\
 X^3\ar[ur,dashed,"\id"]\ar[dr,dashed,"\pi_{13}"']\ar[rr,phantom, near start,"\ucorner"description] & &
 |[alias=coda]|X^2\ar[rr,"A"] & &\ca{V} \\
 & X^2\ar[ur,"\id"]\ar[dr,"f^2"'] & & \\
 & & |[alias=codb]|Y^2\ar[uurr,bend right,"B"'] &
 \arrow[Rightarrow,from=doma,to=coda,"\mlt",shorten >=.1in,shorten <=.1in]
 \arrow[Rightarrow,from=coda,to=codb,"\alpha",shorten >=.1in,shorten <=.1in]
 \end{tikzcd}\;\stackrel{\phi}{\Rightarrow}\;
 \begin{tikzcd}[row sep=.1in,column sep=.25in]
 & & |[alias=doma]|X^4\ar[ddrr,bend left,"A\otimes A"] & & \\
 & X^4\ar[ur,"\id"]\ar[dr,"f^4"'] & && \\
 S\ar[ur,dashed,hook]\ar[dr,dashed,"f^2!f"']\ar[rr,phantom, near start,"\ucorner"description] & &
 |[alias=coda]|Y^4\ar[rr,"B\otimes B"] & &\ca{V} \\
 & Y^3\ar[ur,"1\diag1"]\ar[dr,"\pi_{13}"'] & & \\
 & & |[alias=codb]|Y^2\ar[uurr,bend right,"B"'] &
 \arrow[Rightarrow,from=doma,to=coda,"\alpha\otimes\alpha",shorten >=.1in,shorten <=.1in]
 \arrow[Rightarrow,from=coda,to=codb,"\mlt",shorten >=.1in,shorten <=.1in]
 \end{tikzcd}
 \end{equation}
where we compute $S=\{(x,y,z,w)\in X^4\;|\;fy=fz\}\subseteq X^4$. On the left and on the right,
respectively, we have components
\begin{gather}
(\alpha\circ\mlt)_{xyz}\colon A_{x,y}\otimes A_{y,z}\xrightarrow{\mlt}A_{x,z}\xrightarrow{\alpha_{xz}}B_{fx,fz}\label{Vfunct} \\
(m\circ(\alpha\otimes\alpha))_{xyzw}\colon A_{x,y}\otimes A_{z,w}\xrightarrow{\alpha_{xy}\otimes\alpha_{zw}}B_{fx,fy}\otimes B_{fz,fw}
\xrightarrow{\mlt}B_{fx,fz}\nonumber\qquad(fy=fz\textrm{ in } S)
\end{gather}
To define such a 2-cell $\phi$ in $\Span|\Vv$, we need a map of spans between the induced outer $\spn{Y^2}{X^3}{X^4}{}{}$ and $\spn{Y^2}{S}{X^4}{}{}$ above, that gives a factorization as in \cref{2cellsSpanV}. The function $1\times\diag\times1\colon X^3\to S$ is indeed a map of spans, and the factorization of $(\alpha\circ\mlt)$
\begin{displaymath}
\begin{tikzcd}[row sep=.1in,column sep=.55in]
&& |[alias=doma]|X^4\ar[dr,bend left=15,"A\otimes A"] & \\
X^3\ar[r,"1\diag1"description]\ar[rru, bend left=15,"1\diag1"description]\ar[drr, bend right=15,"f!f"description]
& S\ar[ur,hook]\ar[dr,"f!^2f"description] && \ca{V} \\
&& |[alias=coda]|Y^2\ar[ur,bend right=15,"B"']
\arrow[Rightarrow,from=doma,to=coda,"\mlt\circ(\alpha\ot\alpha)",shorten >=.1in,shorten <=.1in] 
\end{tikzcd}
\end{displaymath}
requires the equality of the two composites~\cref{Vfunct}, but the second only applied to quadruples
$(x,y,y,z)$; this is precisely the first axiom for a $\ca{V}$-functor.
The second 2-cell is
\begin{equation}\label{Vfunctphi0}
 \begin{tikzcd}[row sep=.05in,column sep=.25in]
 & & |[alias=doma]|1\ar[ddrr,bend left,"I"] & & \\
 & X\ar[ur,"!"]\ar[dr,"\diag"'] & && \\
 X\ar[ur,dashed,"\id"]\ar[dr,dashed,"\diag"']\ar[rr,phantom, near start,"\ucorner"description] & &
 |[alias=coda]|X^2\ar[rr,"A"] & &\ca{V}\\
 & X^2\ar[ur,"\id"]\ar[dr,"f^2"'] & & \\
 & & |[alias=codb]|Y^4\ar[uurr,bend right,"B"'] &
 \arrow[Rightarrow,from=doma,to=coda,"\uni",shorten >=.1in,shorten <=.1in]
 \arrow[Rightarrow,from=coda,to=codb,"\alpha",shorten >=.1in,shorten <=.1in]
 \end{tikzcd}\;\stackrel{\phi_0}{\Rightarrow}\;
\begin{tikzcd}[row sep=.1in]
& |[alias=doma]|1\ar[dr,bend left=15,"I"] & \\
Y\ar[ur,bend left=15,"!"]\ar[dr,bend right=15,"\diag"'] & & \ca{V} \\
& |[alias=coda]|Y^2\ar[ur,bend right=15,"B"']
\arrow[Rightarrow,from=doma,to=coda,"\uni",shorten >=.1in,shorten <=.1in] 
\end{tikzcd}
\end{equation}
between natural transformations with components $I\xrightarrow{\uni_x}A_{x,x}\xrightarrow{\alpha_{xx}}B_{fx,fx}$ and
$I\xrightarrow{\uni_y}B_{y,y}$ respectively.
Hence define $\phi_0$ by the map of spans $f\colon X\to Y$ with a factorization
\begin{displaymath}
\begin{tikzcd}[row sep=.1in,column sep=.5in]
&& |[alias=doma]|1\ar[dr,bend left=15,"I"] & \\
X\ar[r,"f"]\ar[rru, bend left=15,"!"description]\ar[drr, bend right=15,"{(f,f)}"description]
& Y\ar[ur,"!"description]\ar[dr,"\diag"description] && \ca{V} \\
&& |[alias=coda]|Y^2\ar[ur,bend right=15,"B"']
\arrow[Rightarrow,from=doma,to=coda,"\uni",shorten >=.1in,shorten <=.1in] 
\end{tikzcd}=\alpha\circ\uni
\end{displaymath}
which gives $\alpha_{xx}\circ\uni_x=\uni_{fx}$, precisely the second axiom for a $\ca{V}$-functor.

In order to verify that $\phi,\phi_0$ endow $\alpha$ with the structure of an oplax pseudomonoid morphism, we compute the appropriate pasting diagrams~\cref{eq:pseudomonmap} using composition formulas like~\cref{SpanVcomposition} and~\cref{hcomp2cellsSpanV}. Their verifications are greatly simplified by \cref{2cellsequality} since they reduce to ones for 2-cells in $\Span$. 

Notice that both $\phi$ and $\phi_0$ are uniquely defined as above, since at each case there is only one span morphism possible, see 
\cref{rem:uniquespanmap}. Therefore $\Vv$-functors are in bijection with oplax monoid morphisms $^\id\alpha^{f\times f}$.

Dually, taking into account that $\Vv$-opfunctors
are $\Vv^\op$-functors hence live inside
\begin{displaymath}
\PsMon_\opl(\Span|(\Vv^\op)) \simeq \PsMon_\opl((\Span|\Vv)^\op) \simeq \PsComon_\opl(\Span|\Vv)^\op
\end{displaymath}
we deduce that an oplax comonoid map of the form $^\id\alpha^{f\times f}$ is a $\ca{V}$-opfunctor
between the corresponding $\Vv$-opcategories.
\end{proof}

As a result, we can realize $\Vv$-$\Cat$ as a subcategory of $\PsMon_\opl(\Span|\Vv)$, the bicategory of pseudomonoids and
oplax morphisms between them, and $\VopCat$ as a subcategory of $\PsComon_\opl(\Span|\Vv)^\op$. The strict associativity and unitality of composition of 1-cells \cref{Vfunctorcell} is due to the fact that they actually come from functions, i.e. are in the image of \cref{embeddingSet}.

We now turn to the expression of semi-Hopf categories as oplax bimonoids
(\cref{oplaxbimonoid}) in $\Span|\Vv$, using the oplax bimonoid structure on $X^2\in\Span$ (\cref{X2oplaxbimonoid}); we first consider 
(pseudo)comonoids over the trivial comonoid structure of $X^2$ as discussed in \cref{sec:trivialmon}.

\begin{proposition}\label{Vgraphsascomonoids}
A comonoid in $\Span|\Vv$ over $X^2$ with the trivial comonoid structure
is precisely a $\Comon(\Vv)$-graph, i.e.\ an $X^2$-indexed family of comonoids in $\ca{V}$.
\end{proposition}

\begin{proof}
Such an object consists of $M\colon X^2\to\ca{V}$ together with natural transformations
\begin{displaymath}
\begin{tikzcd}[row sep=.07in]
& |[alias=doma]|X^2\ar[dr,bend left=15,"M"] & \\
X^2\ar[ur,bend left=15,"\id"]\ar[dr,bend right=15,"\diag"'] & & \ca{V} \\
& |[alias=coda]|X^2\times X^2\ar[ur,bend right=15,"M\otimes M"']
\arrow[Rightarrow,from=doma,to=coda,"\lcomlt",shorten >=.1in,shorten <=.1in] 
\end{tikzcd}\quad \textrm{and} \quad 
\begin{tikzcd}[row sep=.07in]
& |[alias=doma]|X^2\ar[dr,bend left=15,"M"] & \\
X^2\ar[ur,bend left=15,"\id"]\ar[dr,bend right=15,"!"'] & & \ca{V} \\
& |[alias=coda]|1\ar[ur,bend right=15,"I"']
\arrow[Rightarrow,from=doma,to=coda,"\lcouni",shorten >=.1in,shorten <=.1in] 
\end{tikzcd}
\end{displaymath}
which are given by components
\begin{gather*}
\lcomlt_{xy}\colon M_{x,y}\to M_{x,y}\otimes M_{x,y} \quad\textrm{and}\quad \lcouni_{xy}\colon M_{x,y}\to I.
\end{gather*}
Using the same machinery as for earlier proofs, we deduce that these maps satisfy
coassociativity and counity conditions which render each $M_{x,y}\in\Vv$ a strict comonoid; as discussed in \cref{sec:trivialmon}, this 
comultiplication and counit only form a pseudocomonoid structure with identity structure 2-cells.
\end{proof}

\begin{proposition}\label{semiHopfcatsasoplaxbimonoids}
An oplax bimonoid in $\Span|\Vv$ over the oplax bimonoid $X^2$ in $\Span$ is precisely a semi-Hopf $\Vv$-category.
\end{proposition}

\begin{proof}
By~\cref{Vopcatsascomonoids,Vgraphsascomonoids}, we already know that an object $A$ in $\Span|\Vv$
with the structure of a pseudomonoid and (pseudo)comonoid over the
\{groupoid pseudomonoid, trivial comonoid\} $X^2$ is a $\ca{V}$-category which is $\Comon(\ca{V})$-enriched as a graph.
We will now show that requiring $(A,\mlt,\uni,\lcomlt,\lcouni)$
to be an oplax bimonoid in $\Span|\Vv$ in fact endows its enriched composition law and identities
with a comonoid morphism structure, as in~\cref{def:sHopfcat}.

Following~\cref{oplaxbimonoid}, $A$ comes equipped with four 2-cells $\theta,\theta_0,\chi,\chi_0$ in $\Span|\Vv$,
satisfying certain coherence conditions. We address each one of them in detail below.

\fbox{$\theta$}
\begin{displaymath}
\begin{tikzcd}[row sep=.1in,column sep=.3in]
&&& |[alias=doma]|X^4\ar[rrrddd,bend left,"A\otimes A"] &&& \\
&& X^4\ar[ur,"1\diag1"]\ar[dr,"\diag_{X^2}^2"'] &&&& \\
&&& |[alias=coda]|X^8\ar[rrrd,"A\otimes A\otimes A\otimes A"description] &&& \\
P\cong X^3\ar[uurr,dashed]\ar[ddrr,dashed]\ar[rr,phantom, near start,"\ucorner"description] &&
X^8\ar[ur,"\id"]\ar[dr,"1\braid1"'] &&&& \ca{V} \\
&&& |[alias=codb]|X^8\ar[urrr,"A\otimes A\otimes A\otimes A"description] && \\
&& X^6 \ar[ur,"(1\diag1)^2"]\ar[dr,"\pi_{13}^2"'] &&& \\
&&& |[alias=codc]|X^4\ar[rrruuu,bend right,"A\otimes A"']
\arrow[Rightarrow,from=doma,to=coda,"\lcomlt\otimes\lcomlt",shorten >=.1in,shorten <=.1in]
\arrow[Rightarrow,from=coda,to=codb,"1\otimes\braid\otimes1"description,shorten >=.05in,shorten <=.05in]
\arrow[Rightarrow,from=codb,to=codc,"\mlt\otimes\mlt",shorten >=.1in,shorten <=.1in] 
\end{tikzcd}\;\stackrel{\theta}{\cong}\;
 \begin{tikzcd}[row sep=.15in,column sep=.25in]
 & & |[alias=doma]|X^4\ar[ddrr,bend left,"A\otimes A"] & & \\
 & X^3\ar[ur,"1\diag1"]\ar[dr,"\pi_{13}"'] & && \\
 X^3\ar[ur,dashed,"\id"]\ar[dr,dashed,"\pi_{13}"']\ar[rr,phantom, near start,"\ucorner"description] & &
 |[alias=coda]|X^2\ar[rr,"A"] & &\ca{V}\\
 & X^2\ar[ur,"\id"]\ar[dr,"\diag_{X^2}"'] & & \\
 & & |[alias=codb]|X^4\ar[uurr,bend right,"A\otimes A"'] &
 \arrow[Rightarrow,from=doma,to=coda,"\mlt",shorten >=.1in,shorten <=.1in]
 \arrow[Rightarrow,from=coda,to=codb,"\lcomlt",shorten >=.1in,shorten <=.1in]
 \end{tikzcd}
\end{displaymath}
The left hand side pasted composite, using~\cref{tensor1cells} for the tensor of 1-cells, is given by
\begin{displaymath}
A_{x,y}\otimes A_{y,z}\xrightarrow{\lcomlt_{x,y}\otimes\lcomlt_{y,z}}A_{x,y}^{\otimes2}\otimes A_{y,z}^{\otimes2}
\xrightarrow{1\otimes\braid\otimes1}A_{x,y}\otimes A_{y,z}\otimes A_{x,y}\otimes A_{y,z}\xrightarrow{\mlt_{xyz}^{\otimes2}}
A_{x,z}\otimes A_{x,z}
\end{displaymath}
whereas the right hand side is given by $A_{x,y}\otimes A_{y,z}\xrightarrow{\mlt_{xyz}}A_{x,z}\xrightarrow{\lcomlt_{xz}}A_{x,z}
\otimes A_{x,z}$.
The required 2-cell $\theta$ is given by the (unique) span isomorphism  $P\cong X^3$
from~\cref{phi12} with the condition $(\lcomlt\circ\mlt)_{xyz}=\left((\lcomlt\otimes\lcomlt)\circ(1\otimes\braid\otimes1)
\circ(\mlt\otimes\mlt)\right)_{xyz}$, which is precisely the first commutative diagram in~\cref{Hax1}.

\fbox{$\theta_0$}
\begin{displaymath}
 \begin{tikzcd}[row sep=.1in,column sep=.25in]
 & & |[alias=doma]|1\ar[ddrr,bend left,"I"] & & \\
 & X\ar[ur,"!"]\ar[dr,"\diag"'] & && \\
 X\ar[ur,dashed,"\id"]\ar[dr,dashed,"\diag"']\ar[rr,phantom, near start,"\ucorner"description] & &
 |[alias=coda]|X^2\ar[rr,"A"] & &\ca{V}\\
 & X^2\ar[ur,"\id"]\ar[dr,"\diag_{X^2}"'] & & \\
 & & |[alias=codb]|X^4\ar[uurr,bend right,"A\otimes A"'] &
 \arrow[Rightarrow,from=doma,to=coda,"\uni",shorten >=.1in,shorten <=.1in]
 \arrow[Rightarrow,from=coda,to=codb,"\lcomlt",shorten >=.1in,shorten <=.1in]
 \end{tikzcd}\;\stackrel{\theta_0}{\Rightarrow}\;
\begin{tikzcd}[row sep=.1in]
& |[alias=doma]|1\times 1\ar[dr,bend left=15,"I\otimes I"] & \\
X\times X\ar[ur,bend left=15,"!"]\ar[dr,bend right=15,"\diag\times\diag"'] & & \ca{V} \\
& |[alias=coda]|X^2\times X^2\ar[ur,bend right=15,"A\otimes A"']
\arrow[Rightarrow,from=doma,to=coda,"\uni\otimes\uni",shorten >=.1in,shorten <=.1in] 
\end{tikzcd}
\end{displaymath}
The left hand side composite is given by arrows $I\xrightarrow{\uni_x}A_{x,x}\xrightarrow{\lcomlt_{x,x}}A_{x,x}\otimes A_{x,x}$
and the right hand side by $I\xrightarrow{\uni_x\otimes\uni_y}A_{x,x}\otimes A_{y,y}$ in $\ca{V}$. The 2-cell $\theta_0$ is given by the morphism
of spans $X\xrightarrow{\diag}X^2$ from~\cref{phi34} and the factorization
\begin{displaymath}
\begin{tikzcd}[row sep=.1in,column sep=.5in]
&& |[alias=doma]|1\times 1\ar[dr,bend left=15,"I\otimes I"] & \\
X\ar[r,"\diag"]\ar[rru, bend left=15,"!"description]\ar[drr, bend right=15,"\diag^2"description]
& X\times X\ar[ur,"!"description]\ar[dr,"\diag\times\diag"description] && \ca{V} \\
&& |[alias=coda]|X^2\times X^2\ar[ur,bend right=15,"A\otimes A"']
\arrow[Rightarrow,from=doma,to=coda,"\uni\otimes\uni",shorten >=.1in,shorten <=.1in] 
\end{tikzcd}=\lcomlt\circ\lcomlt
\end{displaymath}
i.e.\ $(\lcomlt\circ\uni)_x=(\uni\otimes\uni)_{\diag x}=\uni_x\otimes\uni_x$ which gives the second commutative diagram of~\cref{Hax1}.
 
\fbox{$\chi$}
\begin{displaymath}
\begin{tikzcd}[row sep=.1in,column sep=.25in]
& & |[alias=doma]|X^4\ar[ddrr,bend left,"A\otimes A"] & & \\
& X^3\ar[ur,"1\times\diag\times1"]\ar[dr,"\pi_{13}"'] & && \\
X^3\ar[ur,dashed,"\id"]\ar[dr,dashed,"\pi_{13}"']\ar[rr,phantom, near start,"\ucorner"description] & &
|[alias=coda]|X^2\ar[rr,"A"] & &\ca{V}\\
& X^2\ar[ur,"\id"]\ar[dr,"!"'] & & \\
& & |[alias=codb]|1\ar[uurr,bend right,"I"'] &
\arrow[Rightarrow,from=doma,to=coda,"\mlt",shorten >=.1in,shorten <=.1in]
\arrow[Rightarrow,from=coda,to=codb,"\lcouni",shorten >=.1in,shorten <=.1in]
\end{tikzcd}\;\stackrel{\chi}{\Rightarrow}\;
\begin{tikzcd}[row sep=.1in]
& |[alias=doma]|X^2\times X^2\ar[dr,bend left=15,"A\otimes A"] & \\
X^2\times X^2\ar[ur,bend left=15,"\id"]\ar[dr,bend right=15,"!"'] & & \ca{V} \\
& |[alias=coda]|1\times1\ar[ur,bend right=15,"I\otimes I"']
\arrow[Rightarrow,from=doma,to=coda,"\lcouni\otimes\lcouni",shorten >=.1in,shorten <=.1in] 
\end{tikzcd}
\end{displaymath}
Similarly to above, we have components $A_{x,y}\otimes A_{y,z}\xrightarrow{\mlt_{xyz}}A_{x,z}\xrightarrow{\lcouni_{xz}}I$
and $A_{x,y}\otimes A_{z,w}\xrightarrow{\lcouni_{xy}\otimes\lcouni_{zw}}I\otimes I\cong I$, and the 2-cell $\chi$ is
the a map of spans $X^3\xrightarrow{1\times\diag\times1}X^4$ as in~\cref{phi34} with the factorization
\begin{displaymath}
\begin{tikzcd}[row sep=.1in,column sep=.5in]
&& |[alias=doma]|X^2\times X^2\ar[dr,bend left=15,"A\otimes A"] & \\
X^3\ar[r,"1\diag1"]\ar[rru, bend left=15,"1\diag1"description]\ar[drr, bend right=15,"!"description]
& X^2\times X^2\ar[ur,"\id"description]\ar[dr,"!"description] && \ca{V} \\
&& |[alias=coda]|1\times1\ar[ur,bend right=15,"I\otimes I"']
\arrow[Rightarrow,from=doma,to=coda,"\lcouni\otimes\lcouni",shorten >=.1in,shorten <=.1in] 
\end{tikzcd}
\end{displaymath}
i.e.\ $(\lcouni\circ\mlt)_{xyz}=(\lcouni\otimes\lcouni)_{xyyz}=\lcouni_{xy}\otimes\lcouni_{yz}$
which is the third commutative diagram of~\cref{Hax1}.

\fbox{$\chi_0$}
 \begin{displaymath}
 \begin{tikzcd}[row sep=.1in,column sep=.25in]
 & & |[alias=doma]|1\ar[ddrr,bend left,"I"] & & \\
 & X\ar[ur,"!"]\ar[dr,"\diag"'] & && \\
 X\ar[ur,dashed,"\id"]\ar[dr,dashed,"\diag"']\ar[rr,phantom, near start,"\ucorner"description] & &
 |[alias=coda]|X^2\ar[rr,"A"] & &\ca{V}\\
 & X^2\ar[ur,"\id"]\ar[dr,"!"'] & & \\
 & & |[alias=codb]|1\ar[uurr,bend right,"I"'] &
 \arrow[Rightarrow,from=doma,to=coda,"\uni",shorten >=.1in,shorten <=.1in]
 \arrow[Rightarrow,from=coda,to=codb,"\lcouni",shorten >=.1in,shorten <=.1in]
 \end{tikzcd}\;\stackrel{\chi_0}{\Rightarrow}\;
\begin{tikzcd}[row sep=.1in]
& |[alias=doma]|1\ar[dr,bend left=15,"I"] & \\
1\ar[ur,bend left=15,"\id"]\ar[dr,bend right=15,"\id"'] & & \ca{V} \\
& |[alias=coda]|1\ar[ur,bend right=15,"I"']
\arrow[Rightarrow,from=doma,to=coda,"1_I",shorten >=.1in,shorten <=.1in] 
\end{tikzcd}
\end{displaymath}
with components $I\xrightarrow{\uni_x}A_{x,x}\xrightarrow{\lcouni_x,x}I$ and $I\xrightarrow{\id}I$ in $\ca{V}$
respectively. The 2-cell $\chi_0$ is given by the map of spans $X\xrightarrow{!}1$ of~\cref{phi12} and the factorization
\begin{displaymath}
\begin{tikzcd}[row sep=.1in,column sep=.5in]
&& |[alias=doma]|1\ar[dr,bend left=15,"I"] & \\
X\ar[r,"!"]\ar[rru, bend left=15,"!"description]\ar[drr, bend right=15,"!"description]
& 1\ar[ur,"\id"description]\ar[dr,"\id"description] && \ca{V} \\
&& |[alias=coda]|1\ar[ur,bend right=15,"I"']
\arrow[Rightarrow,from=doma,to=coda,"1_I",shorten >=.1in,shorten <=.1in] 
\end{tikzcd}=\lcouni\ot\uni
\end{displaymath}
i.e.\ $(\lcouni\otimes\uni)_x=(1_I)_x$ which coincides with the fourth diagram of~\cref{Hax1}.

Due to~\cref{2cellsequality}, the proof of~\cref{X2oplaxbimonoid} guarantees that the above 2-cells satisfy
axioms~\cref{ax1} to~\cref{ax10} of an oplax bimonoid, hence the proof is complete.
\end{proof}

Notice that the 2-cells rendering $X^2$ an oplax bimonoid in $\Span$ in~\cref{X2oplaxbimonoid} are all unique, based on
\cref{oplaxbimonoidproperty}. Hence the above proposition establishes a bijection between semi-Hopf $\Vv$-categories
and oplax bimonoids over $X^2$ in $\Span|\Vv$.

Moving on to the morphisms between semi-Hopf categories,
we initially address the comonoid part of our structures. 

\begin{proposition}\label{prop:ComonVgraphmaps}
Suppose we have two comonoids $C_{X^2}$, $D_{Y^2}$ in $\Span|\Vv$ with the trivial comonoid structure on their underlying sets, i.e.~two 
$\Comon(\Vv)$-graphs. A (strict) comonoid map of the form $^\id\alpha^{f\times f}$ is precisely a $\Comon(\Vv)$-graph morphism.
\end{proposition}

\begin{proof}
The 1-cell $\alpha$ is of the form~\cref{Vfunctorcell}, but now the comonoid structures are the trivial and not the groupoid ones; if we denote by 
$\lcomlt$ and $\lcouni$
the comultiplication and counit, the coassociativity isomorphism can again only be the identity, namely  
$(\alpha\ot\alpha)\circ\lcomlt=\lcomlt\circ\alpha$ which is
\begin{equation}\label{eq:identity2cellComongraphs}
 \begin{tikzcd}[row sep=.1in,column sep=.25in]
 & & |[alias=doma]|X^2\ar[ddrr,bend left,"C"] & & \\
 & X^2\ar[ur,"\id"]\ar[dr,"\diag"'] & && \\
 X^2\ar[ur,dashed,"\id"]\ar[dr,dashed,"\diag"']\ar[rr,phantom, near start,"\ucorner"description] & &
 |[alias=coda]|X^4\ar[rr,"C\ot C"] & &\ca{V}= \\
 & X^4\ar[ur,"\id"]\ar[dr,"f^4"'] & & \\
 & & |[alias=codb]|Y^4\ar[uurr,bend right,"D\otimes D"'] &
 \arrow[Rightarrow,from=doma,to=coda,"\lcomlt",shorten >=.1in,shorten <=.1in]
 \arrow[Rightarrow,from=coda,to=codb,"\alpha\ot\alpha",shorten >=.1in,shorten <=.1in]
 \end{tikzcd}
 \begin{tikzcd}[row sep=.1in,column sep=.25in]
 & & |[alias=doma]|X^2\ar[ddrr,bend left,"C"] & & \\
 & X^2\ar[ur,"\id"]\ar[dr,"f^2"'] & && \\
 X^2\ar[ur,dashed,"\id"]\ar[dr,dashed,"f^2"']\ar[rr,phantom, near start,"\ucorner"description] & &
 |[alias=coda]|Y^2\ar[rr,"D"] & &\ca{V} \\
 & Y^2\ar[ur,"\id"]\ar[dr,"\diag"'] & & \\
 & & |[alias=codb]|Y^4\ar[uurr,bend right,"D\otimes D"'] &
 \arrow[Rightarrow,from=doma,to=coda,"\alpha",shorten >=.1in,shorten <=.1in]
 \arrow[Rightarrow,from=coda,to=codb,"\lcomlt",shorten >=.1in,shorten <=.1in]
 \end{tikzcd}
\end{equation}
For the $\Vv$-components, this requires the commutativity of
\begin{displaymath}
 \begin{tikzcd}[row sep=.3in]
 C_{x,y}\ar[r,"\lcomlt_{xy}"]\ar[d,"\alpha_{xy}"'] & C_{x,y}\ot C_{x,y}\ar[d,"\alpha_{xy}\ot\alpha_{xy}"] \\
 D_{fx,fy}\ar[r,"\lcomlt_{fxfy}"'] & D_{fx,fy}\ot D_{fx,fy}
 \end{tikzcd} 
\end{displaymath}
and similarly for counity;
these precisely translate in all $\alpha_{xy}$ being comonoid morphisms.
\end{proof}

Due to~\cref{Vgraphsascomonoids,prop:ComonVgraphmaps}, we conclude that the category $\Comon(\Vv)$-$\sf{Grph}$ is a subcategory
of $\Comon(\Span|\Vv)$ of those objects over $X^2$ with the trivial comonoid
structure and of strict morphisms $^\id\alpha^{f\times f}$.

Using~\cref{def:bimonoidmap}, we can express $\Comon(\Vv)$-functors in $\Span|\Vv$ as well.

\begin{proposition}
An oplax bimonoid morphism of the form $^\id\alpha^{f\times f}$ between two oplax bimonoids $A_{X^2}$ and $B_{Y^2}$ in $\Span|\Vv$ as in~\cref{semiHopfcatsasoplaxbimonoids} is a semi-Hopf $\Vv$-functor.
\end{proposition}

\begin{proof}
The 2-cells $\phi,\phi_0,\psi,\psi_0$ like~\cref{oplaxbimonoidmap} are respectively given by~\cref{Vfunctphi},
\cref{Vfunctphi0} and identities \cref{eq:identity2cellComongraphs}.~\cref{prop:Vfunctors,prop:ComonVgraphmaps} ensure that these make $^\id\alpha^{f\times f}$
into an oplax morphism of pseudomonoids and a morphism of comonoids, producing the four diagrams of a $\Comon(\Vv)$-functor through factorizations as 
usual.
The conditions of~\cref{def:bimonoidmap} can be verified in a straightforward way,
by comparing the appropriate maps of spans due to~\cref{2cellsequality}.
\end{proof}

Consequently, $\sHopfCat{\Vv}$ is a subcategory of $\OplBimon(\Span|\Vv)$ of oplax bimonoids
and morphisms between them, spanned by objects over $X^2$ as in~\cref{X2oplaxbimonoid} and maps of the form $^\id\alpha^{f\times f}$. As mentioned 
earlier, even though $\OplBimon(\Span|\Vv)$ is in general a bicategory, these structures form a category due to the form of morphisms coming from 
sheer functions.

Finally, we turn on to the subcategory of Hopf $\Vv$-categories, \cref{def:Hopfcat}. The oplax Hopf monoid $(X^2,\mu,\eta,\zeta,\nu)$ in $\Span$, 
\cref{X2Hopf}, underlies our oplax Hopf monoids in $\Span|\Vv$, as was the case
in all the above propositions.

\begin{proposition}
An oplax Hopf monoid in $\Span|\Vv$ over the oplax Hopf monoid $X^2$ in $\Span$ is precisely a Hopf $\Vv$-category.
\end{proposition}

\begin{proof}
Following \cref{oplaxHopfmonoid}, we wish to endow the oplax
bimonoid $(A,\mlt,\uni,\lcomlt,\lcouni,\theta,$ $
\theta_0,\chi,\chi_0)$ from \cref{semiHopfcatsasoplaxbimonoids}
with an antipode $\atpd\colon A\to A$ in $\Span|\Vv$
over the span $\spn{X^2}{X^2}{X^2}{\id}{\mathrm{sw}}$
\begin{displaymath}
 \begin{tikzcd}[row sep=.1in]
 & |[alias=doma]|X^2 \ar[dr,bend left=15,"A"] & \\
 X^2 \ar[ur,bend left=15,"\id"] \ar[dr,bend right=15,"\mathrm{sw}"'] & & \ca{V} \\
 & |[alias=coda]|X^2 \ar[ur,bend right=15,"A"']
 \arrow[Rightarrow,from=doma,to=coda,"\atpd",shorten >=.1in,shorten <=.1in]
 \end{tikzcd}
\end{displaymath}
given by families $\atpd_{xy} \colon A_{x,y} \to A_{y,x}$, along with two 
2-cells $\tau_1,\tau_2$ in $\Span|\Vv$ \cref{2cellsSpanV}
where for example, the first one is
\begin{displaymath}
\begin{tikzcd}[row sep=.1in,column sep=.3in]
&&& |[alias=doma]|X^2\ar[rrrddd,bend left,"A"] &&& \\
&& X^2\ar[ur,"\id"]\ar[dr,"\diag_{X^2}"'] &&&& \\
&&& |[alias=coda]|X^4\ar[rrrd,"A\otimes A"description] &&& \\
X^2\ar[uurr,dashed]\ar[ddrr,dashed]\ar[rr,phantom, near start,"\ucorner"description] &&
X^4\ar[ur,"\id"]\ar[dr,"1\mathrm{sw}"'] &&&& \ca{V} \\
&&& |[alias=codb]|X^4\ar[urrr,"A\otimes A"description] && \\
&& X^3 \ar[ur,"1\diag1"]\ar[dr,"\pi_{13}"'] &&& \\
&&& |[alias=codc]|X^2\ar[rrruuu,bend right,"A"']
\arrow[Rightarrow,from=doma,to=coda,"\lcomlt",shorten >=.1in,shorten <=.1in]
\arrow[Rightarrow,from=coda,to=codb,"1\otimes\atpd"description,shorten >=.05in,shorten <=.05in]
\arrow[Rightarrow,from=codb,to=codc,"\mlt",shorten >=.1in,shorten <=.1in] 
\end{tikzcd}\;\stackrel{\tau_1}{\Rightarrow}\;
 \begin{tikzcd}[row sep=.15in,column sep=.25in]
 & & |[alias=doma]|X^2\ar[ddrr,bend left,"A"] & & \\
 & X^2\ar[ur,"\id"]\ar[dr,"!"'] & && \\
 X^3\ar[ur,dashed,"\pi_{12}"]\ar[dr,dashed,"\pi_{3}"']\ar[rr,phantom, near start,"\ucorner"description] & &
 |[alias=coda]|1\ar[rr,"I"] & &\ca{V}\\
 & X\ar[ur,"!"]\ar[dr,"\diag"'] & & \\
 & & |[alias=codb]|X^2\ar[uurr,bend right,"A"'] &
 \arrow[Rightarrow,from=doma,to=coda,"\lcouni",shorten >=.1in,shorten <=.1in]
 \arrow[Rightarrow,from=coda,to=codb,"\uni",shorten >=.1in,shorten <=.1in]
 \end{tikzcd}
\end{displaymath}
The left hand side and right hand side components are, respectively,
\begin{gather*}
(\mlt\circ(1\otimes\atpd)\circ\lcomlt)_{xy}\colon A_{x,y}\xrightarrow{\lcomlt_{xy}}A_{x,y}\ot A_{x,y}
\xrightarrow{1\ot\atpd_{xy}}A_{x,y}\ot A_{y,x}\xrightarrow{\mlt_{xyx}}A_{x,x} \\
(\uni\circ\lcouni)_{xyz}\colon A_{x,y}\xrightarrow{\lcouni_{xy}}I\xrightarrow{\uni_z} A_{z,z}
\end{gather*}
and the 2-cell $\tau_1$ is given by the map of spans $\rho_1=(1\times\mathrm{sw})(\diag\times1)
\colon X^2\to X^3$
just like the upper \cref{tauforX2}, such that
\begin{displaymath}
\begin{tikzcd}[row sep=.1in,column sep=.5in]
&& |[alias=doma]|X^2\ar[dr,bend left=15,"AA"] & \\
X^2\ar[r,"\rho_1"]\ar[rru, bend left=15,"\id"description]\ar[drr, bend right=15,"\diag\circ\pi_1"description]
& X^3\ar[ur,"\pi_{12}"description]\ar[dr,"\diag\circ\pi_3"description] && \ca{V} \\
&& |[alias=coda]|X^2\ar[ur,bend right=15,"A"']
\arrow[Rightarrow,from=doma,to=coda,"\uni\circ\lcouni",shorten >=.1in,shorten <=.1in] 
\end{tikzcd}=\mlt\circ(1\otimes\atpd)\circ\lcomlt
\end{displaymath}
This factorization can be written as the upper commutative diagram of 
the Hopf $\Vv$-category axiom \cref{HopfCatAntipodeEquations},
and similarly for $\tau_2$.
\end{proof}

This establishes a bijection between Hopf $\Vv$-categories and oplax Hopf monoids in $\Span|\Vv$
over $X^2$, and a Hopf $\Vv$-functor is merely an oplax bimonoid morphism in $\Span|\Vv$ like before.

%% file: 7-frob.tex
\section{Frobenius \texorpdfstring{$\Vv$}{V}-categories}\label{FrobeniusVcats}

Hopf algebras are closely related to Frobenius algebras in the classical context,
and more generally to Frobenius monoids in an arbitrary monoidal category $\Vv$.
As a result, we may also ask what a many-object generalization of a Frobenius object could be. In this section, we initially provide an ad-hoc 
definition of a Frobenius $\Vv$-category, which is then realized as a Frobenius pseudomonoid in the same monoidal bicategory $\Span|\Vv$ where Hopf 
$\Vv$-categories live as oplax Hopf monoids, see \cref{sec:SpanV,sec:FrobeniusHopfVCatsSpanV}.
In subsequent work \cite{paper1b}, we study in detail the relation between Frobenius and Hopf categories, providing
a generalized version of the standard Larson-Sweedler theorem. For the time being, we restrict
to the solid introduction of this notion and its expression as a Frobenius pseudomonoid.

\subsection{Frobenius categories}

Naturally, the Frobenius counterpart of Hopf $\Vv$-categories will still have
a notion of multiplication and comultiplication. However, the `local comonoid'
part previously captured via the enrichment in comonoids is now replaced by a global operation, expressed as follows.

\begin{definition}\label{def:Frobeniuscat}
A \emph{Frobenius $\Vv$-category} $A$ is a $\Vv$-category that is also a $\Vv$-opcategory and satisfies `indexed'
Frobenius conditions. Explicitly, it consists of a set of objects $ A_0$ and for every $x,y\in  A_0$
an object $ A_{x,y}$ of $\Vv$ together with maps
\begin{gather*}
\mlt_{xyz}\colon A_{x,y}\otimes  A_{y,z}\to A_{x,z}\qquad \uni_x\colon I\to A_{x,x} \\
\comlt_{abc}\colon A_{a,c}\to A_{a,b}\otimes A_{b,c}\qquad \couni_y\colon A_{y,y}\to I 
\end{gather*}
which satisfy the $\Vv$-category and $\Vv$-opcategory axioms, as well as the commutativity of
\begin{equation}\label{frob1}
\begin{tikzcd}[column sep=.3in,row sep=.2in]
 A_{x,y}\ot A_{y,z}\ar[rr,"\comlt_{xwy}\ot1"]\ar[dd,"1\ot\comlt_{ywz}"']\ar[dr,"\mlt_{xyz}"description] &&
 A_{x,w}\ot A_{w,y}\ot A_{y,z}\ar[dd,"1\ot\mlt_{wyz}"] \\
&  A_{x,z}\ar[dr,"\comlt_{xwz}"description] & \\
 A_{x,y}\ot A_{y,w}\ot A_{w,z}\ar[rr,"\mlt_{xyw}\ot1"'] &&
 A_{x,w}\ot A_{w,z}
\end{tikzcd}
\end{equation}
\end{definition}

\begin{definition}\label{Frobeniusfunctor}
A \emph{Frobenius $\Vv$-functor} between two Frobenius categories $ A$ and $B$ is a $\Vv$-graph morphism simultaneously
in $\VCat$ and $\VopCat^\op$. This amounts to a function $f\colon A_0\to B_0$ between the sets of objects,
along with families of arrows $F_{xy}\colon A_{x,y}\to B_{fx,fy}$ in $\ca{V}$ subject
to the following axioms:
\begin{displaymath}
\begin{tikzcd}[column sep=.7in, row sep=.5in]
 A_{x,y}\otimes A_{y,z}\ar[r,"\mlt_{xyz}"]\ar[d,"F_{xy}\otimes F_{yz}"'] &  A_{x,z}\ar[d,"F_{xz}"] \\
 B_{fx,fy}\otimes B_{fy,fz}\ar[r,"\mlt_{fxfyfz}"'] &  B_{fx,fz}
\end{tikzcd}\qquad
\begin{tikzcd}[column sep=.7in, row sep=.5in]
I\ar[r,"\uni_x"]\ar[dr,"\uni_{fx}"'] &  A_{x,x}\ar[d,"F_{xx}"] \\
&  B_{fx,fx}
\end{tikzcd}
\end{displaymath}
\begin{displaymath}
\begin{tikzcd}[column sep=.7in, row sep=.5in]
 A_{x,z}\ar[r,"\comlt_{xyz}"]\ar[d,"F_{xz}"'] &  A_{x,y}\otimes A_{y,z}\ar[d,"F_{xy}\otimes F_{yz}"]  \\
 B_{fx,fz}\ar[r,"\comlt_{fxfyfz}"'] &  B_{fx,fy}\otimes B_{fy,fz}
\end{tikzcd}\qquad
\begin{tikzcd}[column sep=.7in, row sep=.5in]
 A_{x,x}\ar[d,"F_{xx}"']\ar[r,"\couni_x"] & I  \\
 B_{fx,fx}\ar[ur,"\couni_{fx}"']
\end{tikzcd}
\end{displaymath}
\end{definition}

\begin{remark}
The above introduced notions should not be confused with (related but different) ones existing in literature.  
The name {\em Frobenius category} is also used for an exact category which has enough injectives
and enough projectives and where the class of projectives coincides with the class of
injectives, see~\cite{Heller1960}. 
The name {\em Frobenius functor} is also used for a functor that has an identical left and
right adjoint, see~\cite{Caenepeel2002}. 
\end{remark}

Frobenius $\Vv$-categories and Frobenius $\Vv$-functors form a category $\FrobCat{\Vv}$.
\begin{examples}\hfill
\begin{enumerate}
 \item Every Frobenius monoid in a monoidal category $\Vv$ can be viewed a one-object Frobenius $\Vv$-category; as a result, this definition indeed serves as many-object generalization of Frobenius algebras. In particular, each `diagonal' hom-object $A_{x,x}$ of a Frobenius $\Vv$-category is a Frobenius monoid in $\Vv$.
 \item For a commutative ring $k$, let $\Mat$ be the category whose objects are the natural numbers and whose hom-sets
$\Mat_{m,n}$ are the sets of $m\times n$ matrices with entries in $k$. This is a
$\Mod_k$-category if we take the usual composition of matrices and identity matrices.
The comultiplication and counit are defined by
$$\comlt_{n,p,m} \colon \Mat_{n,m} \to \Mat_{n,p} \ot \Mat_{p,m} \colon e^{n,m}_{i,j} \mapsto
\sum_{t=1}^{p} e^{n,p}_{i,t}\ot e^{p,m}_{t,j}$$
$$\couni_{n,m} \colon \Mat_{n,m} \to k \colon e^{n,m}_{i,j} \mapsto \delta_{i,j}$$
for all $1 \leq i\leq n, 1 \leq j \leq m$, where $e^{n,m}_{i,j}$ denote the elementary matrices of $\Mat_{n,m}$ with a single $1$ in the $i$-th row and $j$-th column and zeroes elsewhere.
This makes $\Mat$ into a Frobenius $k$-linear category, generalizing the classical example of each $\Mat_{n,n}$ being a Frobenius algebra.
\end{enumerate}
\end{examples}

\subsection{Frobenius pseudomonoid structure}

The following results extend \cref{thm:centralthm} to the case of Frobenius monoids;
they are also relevant to \cref{Streetstable}. Again, the fact that $U\colon\Span|\Vv\to\Span$
is strict monoidal via \cref{eq:Ustrict} implies that the underlying set of a Frobenius pseudomonoid
in $\Span|\Vv$ is also a Frobenius pseudomonoid $\Span$ similarly to \cref{prop:Fpreserves}, and the set $X^2$ will once more play this role, with 
structures described in \cref{FrobMonStr}.


\begin{proposition}\label{Frobeniuscatsasmonoids}
A Frobenius pseudomonoid in $\Span|\Vv$ over the Frobenius $X^2$ with groupoid pseudomonoid
and groupoid pseudocomonoid structure in $\Span$ is a Frobenius $\Vv$-category.
\end{proposition}

\begin{proof}
Having established by~\cref{Vopcatsascomonoids} that a pseudomonoid and pseudocomonoid
$(A,\mlt,\uni,\comlt,\couni)$ in $\Span|\Vv$ over $X^2$ of \cref{X2Frobenius} is a
$\ca{V}$-category and $\Vv$-opcategory, we only need to check they form a Frobenius pseudomonoid structure, namely there exist isomorphisms as in \cref{pseudoFrobcond} satisfying the axioms of \cref{sec:Frobpseudo}.
The lower isomorphism is
\begin{displaymath}
 \begin{tikzcd}[row sep=.1in,column sep=.25in]
 & & |[alias=doma]|X^4\ar[ddrr,bend left,"A\otimes A"] & & \\
 & X^3\ar[ur,"1\diag1"]\ar[dr,"\pi_{13}"'] & && \\
 X^4\ar[ur,dashed]\ar[dr,dashed]\ar[rr,phantom, near start,"\ucorner"description] & &
 |[alias=coda]|X^2\ar[rr,"A"] & &\ca{V}\stackrel{\psi}{\cong} \\
 & X^3\ar[ur,"\pi_{13}"]\ar[dr,"1\diag1"'] & & \\
 & & |[alias=codb]|X^4\ar[uurr,bend right,"A\otimes A"'] &
 \arrow[Rightarrow,from=doma,to=coda,"\mlt",shorten >=.1in,shorten <=.1in]
 \arrow[Rightarrow,from=coda,to=codb,"\comlt",shorten >=.1in,shorten <=.1in]
 \end{tikzcd}
 \begin{tikzcd}[row sep=.1in,column sep=.25in]
 & & |[alias=doma]|X^4\ar[ddrr,bend left,"A\otimes A"] & & \\
 & X^2\times X^3\ar[ur,"1^2\pi_{13}"]\ar[dr,"1^3\diag1"'] & && \\
 X^4\ar[ur,dashed]\ar[dr,dashed]\ar[rr,phantom, near start,"\ucorner"description] & &
 |[alias=coda]|X^6\ar[rr,"A\otimes A\otimes A"] & &\ca{V} \\
 & X^3\times X^2\ar[ur,"1\diag1^3"]\ar[dr,"\pi_{13}1^2"'] & & \\
 & & |[alias=codb]|X^4\ar[uurr,bend right,"A\otimes A"'] &
 \arrow[Rightarrow,from=doma,to=coda,"1\otimes\comlt",shorten >=.1in,shorten <=.1in]
 \arrow[Rightarrow,from=coda,to=codb,"\mlt\otimes1",shorten >=.1in,shorten <=.1in]
 \end{tikzcd}
\end{displaymath}
which uses the isomorphism from \cref{X2Frobenius}. The existence of such a 2-isomorphism in $\Span|\ca{V}$ comes with an equality \cref{2cellsSpanV} of composite arrows in $\ca{V}$
\begin{gather*}
A_{xy}\otimes A_{yz}\xrightarrow{1\otimes\comlt_{ywz}}A_{xy}\otimes A_{yw}\otimes A_{wz}
\xrightarrow{\mlt_{xyw}\otimes1}A_{xw}\otimes A_{wz}
 \\
A_{xy}\otimes A_{yz}\xrightarrow{\mlt_{xyz}}A_{xz}\xrightarrow{\comlt_{xwz}}A_{xw}\otimes A_{wz}
\end{gather*}
which is precisely the one of the two conditions for a Frobenius $\Vv$-category~\cref{frob1}. The second condition can be checked similarly, and the axioms hold since those in $\Span$ hold by \cref{2cellsequality} as in earlier proofs.
\end{proof}

Hence we have a bijective correspondence between Frobenius monoids in $\Span|\Vv$ over $X^2$
and Frobenius $\ca{V}$-categories, for any monoidal category $\Vv$. Regarding Frobenius $\Vv$-functors of~\cref{Frobeniusfunctor},
by \cref{prop:Vfunctors} we can express them as morphisms between Frobenius pseudomonoids.

\begin{corollary}
An oplax pseudomonoid and oplax pseudocomonoid morphism of the form $^\id\alpha^{f\times f}$ in $\Span|\Vv$ is a Frobenius $\Vv$-functor. 
\end{corollary}

We can thus realize $\FrobCat{\Vv}$ as the subcategory
of $\sf{Frob}_{\opl,\opl}(\Span|\Vv)$ of Frobenius pseudomonoids in $\Span|\Vv$
over $X^2$ with the groupoid pseudo(co)monoid structures, and morphisms of the form $^\id\alpha^{f\times f}$.

For purposes of completeness, we conclude with the following result.
\begin{proposition}
A Frobenius (strict) monoid in $\Span|\Vv$ over $X^2$ with the trivial monoid and comonoid structure
is a $\Frob(\Vv)$-graph, whereas a oplax monoid and comonoid morphism $^{\id}\alpha^{f{\times}f}$ is a $\Frob(\Vv)$-graph morphism.
\end{proposition}

\begin{proof}
This can be deduced in a straightforward way from \cref{Vgraphsascomonoids} and its dual
by checking the Frobenius conditions. For morphisms, \cref{prop:ComonVgraphmaps} and its dual
suffice.
\end{proof}

Based on the above, the braided monoidal bicategory $\Span|\Vv$ serves indeed as the common framework for Hopf $\Vv$-categories and Frobenius $\Vv$-categories, which are respectively expressed as oplax Hopf monoids and Frobenius pseudomonoids therein. The fact that subsequent work \cite{paper1b} extends the Larson-Sweedler theorem between such many-object generalized structures may in particular indicate a correspondence between these relaxed notions of oplax bimonoids and their Frobenius counterparts in higher structures, yet to be investigated.

%% file: 8-appendix.tex
\section{}

\subsection{Pseudomonoid, oplax maps and 2-cells axioms}\label{sec:oplaxmaps}

A pseudomonoid $A$ with constraints $\alpha,\ell,r $ as in~\cref{alphalambdarho} satisfy the following two axioms
\begin{displaymath}
\begin{tikzcd}
A   A   A   A\ar[rr,"1   1   m"]\ar[d,"m  1  1"']\ar[dr,"1   m  1"description] &\ar[dr,phantom,"\stackrel{1  \alpha}{\cong}"] & A   A   A\ar[d,"1   m "] \\
A   A   A\ar[r,phantom,"\stackrel{\alpha  1}{\cong}"]\ar[d,"m  1"'] & A   A   A\ar[dr,phantom,"\stackrel{\alpha}{\cong}"]\ar[r,"1   m"description]\ar[dl,"m  1"description] & A   A\ar[d,"m"] \\
A   A\ar[rr,"m"'] && A
 \end{tikzcd}=
  \begin{tikzcd}
A   A   A   A\ar[dr,phantom,"\stackrel{c_{m,m}}{\cong}"]\ar[rr,"1  1   m"]\ar[d,"m  1  1"'] && A   A   A\ar[d,"1   m"]\ar[dl,"m  1"description] \\
A   A   A\ar[d,"m  1"']\ar[r,"1   m"description] & A   A\ar[r,phantom,"\stackrel{\alpha}{\cong}"]\ar[ld,phantom,"\stackrel{\alpha}{\cong}"]\ar[dr,"m"description] & A   A\ar[d,"m"] \\
A   A\ar[rr,"m"'] && A
 \end{tikzcd}
\end{displaymath}

\begin{equation}\label{eq:psmoneqs}
\adjustbox{scale=.85,center}{
\begin{tikzcd}[sep=.5in]
A   A\ar[dr,equal]\ar[d,shift left=6,phantom,"\stackrel{1r}{\cong}"]\ar[d,"1\uni1"'] & \\
 A   A   A\ar[r,"1\mlt"description]\ar[dr,phantom,"\stackrel{\alpha}{\cong}"]\ar[d,"\mlt1"'] & A   A\ar[d,"\mlt"] \\
 A   A\ar[r,"\mlt"] & A
\end{tikzcd}=
\begin{tikzcd}[sep=.5in]
 A   A\ar[d,"1\uni1"']\ar[dr,equal] & \\
 A   A   A\ar[d,"\mlt1"']\ar[r,phantom,"\stackrel{\ell1}{\cong}"] & A   A\ar[d,"\mlt"]\ar[dl,equal] \\
 A   A\ar[r,"\mlt"'] & A
\end{tikzcd}}
\end{equation}

A 1-cell $f\colon A\to B$ between pseudomonoids comes with 2-cells $(\phi,\phi_0)$ as in \cref{oplax2cells} that satisfy
\begin{equation}\label{eq:pseudomonmap}
\begin{tikzcd}[column sep=.3in]
& A     A\ar[r,"\mlt"]\ar[d,phantom,"\stackrel{\alpha}{\cong}"description] & A\ar[dr,"f"] & \\
A     A     A\ar[r,"\mlt    1"]\ar[ur,"1    \mlt"]\ar[dr,"f     f     f"'] &
A     A\ar[ur,"\mlt"]\ar[dr,"f     f"]\ar[rr,phantom,"\Downarrow{\scriptstyle\phi}"description]
\ar[d,phantom,"\Downarrow{\scriptstyle\phi    1_f}"description] && B \\
& B     B     B\ar[r,"\mlt    1"'] & B     B\ar[ur,"\mlt"']
\end{tikzcd} 
=
\begin{tikzcd}[column sep=.3in]
& A     A\ar[r,"\mlt"]\ar[dr,"f     f"']\ar[dd,phantom,"\Downarrow{\scriptstyle1_f    \phi}"description] &
A\ar[dr,"f"]\ar[d,phantom,"\Downarrow{\scriptstyle\phi}"description] & \\
A     A     A\ar[dr,"f     f     f"']\ar[ur,"1    \mlt"] & & B     B\ar[r,"\mlt"]
\ar[d,phantom,"\stackrel{\alpha}{\cong}"description] & B \\
& B     B     B\ar[ur,"B    \mlt"]\ar[r,"\mlt    1"'] & B     B\ar[ur,"\mlt"']
\end{tikzcd}
\end{equation}
\begin{displaymath}
\begin{tikzcd}[column sep=.25in]
& |[alias=doma]|A     A\ar[r,"\mlt"]\ar[dr,"f     f"']\ar[drr,phantom,"\Downarrow{\scriptstyle \phi}"description] & A\ar[dr,"f"] & \\
A\cong A     I\ar[ur,"1    \uni"]\ar[rr,"f    \uni"'{name=coda}]\ar[rru,bend left=60,"\id_A","\ell\cong"']
\ar[rrr,bend right=20,"f"',"\ell\cong"] &&
B     B\ar[r,"\mlt"'] & B
\arrow[Rightarrow, from=doma, to=coda,"1_f    \phi_0\;\;"',shorten <=.5em, shorten >=.5em]
\end{tikzcd}=
\begin{tikzcd}[column sep=.25in]
\hole \\
A\ar[rr,bend left,"f"]\ar[rr,bend right,"f"']\ar[rr,phantom,"\Downarrow{\scriptstyle1_f}"description] && B \\
\hole
\end{tikzcd}=
\begin{tikzcd}[column sep=.25in]
& |[alias=doma]|A     A\ar[r,"\mlt"]\ar[dr,"f     f"']\ar[drr,phantom,"\Downarrow{\scriptstyle \phi}"description] & A\ar[dr,"f"] & \\
A\cong I     A\ar[ur,"\uni    1"]\ar[rr,"\uni     f"'{name=coda}]\ar[rru,bend left=60,"\id_A","r \cong"']
\ar[rrr,bend right=20,"f"',"r \cong"] &&
B     B\ar[r,"\mlt"'] & B
\arrow[Rightarrow, from=doma, to=coda,"\phi_0    1_f\;\;"',shorten <=.5em, shorten >=.5em]
\end{tikzcd}
\end{displaymath}

A 2-cell $\beta\colon f\Rightarrow g$ between two oplax 1-cells of pseudomonoids $A$ and $B$ satisfies
\begin{equation}\label{monoidal2cell}
\begin{tikzcd}[row sep=.15in]
& A\ar[rd,bend left=20,"f"] & \\
A   A\ar[ur,bend left=20,"\mlt"]\ar[dr,bend left,"f   f"]\ar[dr,bend right,"g   g"']
\ar[rr,phantom,"\Downarrow{\scriptstyle\phi}"{description,near end}]
\ar[dr,phantom,"\Downarrow{\scriptstyle\alpha  \alpha}"description] && B \\
& B   B\ar[ur,bend right=20,"\mlt"'] &
\end{tikzcd}
\quad=\quad
\begin{tikzcd}[row sep=.1in]
& A\ar[rd,bend left,"f"]\ar[dr,bend right,"g"']\ar[dr,phantom,"\Downarrow{\scriptstyle\alpha}"description] & \\
A   A\ar[ur,bend left=20,"\mlt"]\ar[dr,bend right=20,"g   g"']\ar[rr,phantom,"{\scriptstyle\psi}\Downarrow"{description,near start}]
&& B \\
& B   B\ar[ur,bend right=20,"\mlt"'] &
\end{tikzcd}
\end{equation}
\begin{displaymath}
\begin{tikzcd}[row sep=.1in,column sep=.6in]
& A\ar[dr,bend left=20,"f"] & \\
I\ar[ur,bend left=20,"\uni"]\ar[rr,bend right,"\uni"']\ar[rr,phantom,"\Downarrow{\scriptstyle\phi_0}"description] && B
\end{tikzcd}
\quad=\quad
\begin{tikzcd}[row sep=.15in,column sep=.6in]
& A\ar[dr,bend left,"f"]\ar[dr,bend right,"g"']\ar[dr,phantom,"\Downarrow{\scriptstyle \alpha}"description] & \\
I\ar[ur,bend left=20,"\uni"]\ar[rr,bend right,"\uni"']\ar[rr,phantom,"{\scriptstyle\psi_0}\Downarrow"{description,near start}] && B.
\end{tikzcd}
\end{displaymath}

\subsection{Oplax bimonoid axioms}\label{oplaxbimonoidaxioms}

An oplax bimonoid $(M,\mlt,\uni,\lcomlt,\lcouni,\theta,\theta_0,\chi,\chi_0)$ in a braided monoidal bicategory $\Kk$ as in \cref{oplaxbimonoid}, namely an object in $\PsComon(\PsMon_\opl(\Kk))$, satisfies a number of axioms listed below.

Explicitly, $(M,\mlt,\uni)$ comes with invertible associativity and unit 2-cells $(\alpha,\ell,r)$ satisfying \cref{eq:psmoneqs} and $(M,\lcomlt,\lcouni)$ come with invertible coassociativity and unit 2-cells $(\beta,s,t)$ satisfying dual axioms. Moreover, \cref{ax1,ax2} express that $\lcomlt\colon M\to M\otimes M$ equipped with $\theta,\theta_0$ is an oplax morphism between pseudomonoids as in \cref{eq:pseudomonmap}; conditions~\cref{ax3,ax4} express that $\lcouni\colon M\to I$ equipped with $\chi,\chi_0$ is an oplax morphism between pseudomonoids; conditions \cref{ax5,ax6} express that the coassociativity isomorphism $\beta$ is a monoidal 2-cell namely satisfies~\cref{monoidal2cell}; and conditions~\cref{ax7,ax8,ax9,ax10} do the same for the left and right counit isomorphisms $s,t$.

All empty faces are filled appropriately (co)associativity and (co)unit isomorphisms, or by coherence isomorphisms coming from the braided monoidal structure on the bicategory (Gray monoid) $\Kk$.

\begin{equation}\label{ax1}
\resizebox{0.9\hsize}{!}{$
\cd[@C-1em]{
 & & M^3 \ar[r]^{\lcomlt 11} \ar[dd]_{1\mlt} \ar[ddll]_{\mlt 1} & M^4 \ar[r]^{11\lcomlt \lcomlt} \ar[dd]_{11\mlt} \ar@{}[ddr]|{11 {\theta} \Uparrow} & M^6 \ar[d]^{111\braid 1} \\
 & & & & M^6 \ar[d]^{11\mlt\mlt} \\
 M^2 \ar[ddrr]_{\mlt} & 
 & M^2 \ar[dd]_{\mlt} \ar[r]^{\lcomlt 1} & M^3 \ar[r]^{11\lcomlt} & M^4 \ar[d]^{1\braid 1} \\ 
 & & & {\theta} \Uparrow & M^4 \ar[d]^{\mlt\mlt} \\
 & & M \ar[rr]_{\lcomlt} & & M^2
}
\quad = \quad
\cd[@C-1em]{
 M^3 \ar[r]^{11\lcomlt} \ar[dd]_{\mlt 1} & M^4 \ar[r]^{\lcomlt \lcomlt 11} \ar[dd]_{\mlt 11} \ar@{}[ddr]|{{\theta} 11 \Uparrow} & M^6 \ar[d]|{1\braid 111} \ar[dr]^{111\braid 1} & & \\
 & & M^6 \ar[d]^{\mlt\mlt 11} & M^6 \ar[dr]^{11\mlt\mlt} & \\
 M^2 \ar[dd]_{\mlt} \ar[r]^{1\lcomlt} & M^3 \ar[r]^{\lcomlt 11} & M^4 \ar[d]^{1\braid 1} &  & M^4 \ar[dl]^{1\braid 1} \\ 
 & {\theta} \Uparrow & M^4 \ar[d]^{\mlt\mlt} & \ar[dl]^{\mlt\mlt} M^4 & \\
 M \ar[rr]_{\lcomlt} & & M^2 & & 
}$}
\end{equation}
\begin{equation}\label{ax2}
\cd[@C-1em]{
M \ar[r]^{\lcomlt} \ar[d]|{\uni 1} \ar@/_3em/[ddd]_{1} & M^2 \ar[r]^{1} \ar[d]_{\uni 11} \ar@{}[dr]|{{\theta_0} 11 \Uparrow} & M^2 \ar[d]^{ \uni\uni 1} \ar@/^3em/[ddd]^{1} \\
M^2 \ar@{}[ddrr]|{{\theta} \Uparrow} \ar[dd]^{\mlt} \ar[r]^{1\lcomlt} & M^3 \ar[r]^{\lcomlt 11} & M^4 \ar[d]_{1\braid 1} \\
& & M^4 \ar[d]_{\mlt\mlt} \\
M \ar[rr]_{\lcomlt} & & M^2 
}
\; = \; 
\id_{\lcomlt}
\; = \; 
\cd[@C-1em]{
M \ar[r]^{\lcomlt} \ar[d]|{1\uni} \ar@/_3em/[ddd]_{1} & M^2 \ar[r]^{1} \ar[d]_{11\uni} \ar@{}[dr]|{11{\theta_0} \Uparrow} & M^2 \ar[d]^{11\uni \uni} \ar@/^3em/[ddd]^{1} \\
M^2 \ar@{}[ddrr]|{{\theta} \Uparrow} \ar[dd]^{\mlt} \ar[r]^{\lcomlt 1} & M^3 \ar[r]^{11\lcomlt} & M^4 \ar[d]_{1\braid 1} \\
& & M^4 \ar[d]_{\mlt\mlt} \\
M \ar[rr]_{\lcomlt} & & M^2 
} 
\end{equation}
\begin{equation}\label{ax3}
\cd[]{
& M^3 \ar[d]_{\mlt 1} \ar[r]^{\lcouni \lcouni 1} \ar@/_1em/[dl]_{1\mlt} \ar@{}[dr]|{{\chi} 1 \Uparrow} & M \ar[d]^{1}\\
M^2 \ar@/_1em/[dr]_{\mlt}  & M^2 \ar[d]_{\mlt} \ar[r]^{\lcouni 1} \ar@{}[dr]|{{\chi} \Uparrow} & M \ar[d]^{\lcouni}\\
& M \ar[r]_{\lcouni} & I
}
\; = \;
\cd[]{
M^3 \ar[d]_{1\mlt} \ar[r]^{1\lcouni \lcouni} \ar@{}[dr]|{1 {\chi} \Uparrow} & M \ar[d]^{1}\\
M^2 \ar[d]_{\mlt} \ar[r]^{1\lcouni} \ar@{}[dr]|{{\chi} \Uparrow} & M \ar[d]^{\lcouni}\\
M \ar[r]_{\lcouni} & I
}
\end{equation}
\begin{equation}\label{ax4}
\cd[]{
 & & M \ar[dr]^{\lcouni} & \\
 M \ar[r]^{\uni 1} \ar@/^1em/[urr]^{1} \ar@/_1em/[drr]_{1} 
 & M^2 \ar[ur]^{\lcouni 1} \ar[dr]_{\mlt} \ar@{}[u]|{{\chi_0} 1 \Uparrow} \ar@{}[rr]|{{\chi} \Uparrow} & & I \\
 & & M \ar[ur]_{\lcouni} & 
}
\; = \;
\id_{\lcouni}
\; = \;
\cd[]{
 & & M \ar[dr]^{\lcouni} & \\
 M \ar[r]^{1\uni} \ar@/^1em/[urr]^{1} \ar@/_1em/[drr]_{1} 
 & M^2 \ar[ur]^{1\lcouni} \ar[dr]_{\mlt} \ar@{}[u]|{1{\chi_0} \Uparrow} \ar@{}[rr]|{{\chi} \Uparrow} & & I \\
 & & M \ar[ur]_{\lcouni} & 
}
\end{equation}

\begin{equation}\label{ax5}
\begin{tikzcd}[column sep=.5in, row sep=.3in]
M^2 \ar[ddd,"\mlt"'] \ar[r,"\lcomlt\lcomlt"{name=coda}] & M^4 \ar[d,"\mlt 11.1\braid 1"'] \ar[r,"\lcomlt 1\lcomlt 1"] & M^6 \ar[d,"1111\mlt .\braid.\braid\lcomlt"] \\
 & M^4 \ar[dd,"\mlt 1"'] \ar[r,"\lcomlt\lcomlt 1"{name=codb}] & M^5 \ar[d,"1\braid 11"] \\
 & & M^5 \ar[d,"\mlt\mlt 1"] \\
M \ar[dr,"\lcomlt"'] \ar[r,"\lcomlt"{name=doma}] & M^2 \ar[r,"\lcomlt 1"{name=domb}] & M^3 \\
 & M^2 \ar[ur,"1\lcomlt"'] & 
\arrow[Rightarrow,from=doma,to=coda,shorten >=4em, shorten <=4em, "\theta"]
\arrow[Rightarrow,from=domb,to=codb,shorten >=2em, shorten <=2em, "\theta 1"]
\end{tikzcd}
\quad = \quad
\begin{tikzcd}[column sep=.5in, row sep=.3in]
 & M^2 \ar[dr,"\lcomlt 1\lcomlt 1"] & \\
M^2 \ar[ddd,"\mlt"'] \ar[r,"\lcomlt\lcomlt"{name=coda}] \ar[ur,"\lcomlt\lcomlt"] & M^4 \ar[d,"\mlt 11.1\braid 1"'] \ar[r,"1\lcomlt 1\lcomlt"] & M^6 \ar[d,"1111\mlt .\braid.\braid\lcomlt"] \\
 & M^4 \ar[dd,"1\mlt"'] \ar[r,"1 \lcomlt\lcomlt"{name=codb}] & M^5 \ar[d,"11\braid 1"] \\
 & & M^5 \ar[d,"1\mlt\mlt"] \\
M \ar[r,"\lcomlt"'{name=doma}] & M^2 \ar[r,"1\lcomlt"'{name=domb}] & M^3 
\arrow[Rightarrow,from=doma,to=coda,shorten >=4em, shorten <=4em, "\theta"]
\arrow[Rightarrow,from=domb,to=codb,shorten >=2em, shorten <=2em, "1\theta"]
\end{tikzcd}
\end{equation}

\begin{equation}\label{ax6}
\begin{tikzcd}[column sep=.5in, row sep=.3in]
I \ar[dd,"\uni"'] \ar[r,"\sim"{name=coda}] & 
II \ar[d,"1\uni"] \ar[r,"\sim"] & III \ar[d,"11\uni"] \\
 & IM \ar[d,"\uni 1"] \ar[r,"\sim"{name=codb}] & IIM \ar[d,"\uni\uni 1"] \\
M \ar[dr,"\lcomlt"'] \ar[r,"\lcomlt"'{name=doma}] & M^2 \ar[r,"\lcomlt 1"'{name=domb}] & M^3 \\
 & M^2 \ar[ur,"1\lcomlt"'] & 
\arrow[Rightarrow,from=doma,to=coda,shorten >=2em, shorten <=2em, "\theta_0"]
\arrow[Rightarrow,from=domb,to=codb,shorten >=1em, shorten <=1em, "1\theta_0"']
\end{tikzcd}
\quad = \quad
\begin{tikzcd}[column sep=.5in, row sep=.3in]
 & II \ar[dr,"\sim"] & \\
I \ar[dd,"\uni"'] \ar[r,"\sim"{name=coda}] \ar[ur,"\sim"] & II \ar[d,"\uni 1"] \ar[r,"\sim"] & III \ar[d,"\uni 11"] \\
 & MI \ar[d,"1\uni"] \ar[r,"\sim"{name=codb}] & MII \ar[d,"1\uni\uni"] \\
M \ar[r,"\lcomlt"'{name=doma}] & M^2 \ar[r,"1\lcomlt"'{name=domb}] & M^3 
\arrow[Rightarrow,from=doma,to=coda,shorten >=2em, shorten <=2em, "\theta_0"]
\arrow[Rightarrow,from=domb,to=codb,shorten >=1em, shorten <=1em, "\theta_0 1"']
\end{tikzcd}
\end{equation}

\begin{equation}\label{ax7}
\begin{tikzcd}[column sep=.5in, row sep=.3in]
M^2 \ar[rr,"\sim", bend left]  \ar[dd,"\mlt"'] \ar[r,"\lcomlt\lcomlt"{name=coda}] & M^4 \ar[d,"\mlt 11.1\braid 1"] \ar[r,"1\lcouni 1\lcouni "] & MIMI \ar[d,"m11.1\braid 1"] \\
& M^3 \ar[d,"1\mlt"] \ar[r,"1\lcouni \lcouni "{name=codb}] & MII \ar[d,"\sim"] \\
M \ar[rr,"\sim"', bend right] \ar[r,"\lcomlt"{name=doma}] & M^2 \ar[r,"1\lcouni"{name=domb}] & MI
\arrow[Rightarrow,from=doma,to=coda,shorten >=2em, shorten <=2em, "\theta"]
\arrow[Rightarrow,from=domb,to=codb,shorten >=.6em, shorten <=.6em, "1 \chi"']
\end{tikzcd}
\quad = \quad 
\begin{tikzcd}[column sep=.5in, row sep=.3in]
M^2 \ar[r,"\sim"]  \ar[dd,"\mlt"'] & MIMI \ar[d,"m11.1\braid 1"] \\
& MII \ar[d,"\sim"] \\
M \ar[r,"\sim"'] & MI
\end{tikzcd}
\end{equation}

\begin{equation}\label{ax8}
\begin{tikzcd}[column sep=.5in, row sep=.3in]
I \ar[rr,"\sim", bend left]  \ar[dd,"\uni"'] \ar[r,"\sim"{name=coda}] & II \ar[d,"\uni 1"] \ar[r,"1"] & II \ar[d,"\uni 1"] \\
& MI \ar[d,"1\uni"] \ar[r,"1"{name=codb}] & MI \ar[d,"1"] \\
M \ar[rr,"\sim"', bend right] \ar[r,"\lcomlt"{name=doma}] & M^2 \ar[r,"1\lcouni "{name=domb}] & MI
\arrow[Rightarrow,from=doma,to=coda,shorten >=2em, shorten <=2em, "\theta_0"]
\arrow[Rightarrow,from=domb,to=codb,shorten >=.6em, shorten <=.6em, "1 \chi_0"']
\end{tikzcd}
\quad = \quad 
\begin{tikzcd}[column sep=.5in, row sep=.3in]
I \ar[r,"\sim"]  \ar[dd,"\uni"'] & II \ar[d,"\uni 1"] \\
& MI \ar[d,"\id"] \\
M \ar[r,"\sim"'] & MI
\end{tikzcd}
\end{equation}

\begin{equation}\label{ax9}
\begin{tikzcd}[column sep=.5in, row sep=.3in]
M^2 \ar[rr,"\sim", bend left]  \ar[dd,"\mlt"'] \ar[r,"\lcomlt\lcomlt"{name=coda}] & M^4 \ar[d,"11\mlt.1\braid 1"] \ar[r,"\lcouni 1\lcouni 1"] & IMIM \ar[d,"11m.1\braid 1"] \\
& M^3 \ar[d,"\mlt 1"] \ar[r,"\lcouni \lcouni 1"{name=codb}] & IMI \ar[d,"\sim"] \\
M \ar[rr,"\sim"', bend right] \ar[r,"\lcomlt"{name=doma}] & M^2 \ar[r,"\lcouni 1"{name=domb}] & IM
\arrow[Rightarrow,from=doma,to=coda,shorten >=2em, shorten <=2em, "\theta"]
\arrow[Rightarrow,from=domb,to=codb,shorten >=.6em, shorten <=.6em, "\chi 1"']
\end{tikzcd}
\quad = \quad 
\begin{tikzcd}[column sep=.5in, row sep=.3in]
M^2 \ar[r,"\sim"]  \ar[dd,"\mlt"'] & IMIM \ar[d,"11m.1\braid 1"] \\
& MII \ar[d,"\sim"] \\
M \ar[r,"\sim"'] & MI
\end{tikzcd}
\end{equation}

\begin{equation}\label{ax10}
\begin{tikzcd}[column sep=.5in, row sep=.3in]
I \ar[rr,"\sim", bend left]  \ar[dd,"\uni"'] \ar[r,"\sim"{name=coda}] & II \ar[d,"1\uni"] \ar[r,"1"] & II \ar[d,"1\uni"] \\
& IM \ar[d,"\uni 1"] \ar[r,"1"{name=codb}] & IM \ar[d,"1"] \\
M \ar[rr,"\sim"', bend right] \ar[r,"\lcomlt"{name=doma}] & M^2 \ar[r,"\lcouni 1"{name=domb}] & IM
\Twocell{doma}{coda}{2em}{"\theta_0"}
\arrow[Rightarrow,from=domb,to=codb,shorten >=.6em, shorten <=.6em, "\chi_0 1"']
\end{tikzcd}
\quad = \quad 
\begin{tikzcd}[column sep=.5in, row sep=.3in]
I \ar[r,"\sim"]  \ar[dd,"\uni"'] & II \ar[d,"1\uni"] \\
& IM \ar[d,"1"] \\
M \ar[r,"\sim"'] & IM
\end{tikzcd}
\end{equation}

\subsection{Oplax bimonoid morphisms axioms}\label{oplaxbimonoidmapaxioms}

Below we give the axioms for \cref{def:bimonoidmap} where
an oplax bimonoid morphism between is a 1-cell $(f,\phi,\phi_0,\psi,\psi_0) \colon M \to N$ in $\PsComon_\opl(\PsMon_\opl(\ca{K}))$, namely satisfying the oplax pseudomonoid and pseudocomonoid morphism axioms \cref{eq:pseudomonmap} and moreover $\psi,\psi_0$ are monoidal 2-cells \cref{monoidal2cell}, i.e.

\begin{equation}\label{eq:bimonmapax1}
\begin{tikzcd}[column sep=.3in]
& M\ar[rr,"\lcomlt"]\ar[dr,phantom,"\Downarrow{\scriptstyle\theta}"description] && M  M\ar[dr,bend left=15,"f  f"] & \\
M    M\ar[r,"\lcomlt   \lcomlt"]\ar[ur,bend left=15,"\mlt"]\ar[dr,bend right=15,"f    f"'] &
M^{ 4}\ar[r,"1\braid1"]\ar[dr,"f^{ 4}"description]\ar[d,phantom,"\Downarrow{\scriptstyle\psi \psi}"description] &
M^{   4}\ar[ur,"\mlt \mlt"description]\ar[dr,"f^{   4}"description]
\ar[rr,phantom,"\Downarrow{\scriptstyle\phi \phi}"description]
&& N  N \\
& N    N\ar[r,"\lcomlt   \lcomlt"'] & N^{   4}\ar[r,"1\braid1"'] & N^{ 4}\ar[ur,bend right=15,"\mlt   \mlt"'] &
\end{tikzcd}=
\begin{tikzcd}[column sep=.2in]
& M\ar[r,"\lcomlt"]\ar[dr,"f"description]\ar[dd,phantom,"\Downarrow{\scriptstyle\phi}"description] &
M  M\ar[drr,bend left,"f  f"]\ar[d,phantom,"\Downarrow{\scriptstyle\psi}"description] && \\
M    M\ar[dr,bend right=15,"f    f"']\ar[ur,bend left=15,"\mlt"] & & N\ar[rr,"\lcomlt"]
\ar[dr,phantom,"\Downarrow{\scriptstyle\xi}"description] && N \\
& N    N\ar[ur,"\mlt"description]\ar[r,"\lcomlt   \lcomlt"'] & N^{   4}\ar[r,"1\braid1"'] &
N^{ 4}\ar[ur,bend right=15,"\mlt \mlt"']
\end{tikzcd}
\end{equation}
\begin{equation}
\begin{tikzcd}
& M\ar[r,"\lcomlt"]\ar[d,phantom,"\Downarrow{\scriptstyle\theta_0}"description] &
M  M\ar[dr,bend left=20,"f  f"] & \\
I\ar[ur,bend left=15,"\uni"]\ar[urr,bend right=15,"\uni \uni"']\ar[rrr,bend right=20,"\uni \uni"'] &
\phantom{M} & \Downarrow{\scriptstyle\phi_0 \phi_0}& N  N
\end{tikzcd}=
\begin{tikzcd}
& M\ar[r,"\lcomlt"]\ar[dr,"f"']\ar[d,phantom,"\Downarrow{\scriptstyle\phi_0}"description] &
M  M\ar[dr,bend left=20,"f  f"]\ar[d,phantom,"\Downarrow{\scriptstyle\psi}"description] & \\
I\ar[ur,bend left=15,"\uni"]\ar[rr,"\uni"']\ar[rrr,bend right=20,"\uni \uni"',"\Downarrow{\scriptstyle\xi_0}"] &
\phantom{M} & N\ar[r,"\lcomlt"'] & N  N
\end{tikzcd}
\end{equation}
\begin{equation}
\begin{tikzcd}[row sep=.2in]
&& M\ar[dr,bend left=20,"\lcouni"]\ar[dd,phantom,"\Downarrow{\scriptstyle\chi}"description] & \\
M\ar[urr,bend left=25,"\mlt"]\ar[rrd,bend left=20,"\lcouni \lcouni"]\ar[dr,bend right=15,"f  f"']
\ar[drr,phantom,"\Downarrow{\scriptstyle\psi_0 \psi_0}"description] &&& I \\
& N  N\ar[r,"\lcouni \lcouni"'] & I    I\ar[ur,bend right=20,"\sim"'] &
\end{tikzcd}=
\begin{tikzcd}[row sep=.2in]
& M\ar[rrd,bend left=20,"\lcouni","\Downarrow{\scriptstyle\psi_0}"']\ar[dr,"f"']\ar[dd,phantom,"\Downarrow{\scriptstyle\phi}"description] && \\
M  M\ar[ur,bend left=20,"\mlt"]\ar[dr,bend right=20,"f  f"'] && N\ar[r,"\lcouni"]\ar[d,phantom,"\Downarrow{\scriptstyle\omega}"description]
& I \\
& N  N\ar[r,"\lcouni \lcouni"']\ar[ur,"\mlt"] & I  I\ar[ur,bend right=20,"\sim"'] &
\end{tikzcd}
\end{equation}
\begin{equation}
\begin{tikzcd}[row sep=.15in]
& M\ar[dr,bend left=20,"\lcouni"] & \\
I\ar[ur,bend left=20,"\uni"]\ar[rr,bend right=25,"\sim"']\ar[rr,phantom,"\Downarrow{\scriptstyle\chi_0}"] && I
\end{tikzcd}=
\begin{tikzcd}[row sep=.2in]
& M\ar[drr,bend left=20,"\lcouni","\Downarrow{\scriptstyle\psi_0}"']\ar[dr,"f"description] && \\
I\ar[ur,bend left=20,"\uni"]\ar[rr,"\uni"',"\Downarrow{\scriptstyle\phi_0}"]
\ar[rrr,bend right=20,"\sim"',"\Downarrow{\scriptstyle\omega_0}"] && N\ar[r,"\lcouni"'] & I
\end{tikzcd}
\end{equation}

More precisely, for example recall that the 2-cell $\psi\colon (f\ot f)\circ \lcomlt\Rightarrow\lcomlt\circ f$ from \cref{oplaxbimonoidmap} has as codomain the pseudomonoid morphism whose structure maps arise as composites from the ones for $\lcomlt$ and $f$, for example
\begin{displaymath}
 \begin{tikzcd}
  M    M\ar[rr,"\mlt"]\ar[d,"f  f"'] \ar[drr, phantom,"\scriptstyle\Downarrow\phi"] && M\ar[d,"f"] \\
  N  N\ar[rr,"\mlt"]\ar[d,"\lcomlt \lcomlt"']\ar[drr,phantom,"\scriptstyle\Downarrow\xi"] && N\ar[d,"\lcomlt"] \\
  N  N  N  N\ar[r,"1 \braid 1"'] & N  N  N  N\ar[r,"\mlt \mlt"'] & N  N
 \end{tikzcd}
\end{displaymath}
and similarly for its domain, where the monoidal product of pseudomonoid maps using the braiding is needed, giving
\begin{displaymath}
\begin{tikzcd}
 M  M\ar[drr,phantom,"\scriptstyle\Downarrow\theta"]\ar[rr,"\mlt"] \ar[d,"\lcomlt \lcomlt"'] && M\ar[d,"\lcomlt"] \\
 M  M  M  M\ar[r,"1 \braid 1"]\ar[d,"f  f  f  f"'] & M  M   M   M\ar[d,"f^{  4}"] \ar[r,"\mlt  \mlt"]\ar[dr,phantom,"\scriptstyle\Downarrow\phi \phi"] & M  M\ar[d,"f  f"] \\
 N  N  N  N\ar[r,"1 \braid 1"'] & N   N   N   N\ar[r,"\mlt \mlt"'] & N  N
\end{tikzcd}
 \end{displaymath}
which are both used in axiom \cref{eq:bimonmapax1}.

\subsection{Oplax module, morphism and transformation axioms}\label{oplaxmoduleaxioms}

We detail the conditions satisfied by the structure maps of a right oplax module $(X,\rho,\xi,\xi_0)$ for a pseudomonoid $(M,\mlt,\uni,\alpha,r ,\ell)$ as defined in \cref{def:oplaxmod}.

\begin{equation}\label{oplaxmod1}
\begin{tikzcd}
XMMM\ar[rr,"11\mlt"]\ar[drr,phantom,"\stackrel{1\alpha}{\cong}"]\ar[d,"\rho11"']\ar[dr,"1\mlt1"description] && XMM\ar[d,"1\mlt"] \\
XMM\ar[r,phantom,"\Downarrow{\scriptstyle\xi1}"description]\ar[d,"\rho1"'] & XMM\ar[r,"1\mlt"]\ar[dl,"\rho1"]\ar[dr,phantom,"\Downarrow{\scriptstyle\xi}"description]
& XM\ar[d,"\rho"] \\
XM\ar[rr,"\rho"'] && X
\end{tikzcd}\;=\;
\begin{tikzcd}
XMMM\ar[dr,phantom,"\stackrel{c_{\rho,\mlt}}{\cong}"]\ar[rr,"11\mlt"]\ar[d,"\rho11"'] && XMM\ar[dl,"\rho1"']\ar[d,"1\mlt"] \\
XMM\ar[r,"1\mlt"]\ar[drr,phantom,"\Downarrow{\scriptstyle\xi}"description]
\ar[d,"\rho1"'] & XM\ar[dr,"\rho"']\ar[r,phantom,"\Downarrow{\scriptstyle\xi}"description]
& XM\ar[d,"\rho"] \\
XM\ar[rr,"\rho"'] && X
\end{tikzcd}
\end{equation}

\begin{equation}\label{oplaxmod2}
\begin{tikzcd}
XM\ar[r,"1\uni1"]\ar[rr,bend left=35,"1"description]\ar[rr,bend left=20,phantom,"\stackrel{1r}{\cong}"]\ar[dr,phantom,"\Downarrow{\scriptstyle\xi_01}"description]\ar[dr,bend right,"1"'] & XMM\ar[d,"\rho1"'] \ar[r,"1\mlt"]\ar[dr,phantom,"\Downarrow{\scriptstyle\xi}"description] 
& XM\ar[d,"\rho"] \\
& XM\ar[r,"\rho"'] & X
\end{tikzcd}\;=\;\id_\rho
\end{equation}


An oplax module morphism $f\colon X\to Y$ comes with a 2-cell $\phi\colon f\circ\rho\Rightarrow\rho\circ f\ot 1$ as in \cref{eq:oplaxmorphism} satisfying

\begin{equation}\label{oplaxmodmorph1}
\begin{tikzcd}
 XMM\ar[rr,"1\mlt"]\ar[d,"f11"']\ar[dr,"\rho1"]\ar[drr,phantom,"\Downarrow{\scriptstyle\xi}"] &  & XM\ar[d,"\rho"] \\
 YMM\ar[d,"\rho1"']\ar[r,phantom,"\Downarrow{\scriptstyle\phi1}"description] &XM\ar[r,"\rho"]\ar[dl,"f1"]
 \ar[dr,phantom,"\Downarrow{\scriptstyle\phi}"description] & X\ar[d,"f"] \\
 YM\ar[rr,"\rho"'] && Y
\end{tikzcd}\;=\;
\begin{tikzcd}
 XMM\ar[dr,phantom,"\stackrel{c_{f,\mlt}}{\cong}"]\ar[rr,"1\mlt"]\ar[d,"f11"'] && XM\ar[d,"\rho"]\ar[dl,"f1"'] \\
 YMM\ar[r,"1\mlt"]\ar[d,"\rho1"']\ar[drr,phantom,near start,"\Downarrow{\scriptstyle\xi}"description] & YM\ar[dr,"\rho"']\ar[r,phantom,"\Downarrow{\scriptstyle\phi}"description] & X\ar[d,"f"] \\
 YM\ar[rr,"\rho"'] && Y
\end{tikzcd}
\end{equation}

\begin{equation}\label{oplaxmodmorph2}
\begin{tikzcd}
X\ar[dr,phantom,bend left=10,"\stackrel{c_{f,\uni}}{\cong}"]\ar[rr,"1\uni"]\ar[d,"f"'] & & XM\ar[d,"\rho"]\ar[dl,"f1"'] \\
Y\ar[r,"1\uni"]\ar[drr,bend right,"1"description]\ar[drr,phantom,"\Downarrow{\scriptstyle\xi_0}"] & YM\ar[dr,"\rho"]\ar[r,phantom,"\Downarrow{\scriptstyle\phi}"description] & X\ar[d,"f"] \\
&& Y
\end{tikzcd}\;=\;
\begin{tikzcd}
 X\ar[rr,"1\uni"]\ar[drr,bend right=15,"1"']\ar[ddrr,bend right=20,"f"description]\ar[drr,phantom,bend left=5,"\Downarrow{\scriptstyle\xi_0}"] && XM\ar[d,"\rho"] \\
 && X\ar[d,"f"] \\
 && Y
\end{tikzcd}
\end{equation}

An oplax module transformation $\alpha\colon(f,\phi)\Rightarrow (g,\psi)$ satisfies 
\begin{equation}\label{oplaxmodtransax}
\begin{tikzcd}[sep=.6in]
XM\ar[r,"\rho"]\ar[d,bend right,"g1"']\ar[d,bend left,"f1"]\ar[d,phantom,"\stackrel{\alpha1}{\Leftarrow}"description] & |[alias=doma]|X\ar[d,"f"] \\
|[alias=coda]|YM\ar[r,"\rho"'] & Y
\arrow[Rightarrow,from=doma,to=coda,"\phi"',shorten >=.35in,shorten <=.35in]
\end{tikzcd}\;=\;
\begin{tikzcd}[sep=.6in]
XM\ar[d,"g1"']\ar[r,"\rho"] & |[alias=doma]|X\ar[d,bend left,"f"]\ar[d,bend right,"g"']\ar[d,phantom,"\stackrel{\alpha}{\Leftarrow}"description] \\
|[alias=coda]|YM\ar[r,"\rho"'] & Y
\arrow[Rightarrow,from=doma,to=coda,"\psi"',shorten >=.35in,shorten <=.35in]
\end{tikzcd}
\end{equation}

\subsection{Frobenius pseudomonoid axioms}\label{sec:Frobpseudo}

In \cite[Prop.~25]{LaudaFrobAlgs}, a definition for a pseudomonoid $A$ in a Gray monoid to be Frobenius is given, in terms of the existence of a counit and a map $I\to A\ot A$ along with two 2-isomorphisms satisfying axioms. We will here use an alternative definition that naturally generalizes the monoid-comonoid definition in the ordinary monoidal setting, using a subset of axioms of \cite[Lem.~32]{LaudaFrobAlgs}.

An object which is a pseudomonoid $(A,\mlt,\uni,\alpha,\ell,r )$ and a pseudocomonoid $(A,\comlt,\couni,\beta,s,t)$ is Frobenius when it comes equipped with 2-isomorphisms $\psi,\phi$ as in \cref{pseudoFrobcond}
subject to axioms that express that $\comlt\colon A\to A\ot A$ is a left $A$, right $A$ pseudomodule morphism and also $\mlt\colon A\ot A\to A$ is a left $A$, right $A$ pseudocomodule morphism. Explicitly, we have the following eight axioms where 
axioms \cref{eq:rightpsmod12} express that $\comlt$ is a right $A$-pseudomodule pseudomap between $(A,\mlt\colon A\ot A\to A)$ and $(A\ot A,1\ot\mlt\colon A\ot A\ot A\to A\ot A$) similarly to \cref{oplaxmodmorph1,oplaxmodmorph2},
axioms \cref{eq:leftpsmod12} express that $\comlt$ is a left $A$-pseudomodule pseudomap, axioms \cref{eq:leftpscomod12} express that $\mlt$ is a left $A$-pseudocomodule map between $(A,\comlt\colon A\to A\ot A)$ and $(A\ot A,\comlt\ot1\colon A\ot A\to A\ot A\ot A)$, and axioms \cref{eq:rightpscomod12} express that $\mlt$ is a right $A$-pseudocomodule map.
\begin{equation}\label{eq:rightpsmod12}
\begin{tikzcd}[column sep=.2in,row sep=.4in]
 AAA\ar[rr,"1\mlt"]\ar[d,"\comlt11"']\ar[dr,"\mlt1"description] && AA\ar[d,"\mlt"] \\
 AAAA\ar[d,"1\mlt1"'] & AA\ar[r,"\mlt"]\ar[l,phantom,"\scriptstyle\stackrel{\phi1}{\cong}"]\ar[dr,phantom,"\scriptstyle\stackrel{\phi}{\cong}"]\ar[ur,phantom,"\scriptstyle\stackrel{\alpha}{\cong}"]\ar[dl,"\comlt1"description] & A\ar[d,"\comlt"] \\
AAA\ar[rr,"1\mlt"'] && AA 
\end{tikzcd}{=}
\begin{tikzcd}[column sep=.2in,row sep=.4in]
 AAA\ar[rr,"1\mlt"]\ar[d,"\comlt11"']\ar[dr,phantom,"\scriptstyle\stackrel{c_{\comlt,m}}{\cong}"] && AA\ar[d,"\mlt"]\ar[dl,"\comlt1"description] \\
 AAAA\ar[d,"1\mlt1"']\ar[r,"11\mlt"] & AAA\ar[dr,"1\mlt"description]\ar[ld,phantom,"\scriptstyle\stackrel{1\alpha}{\cong}"]\ar[r,phantom,"\scriptstyle\stackrel{\phi}{\cong}"] & A\ar[d,"\comlt"] \\
AAA\ar[rr,"1\mlt"'] && AA 
\end{tikzcd}
\quad
\begin{tikzcd}[column sep=.2in,row sep=.4in]
A\ar[rr,"1\uni"]\ar[d,"\comlt"']\ar[dr,phantom,"\scriptstyle\stackrel{c_{\comlt,\uni}}{\cong}"] && AA\ar[d,"\mlt"]\ar[dl,"\comlt1"description]  \\
AA\ar[r,"11\uni"]\ar[drr,bend right,"\id"'] & AAA\ar[dr,"1\mlt"description]\ar[r,phantom,"\scriptstyle\stackrel{\phi}{\cong}"] & A\ar[d,"\comlt"] \\ 
\phantom{A}\ar[ur,phantom,near end,"\scriptstyle\stackrel{1r }{\cong}"] && AA
\end{tikzcd}{=}
\begin{tikzcd}[column sep=.2in,row sep=.4in]
A\ar[r,"1\uni"]\ar[dr,phantom,"\scriptstyle\stackrel{r }{\cong}"]\ar[d,"\comlt"']\ar[dr,bend right,"\id"description] & AA\ar[d,"\mlt"] \\
AA\ar[dr,"\id"'] & A\ar[d,"\comlt"] \\
& AA
\end{tikzcd}
\end{equation}

\begin{equation}\label{eq:leftpsmod12}
\begin{tikzcd}[column sep=.2in,row sep=.4in]
 AAA\ar[rr,"\mlt1"]\ar[d,"11\comlt"']\ar[dr,"1\mlt"description] && AA\ar[d,"\mlt"] \\
 AAAA\ar[d,"1\mlt1"'] & AA\ar[r,"\mlt"]\ar[l,phantom,"\scriptstyle\stackrel{1\psi}{\cong}"]\ar[dr,phantom,"\scriptstyle\stackrel{\psi}{\cong}"]\ar[ur,phantom,"\scriptstyle\stackrel{\alpha}{\cong}"]\ar[dl,"1\comlt"description] & A\ar[d,"\comlt"] \\
AAA\ar[rr,"\mlt1"'] && AA 
\end{tikzcd}{=}
\begin{tikzcd}[column sep=.2in,row sep=.4in]
 AAA\ar[rr,"\mlt1"]\ar[d,"11\comlt"']\ar[dr,phantom,"\scriptstyle\stackrel{c_{m,\comlt}}{\cong}"] && AA\ar[d,"\mlt"]\ar[dl,"1\comlt"description] \\
 AAAA\ar[d,"1\mlt1"']\ar[r,"\mlt11"] & AAA\ar[dr,"\mlt1"description]\ar[ld,phantom,"\scriptstyle\stackrel{\alpha1}{\cong}"]\ar[r,phantom,"\scriptstyle\stackrel{\psi}{\cong}"] & A\ar[d,"\comlt"] \\
AAA\ar[rr,"\mlt1"'] && AA 
\end{tikzcd}\quad
\begin{tikzcd}[column sep=.2in,row sep=.4in]
A\ar[rr,"\uni1"]\ar[d,"\comlt"']\ar[dr,phantom,"\scriptstyle\stackrel{c_{\uni,\comlt}}{\cong}"] && AA\ar[d,"\mlt"]\ar[dl,"1\comlt"description]  \\
AA\ar[r,"\uni11"]\ar[drr,bend right,"\id"'] & AAA\ar[dr,"\mlt1"description]\ar[r,phantom,"\scriptstyle\stackrel{\psi}{\cong}"] & A\ar[d,"\comlt"] \\ 
\phantom{A}\ar[ur,phantom,near end,"\scriptstyle\stackrel{\ell1}{\cong}"] && AA
\end{tikzcd}{=}
\begin{tikzcd}[column sep=.2in,row sep=.4in]
A\ar[r,"\uni1"]\ar[dr,phantom,"\scriptstyle\stackrel{\ell}{\cong}"]\ar[d,"\comlt"']\ar[dr,bend right,"\id"description] & AA\ar[d,"\mlt"] \\
AA\ar[dr,"\id"'] & A\ar[d,"\comlt"] \\
& AA
\end{tikzcd}
\end{equation}

\begin{equation}\label{eq:leftpscomod12}
\begin{tikzcd}[column sep=.2in,row sep=.4in]
 AA\ar[rr,"\mlt"]\ar[d,"\comlt1"']\ar[dr,phantom,"\scriptstyle\stackrel{\phi}{\cong}"] && A\ar[d,"\comlt"]\ar[dl,"\comlt"description] \\
 AAA\ar[d,"1\comlt1"']\ar[r,"1\mlt"] & AA\ar[dr,"1\comlt"description]\ar[ld,phantom,"\scriptstyle\stackrel{1\phi}{\cong}"]\ar[r,phantom,"\scriptstyle\stackrel{\beta}{\cong}"] & AA\ar[d,"\comlt1"] \\
AAAA\ar[rr,"11\mlt"'] && AAA 
\end{tikzcd}{=}
\begin{tikzcd}[column sep=.2in,row sep=.4in]
 AA\ar[rr,"\mlt"]\ar[d,"\comlt1"']\ar[dr,"\comlt1"description] && AA\ar[d,"\comlt"] \\
 AAA\ar[d,"1\comlt1"'] & AAA\ar[r,"1\mlt"]\ar[l,phantom,"\scriptstyle\stackrel{\beta1}{\cong}"]\ar[dr,phantom,"\scriptstyle\stackrel{c_{\comlt,\mlt}}{\cong}"]\ar[ur,phantom,"\scriptstyle\stackrel{\phi}{\cong}"]\ar[dl,"\comlt11"description] & AA\ar[d,"\comlt1"] \\
AAAA\ar[rr,"11\mlt"'] && AAA 
\end{tikzcd}
\quad
\begin{tikzcd}[column sep=.2in,row sep=.4in]
AA\ar[rr,"\comlt1"]\ar[d,"\mlt"']\ar[dr,phantom,"\scriptstyle\stackrel{\phi}{\cong}"] && AAA\ar[d,"\couni11"]\ar[dl,"1\mlt"description]  \\
A\ar[r,"\comlt"]\ar[drr,bend right,"\id"'] & AA\ar[dr,"\couni1"description]\ar[r,phantom,"\scriptstyle\stackrel{c_{\couni,\mlt}}{\cong}"] & AA\ar[d,"\mlt"] \\ 
\phantom{A}\ar[ur,phantom,near end,"\scriptstyle\stackrel{t}{\cong}"] && A
\end{tikzcd}{=}
\begin{tikzcd}[column sep=.2in,row sep=.4in]
AA\ar[r,"\comlt1"]\ar[dr,phantom,"\scriptstyle\stackrel{t1}{\cong}"]\ar[d,"\mlt"']\ar[dr,bend right,"\id"description] & AAA\ar[d,"\couni11"] \\
A\ar[dr,"\id"'] & AA\ar[d,"\mlt"] \\
& A
\end{tikzcd}
\end{equation}

\begin{equation}\label{eq:rightpscomod12}
\begin{tikzcd}[column sep=.2in,row sep=.4in]
 AA\ar[rr,"\mlt"]\ar[d,"1\comlt"']\ar[dr,phantom,"\scriptstyle\stackrel{\psi}{\cong}"] && A\ar[d,"\comlt"]\ar[dl,"\comlt"description] \\
 AAA\ar[d,"1\comlt1"']\ar[r,"\mlt1"] & AA\ar[dr,"\comlt1"description]\ar[ld,phantom,"\scriptstyle\stackrel{\psi1}{\cong}"]\ar[r,phantom,"\scriptstyle\stackrel{\beta}{\cong}"] & AA\ar[d,"1\comlt"] \\
AAAA\ar[rr,"\mlt11"'] && AAA 
\end{tikzcd}{=}
\begin{tikzcd}[column sep=.2in,row sep=.4in]
 AA\ar[rr,"\mlt"]\ar[d,"1\comlt"']\ar[dr,"1\comlt"description] && AA\ar[d,"\comlt"] \\
 AAA\ar[d,"1\comlt1"'] & AAA\ar[r,"\mlt1"]\ar[l,phantom,"\scriptstyle\stackrel{1\beta}{\cong}"]\ar[dr,phantom,"\scriptstyle\stackrel{c_{\mlt,\comlt}}{\cong}"]\ar[ur,phantom,"\scriptstyle\stackrel{\psi}{\cong}"]\ar[dl,"11\comlt"description] & AA\ar[d,"1\comlt"] \\
AAAA\ar[rr,"\mlt11"'] && AAA 
\end{tikzcd}
\quad
\begin{tikzcd}[column sep=.2in,row sep=.4in]
AA\ar[rr,"1\comlt"]\ar[d,"\mlt"']\ar[dr,phantom,"\scriptstyle\stackrel{\psi}{\cong}"] && AAA\ar[d,"11\couni"]\ar[dl,"\mlt1"description]  \\
A\ar[r,"\comlt"]\ar[drr,bend right,"\id"'] & AA\ar[dr,"1\couni"description]\ar[r,phantom,"\scriptstyle\stackrel{c_{\mlt,\couni}}{\cong}"] & AA\ar[d,"\mlt"] \\ 
\phantom{A}\ar[ur,phantom,near end,"\scriptstyle\stackrel{s}{\cong}"] && A
\end{tikzcd}{=}
\begin{tikzcd}[column sep=.2in,row sep=.4in]
AA\ar[r,"1\comlt"]\ar[dr,phantom,"\scriptstyle\stackrel{1s}{\cong}"]\ar[d,"\mlt"']\ar[dr,bend right,"\id"description] & AAA\ar[d,"11\couni"] \\
A\ar[dr,"\id"'] & AA\ar[d,"\mlt"] \\
& A
\end{tikzcd}
\end{equation}